\documentclass[a4paper,10pt]{article}
\usepackage[utf8]{inputenc}
\usepackage{hyperref}
\usepackage{verbatim} %to allow long "comment-out-'s"

\usepackage{graphicx}
\usepackage{caption}
\usepackage{amssymb}
\usepackage{ textcomp } %Promille sign
\usepackage{color}
\usepackage{enumerate}
\usepackage{mathrsfs}
\usepackage{multicol}
\usepackage{hyperref}
\usepackage{amsopn}

\usepackage{amsmath}
\usepackage{amsthm}
\usepackage{amssymb}

\usepackage{multicol}

\usepackage{algorithm}
\usepackage{algpseudocode}

\newcommand{\N}{\mathbb{N}}

\newcommand{\R}{\mathbb{R}}

\newcommand{\calC}{\mathcal{C}}
\newcommand{\calD}{\mathcal{D}}

\newcommand{\calH}{\mathcal{H}}

\newcommand{\calM}{\mathcal{M}}
\newcommand{\calP}{\mathcal{P}}

\newcommand{\eqdef}{\ensuremath{\stackrel{\mbox{\upshape\tiny def.}}{=}}}

\newcommand{\abs}[1]{\left\vert #1 \right\vert}
\newcommand{\norm}[1]{\Vert #1 \Vert}

\newcommand{\set}[1]{\left\lbrace #1\right\rbrace}
\newcommand{\sse}{\subseteq}
\newcommand{\sprod}[1]{\left\langle #1 \right\rangle}

\usepackage{dsfont}

\newcommand{\matleq}{\preccurlyeq}
\newcommand{\matgeq}{\succcurlyeq}

\newcommand{\leqsim}{\lesssim}

\DeclareMathOperator{\dist}{dist}

\DeclareMathOperator{\sgn}{sign}
\DeclareMathOperator{\supp}{supp}
\DeclareMathOperator{\id}{id}
\DeclareMathOperator{\Id}{Id}
\newcommand{\argmin}{\mathop{\mathrm{argmin}}}
\newcommand{\argmax}{\mathop{\mathrm{argmax}}}

\DeclareMathOperator{\dom}{dom}
\DeclareMathOperator{\hess}{hess}

\newtheorem{lem}{Lemma}
\newtheorem{prop}[lem]{Proposition}

\newtheorem{theo}[lem]{Theorem}
\newtheorem{cor}[lem]{Corollary}
\newtheorem{assumption}{Assumption}

\newtheorem{defi}{Definition}
\newtheorem{rem}{Remark}

\theoremstyle{definition}

\numberwithin{lem}{section}

\usepackage{etoolbox}
\let\bbordermatrix\bordermatrix
\patchcmd{\bbordermatrix}{8.75}{4.75}{}{}
\patchcmd{\bbordermatrix}{\left(}{\left[}{}{}
\patchcmd{\bbordermatrix}{\right)}{\right]}{}{}

\usepackage{upgreek}
\usepackage{subfig}

\usepackage[affil-it]{authblk}

\author[1]{Axel Flinth}
\author[1,2]{Fr\'ed\'eric de Gournay}
\author[1,2]{Pierre Weiss}

\affil[1]{IMT, Universit\'{e} de Toulouse, CNRS}
\affil[2]{ITAV, Universit\'{e} de Toulouse, CNRS}

\title{On the linear convergence rates of exchange and continuous methods for total variation minimization}

\begin{document}

\maketitle
% {\small
% \tableofcontents
% }
% \newpage

\begin{abstract}
We analyze an exchange algorithm for the numerical solution total-variation regularized inverse problems over the space $\calM(\Omega)$ of Radon measures on a subset $\Omega$ of $\R^d$. Our main result states that under some regularity conditions, the method eventually converges linearly. Additionally, we prove that continuously optimizing the amplitudes of positions of the target measure will succeed at a linear rate with a good initialization. Finally, we propose to combine the two approaches into an alternating method and discuss the comparative advantages of this approach. 


\end{abstract}

\subsubsection*{Acknowledgement}
The authors acknowledge support from ANR
JCJC OMS.

\section{Introduction}

\subsection{The problem}

The main objective of this paper is to develop and analyze iterative algorithms to solve the following infinite dimensional problem:
\begin{equation}\label{eq:primal}
 \tag{$\calP(\Omega)$} \inf_{\mu \in \calM(\Omega)} J(\mu)\eqdef \norm{\mu}_{\calM} + f(A\mu),
\end{equation}
where $\Omega$ is a bounded open domain of $\R^d$, $\calM(\Omega)$ is the set of Radon measures on $\Omega$, $\norm{\mu}_{\calM}$ is the total variation (or mass) of the measure $\mu$, $f:\R^m\to \R\cup\{+\infty\}$ is a convex lower semi-continuous function with non-empty domain and $A: \calM(\Omega) \to \R^m$ is a linear measurement operator. 

An important property of Problem \eqref{eq:primal} is that at least one of its solutions $\mu^\star$ has a support restricted to $s$ distinct points with $s\leq m$ (see e.g. \cite{Zuhovickii1948,fisher_spline_1975,boyer2018representer}), i.e. is of the form
\begin{equation}
\mu^\star = \sum_{i=1}^s \alpha_i^\star \delta_{\xi_i},
\end{equation}
with $\xi_i\in \Omega$ and $\alpha_i^\star\in \R$. This property motivates us to study a class of \emph{exchange} algorithms. They were introduced as early as 1934 \cite{remes1934procede} and then extended in various manners \cite{rettich98}. They consist in discretizing the domain $\Omega$ coarsely and then refining it adaptively based on the analysis of so-called dual certificates. If the refinement process takes place around the locations $(\xi_i)$ only, these methods considerably reduce the computational burden compared to a finely discretized mesh.

Our main results consist in a set of convergence rates for this algorithm that depend on the regularity of $f$ and on the non-degeneracy of a dual certificate at the solution. We also show the linear convergence rate for first order algorithms that continuously vary the coefficients $\alpha_i$ and $x_i$ of a discrete measure. Finally, we show that algorithms alternating between an exchange step and a continuous method share the best of both worlds: the global convergence guarantees of exchange algorithms together with the efficiency of first order methods. This yields a fast adaptive method with strong  convergence guarantees for total variation minimization and related problems.

\subsection{Applications}

Our initial motivation to study the problem \eqref{eq:primal} stems from signal processing applications. 
We recover an infinite dimensional  version of the \emph{basis pursuit} problem \cite{chen2001atomic} by setting
\begin{equation*}
f(x)=\iota_{\{y\}}(x) = \begin{cases}
                                     0 & \textrm{ if } x=y \\
                                     +\infty & \textrm{otherwise.}
                                    \end{cases} 
\end{equation*}
Similarly, the choice $f(x) = \frac{\tau}{2}\|x-y\|_2^2$, leads to an extension of the LASSO \cite{tibshirani1996regression} called Beurling LASSO \cite{de2012exact}. Both problems proved to be extremely useful in engineering applications. They got a significant attention recently thanks to theoretical  progresses in the field of super-resolution \cite{de2012exact,tang2013compressed,candes2014towards,duval2015exact}. Our results are particularly strong for the quadratic fidelity term.

\subsection{Numerical approaches in signal processing}

The progresses on super-resolution \cite{de2012exact,tang2013compressed,candes2014towards,duval2015exact} motivated researchers from this field to develop numerical algorithms for the resolution of Problem \eqref{eq:primal}. By far the most widespread approach is to use a fine uniform discretization and solve a finite dimensional problem. The complexity of this approach is however too large if one wishes high precision solutions. This approach was analyzed from a theoretical point of view in \cite{tang2013sparse,duval2015exact} for instance. The first papers investigating the use of \eqref{eq:primal} for super-resolution purposes advocated the use of semi-definite relaxations \cite{tang2013compressed,candes2014towards}, which are limited to specific measurement functions and domains, such as trigonometric polynomials on the 1D torus $\mathbb{T}$. 
The limitations were significantly reduced in \cite{de2017exact}, where the authors suggested the use of Lasserre hierarchies. These methods are however currently unable to deal with large scale problems.
Another approach suggested in \cite{bredies2013inverse}, and referred to as a Frank-Wolfe algorithm, consists in adding one point to a discretization set iteratively, where a so-called dual certificate is maximal.  
More recently, \cite{traonmilin2018basins} began investigating the use of methods that continuously vary the positions $(x_i)$ and amplitudes $(\alpha_i)$ of discrete measures parameterized as $\mu=\sum_{i=1}^s \alpha_i \delta_{x_i}$. The authors gave sufficient conditions for a simple gradient descent on the product-space $(\alpha,x)$ to converge. 
In \cite{boyd2017alternating} and \cite{denoyelle2018sliding}, this method was used alternatively with a Frank-Wolfe algorithm, the idea being to first add Dirac masses roughly at the right locations and then to optimize their locations and position continuously, leading to promising numerical results. Surprisingly enough, it seems that the connection with the mature field of semi-infinite programming has been ignored (or not explicitly stated) in all the mentioned references. 

%In particular the approaches \cite{bredies2013inverse,denoyelle2018sliding} are strongly related to the class of exchange and algorithms, which are described hereafter.

% [ORGANIZATION: maybe discuss various procedures for the refinement process...]

\subsection{Some numerical approaches in semi-infinite programming}

A semi-infinite program \cite{rettich98,hettich1993semi} is traditionally defined as a problem of the form
\begin{align}\label{eq:SIP}
   \tag{SIP$[\Omega]$}  \min_{\substack{q \in Q\\ c(x,q) \leq 0 , x \in \Omega}} u(q) 
\end{align}
where $Q$ and $\Omega$ are subsets of $\R^n$ and $\R^m$ respectively, $u: Q \to \R$ and $c : \Omega \times Q\to \R$ are functions. The term semi-infinite stems from the fact that the variable $q$ is finite-dimensional, but it is subject to infinitely many constraints $c(x,q) \leq 0$ for $x\in \Omega$. In order to see the connection between the semi-infinite program \eqref{eq:SIP} and our problem \eqref{eq:primal}, we can formulate its \emph{dual}, which reads as 
\begin{equation}\label{eq:dual}
     \tag{$\calD(\Omega)$} \sup_{q \in \R^m, \|A^*q\|_\infty \leq 1} - f^*(q).
\end{equation}
This dual will play a critical role in all the paper and it is easy to relate it to a SIP by setting $Q=\R^m$, $u= f^*$ and $c= \abs{(A^*q)(x)}-1$.

Many numerical methods have been and are still being developed for semi-infinite programs and we refer the interested reader to the excellent chapter 7 of the survey book \cite{rettich98} for more insight. We sketch below two classes of methods that are of interest for our concerns.

\subsubsection{Exchange algorithms} 

A canonical way of discretizing a semi-infinite program is to simply control finitely many of the constraints, say $c(x,q) \leq 0$ for $x \in \Omega_0 \sse \Omega$, where $\Omega_0$ is finite. The discretized problem SIP$[\Omega_0]$ can then be solved by standard proximal methods or interior point methods. In order to obtain convergence towards an exact solution of the problem, it is possible to choose a sequence $(\Omega_k)$ of nested sets such that $\bigcup_{k} \Omega_k$ is dense in $\Omega$. Solving the problems SIP$[\Omega_k]$ for large $k$ however leads to a high numerical complexity due to the high number of discretization points. 
The idea of exchange algorithms is to iteratively update the discretization sets $\Omega_k$ in a more clever manner than simply making them denser. A generic description is given by Algorithm \ref{alg:exchange}. 

\begin{algorithm}[htbp]
\caption{A Generic Exchange Algorithm\label{alg:exchange}}
\begin{algorithmic}[1]
\State \textbf{Input:} Objective function $u$, Constraint function $c$, Constraint sets $\Omega$ and $Q$, Initial discretization set $\Omega_0$.
\While{Not converged}
%\State Compute $\nabla F(y^{(k)}) = A^*(A y^{(k)} -u_0 )$. \Comment{\textcolor{red}{\textbf{99.35$''$}}}
  \State Set $\displaystyle q_k\in \argmin_{\substack{q \in Q \\ c(x,q)\leq 0, x\in \Omega_k}} u(q)$ \\
  \State Set  $\Omega_{k+1} = \mathrm{Update\_Rule}(\Omega_k,q_k,k)$.
\EndWhile
\State \textbf{Output:} The last iterate $q_\infty$.
\end{algorithmic}
\end{algorithm}

In this paper, we consider $\mathrm{Update\_Rule}$s of the form
\begin{align*}
 \Omega_{k+1} \subset \Omega_k \cup \{x_k^1,\hdots, x_k^{p_k}\},
\end{align*}
where the points $x_k^{i}$ are \emph{local maximizers} of $c(\cdot,q_k)$. At each iteration, the set of discretization points can therefore be updated by adding and dropping a few prescribed points, explaining the name 'exchange'. The simplest rule consists of adding the single most violating point, i.e.
\begin{align}\label{eq:Frank_Wolfe}
    \Omega_{k+1}= \Omega_k \cup \argmax_{x\in \Omega} c(x, q_k).
\end{align}
It seems to be the first exchange algorithm and is nearly equivalent to the Remez algorithm from the 30's \cite{remes1934procede}. It can be shown to be equivalent to a Frank-Wolfe (a.k.a. conditional gradient) method up to an epigraphical lift \cite{denoyelle2018sliding}. These methods were introduced in the field of signal processing in \cite{bredies2013inverse} and the connection with exchange algorithms was proposed in \cite{eftekhari2018bridge}.
The update rule \eqref{eq:Frank_Wolfe} is sufficient to guarantee convergence in the generic case and to ensure a decay of the cost function in $O\left(\frac{1}{k}\right)$, see \cite{levitin1966constrained}. Although 'exchange' suggests that points are both added and subtracted, methods for which $\Omega_k \sse \Omega_{k+1}$ are also coined exchange algorithms. The use of such rules often leads to easier convergence analyses, since we get monotonicity of the objective values $u(q_k)$ for free \cite{hettich1993semi}. Other examples \cite{hettich1982} include only adding points if they exceed a certain margin, i.e. $c(x,y) \geq \epsilon_k$, or all local maxima of $c(q_k,\cdot)$. In the case of convex functions $f$, algorithms that both add and remove points can be derived and analyzed with the use of cutting plane methods. All these instances have their pros and cons and perform differently on different types of problems. Since a semi-infinite program basically allows to minimize \emph{arbitrary} continuous and finite dimensional problems, a theoretical comparison should depend on additional properties of the problem.

% \paragraph{Discretization methods}
% 
% The discretization methods differ from exchange rules in that the points added at each iteration belong to a fixed predefined grid. They will be of central interest in this paper. Let $(X_k)_{k\in \mathbb{N}}$ denote a sequence of nested discrete sets of points in $\Omega$ such that $\lim_{k\to \infty} \mathrm{dist}(X_k,\Omega)=0$. The update rule of a discretization method takes the form
% \begin{align}\label{eq:discretization_method}
%     \Omega_{k+1} \subset X_{k+1},
% \end{align}
% where $\Omega_{k+1}$ is chosen on the basis of $\Omega_k$ and $q_k$. For instance, one can consider the family of methods 
% \begin{equation}
%  \Omega_{k+1} \supseteq X_k \cup \argmax_{x\in X_{k+1}} c(q_k, x).
% \end{equation}
% The interest of that rule is that it discards the need to find the local maximizers of $c(\cdot, q_k)$, which might be complicated in practice. Under mild conditions, this approach was proven in \cite{reemtsen1991discretization} to produce a sequence $(q_k)$ that converges - up to a subsequence - to the global minimizer of \eqref{eq:SIP}. In this paper, we will focus on a similar rule. 

\subsubsection{Continuous methods}

Every iteration of an exchange algorithm can be costly: it requires solving a convex program with a number of constraints that increases if no discretization point is dropped. In addition, the problems tend to get more and more degenerate as the discretization points cluster, leading to numerical inaccuracies. In practice it is therefore tempting to use the following two-step strategy: i) find an approximate solution $\mu_k=\sum_{i=1}^{p_k} \alpha_k^i \delta_{x_k^i}$ of the primal problem \eqref{eq:primal} using $k$ iterations of an exchange algorithm and ii) continuously move the positions $X=(x_i)$ and amplitudes $\alpha=(\alpha_i)$ starting from $(\alpha_k,X_k)$ to minimize \eqref{eq:primal} using a nonlinear programming approach such as a gradient descent, a conjugate gradient algorithm or a Newton approach. 

This procedure supposes that the output $\mu_k$ of the exchange algorithm has the right number $p_k=s$ of Dirac masses, that their amplitudes satisfy $\sgn(\alpha_i)=\sgn(\alpha_i^\star)$ and that $\mu_k$ lies in the basin of attraction of the optimization algorithm around the global minimum $\mu^\star$. To the best of our knowledge, knowing a priori when those conditions are met is still an open problem and deciding when to switch from an exchange algorithm to a continuous method therefore relies on heuristics such as detecting when the number of masses $p_k$ stagnates for a few iterations. 
The cost of continuous methods is however much smaller than that of exchange algorithms since they amount to work over a small number $s(d+1)$ of variables. In addition, the instabilities mentioned earlier are significantly reduced for these methods. This observation was already made in \cite{boyd2017alternating,denoyelle2018sliding} and proved in \cite{traonmilin2018basins} for specific problems.

\subsection{Contribution}

Many recent results in the field of super-resolution provide sufficient conditions for a \emph{non degenerate source condition} to hold \cite{candes2014towards, tang2013compressed,dossal2017sampling,bodmann2018compressed,poon2018support}. The non degeneracy means that the solution $q^\star$ of \eqref{eq:dual} is unique and that the \emph{dual certificate} $|A^*q^\star|$ reaches $1$ at exactly $s$ points, where it is strictly concave. The main purpose of this paper is to study the implications of this non degeneracy for the convergence of a class of exchange algorithms and for continuous methods based on gradient descents. Our main results are as follows:
\begin{enumerate}
 \item We show an eventual linear convergence rate of a class of exchange algorithms for convex functions $f$ with Lipschitz continuous gradient. More precisely, we prove that after a finite number of iterations $N$ the algorithm outputs vectors $q_k$ such that the set 
 \begin{equation}\label{eq:defXk}
 X_k \eqdef \{x\in \Omega \, \vert \,  x_k \text{ local maximizer of } \abs{A^*q_k}, \ |A^*q_k|(x)\geq 1\} 
 \end{equation}
 contains exactly $s$-points $(x_k^1, \hdots, x_k^s)$. 
 
 Letting $\widehat{\mu}_k=\sum_{i=1}^s \alpha_i^{k} \delta_{x_i^k}$ denote the solution of the finite dimensional problem $\inf_{\mu \in \calM(X_k)} \|\mu\|_{\calM} + f(A\mu)$, we also show the linear convergence rate of the cost function $J(\mu_k)$ to $J(\mu^\star)$ and of the support in the following sense:  after a number $N$ of initial iterations, it will take no more that $k_\tau= C\log(\tau^{-1})$ iterations to ensure that $\dist(X_{k_\tau + N},\xi)\leq \tau$. A similar statement holds for the coefficient vectors $\alpha^{k}$. %Under additional assumptions we also obtain a linear rate on the coefficients amplitude $\|\alpha_{k+1} -\alpha^\star\|_2^2\leq c\|\alpha_k - \alpha^\star\|_2^2$ for some $c<1$.
 
 \item We also show that a well-initialized gradient descent algorithm on the pair $(\alpha,x)$ converges linearly to the true solution $\mu^\star$ and explicit the width of the basin of attraction.
 
 \item We then show how the proposed guarantees may explain the success of methods alternating between exchange methods and continuous methods at each step, in a spirit similar to the sliding Frank-Wolfe algorithm \cite{denoyelle2018sliding}.
 
 \item We finally illustrate the above results on total variation based problems in 1D and 2D.
\end{enumerate}

\section{Preliminaries}

\subsection{Notation}

In all the paper, $\Omega$ designs an open \emph{bounded} domain of $\R^d$. The boundedness assumptions plays an important role to control the number of elements in discretization procedures. 
A \emph{grid} $\Omega_k$ is a finite set of points in $\Omega$. Its cardinality is denoted by $|\Omega_k|$.
The distance between two sets $\Omega_1$ and $\Omega_2$ is defined by
\begin{equation}
 \dist(\Omega_1,\Omega_2)=\sup_{x_2\in \Omega_2} \inf_{x_1\in \Omega_1} \|x_1-x_2\|_2.
\end{equation}
Note that this definition of distance is not symmetric: in general $\dist(\Omega_1,\Omega_2)\neq \dist(\Omega_2,\Omega_1)$.

We let $\calC_0(\Omega)$ denote the set of continuous functions on $\Omega$ vanishing on the boundary.
The set of Radon measures $\calM(\Omega)$ can be identified as the dual of $\calC_0(\Omega)$, i.e. the set of continuous linear forms on $\calC_0(\Omega)$. 
For any sub-domain $\Omega_k\subset \Omega$, we let $\calM(\Omega_k)$ denote the set of Radon measures supported on $\Omega_k$. For $p\in [1,+\infty]$, the $L^p$-norm of a function $u\in \calC_0(\Omega)$ is denoted by $\|u\|_p$. 
The total variation of a measure $\mu \in \calM(\Omega)$ is denoted $\norm{\mu}_{\calM}$. It can be defined by duality as
\begin{equation}
\norm{\mu}_{\calM} = \sup_{\substack{ u\in \calC_0(\Omega) \\ \|u\|_\infty\leq 1}} \mu(u).
\end{equation}
The $\ell^p$-norm of a vector $x\in \R^m$ is also denoted $\|x\|_p$. The Frobenius norm of a matrix $M$ is denoted by $\|M\|_F$.

Let $f:\R^m\to \R\cup\{+\infty\}$ denote a convex lower semi-continuous function with non-empty domain $\dom(f) = \{x\in \R^m, f(x)<+\infty\}$.
Its subdifferential is denoted $\partial f$. Its Fenchel transform $f^*$ is defined by
\begin{equation*}
 f^*(y)=\sup_{x\in \R^m} \langle x,y \rangle - f(x).
\end{equation*}
If $f$ is differentiable, we let $f'\in \R^m$ denote its gradient and if it is twice differentiable, we let $f''\in \R^{m\times m}$ denote its Hessian matrix. We let $\|f'\|_\infty=\sup_{x\in \Omega} \|f'(x)\|_2$ and $\|f''\|_\infty=\sup_{x\in \Omega} \|f''(x)\|$, where $\|f''(x)\|$ is the largest singular value of $f''(x)$. 
A convex function $f$ is said to be $l$-strongly convex if 
\begin{equation}
 f(x_2)\geq f(x_1) + \langle \eta , x_2-x_1 \rangle + \frac{l}{2}\|x_2-x_1\|_2^2
\end{equation}
for all $(x_1,x_2)\in \R^m\times \R^m$ and all $\eta \in \partial f(x_1)$.
A differentiable function $f$ is said to have an $L$-Lipschitz gradient if it satisfies $\|f'(x_1)-f'(x_2)\|_2\leq L \|x_1-x_2\|_2$. This implies that
\begin{equation} \label{eq:LSmoothVariation}
 f(x_2)\leq f(x_1) + \langle f'(x_1) , x_2-x_1 \rangle +\frac{L}{2}\|x_2-x_1\|_2^2 \mbox{ for all } (x_1,x_2)\in \R^m\times \R^m.
\end{equation}
We recall the following equivalence \cite{hiriart2013convex}:
\begin{prop}\label{eq:duality_lipschitz_stronglyconvex}
Let $f:\R^m\to \R\cup\{+\infty\}$ denote a convex and closed function with non empty domain.
 Then the following two statements are equivalent:
 \begin{itemize}
  \item $f$ has an $L$-Lipschitz gradient.
  \item $f^*$ is $\frac{1}{L}$-strongly convex.
 \end{itemize}
\end{prop}

The linear measurement operators $A$ considered in this paper can be viewed as a collection of $m$ continuous functions $(a_i)_{1\leq i\leq m}$. For $x\in \Omega$, the notation $A(x)$ corresponds to the vector $[a_1(x),\hdots, a_m(x)] \in \R^m$.

%$A: \calM(\Omega) \to \R^m$ is a linear measurement operator. 
%A triangulation $\calT$ is a partition of $\Omega$ as a set of simplices $\omega_l$. [REMOVE?]

\subsection{Existence results and duality}

In order to obtain existence and duality results, we will now make further assumptions.
\begin{assumption} \label{ass:f}
    $f: \R^m \to \R \cup \set{\infty}$ is convex and lower bounded. In addition, we assume that either $\dom(f)=\R^m$ or that $f$ is polyhedral (that is, its epigraph is a finite intersection of closed halfspaces).
\end{assumption}
\begin{assumption} \label{ass:A}
    The operator $A$ is weak-$*$-continuous. Equivalently, the \emph{measurement functionals} $a_i^*$ defined by $\sprod{a_i^*, \mu} = (A(\mu))_i$ are given by 
    \begin{align*}
        \sprod{a_i^*, \mu} = \int_{\Omega} a_i d\mu,
    \end{align*}
    for functions $a_i \in \calC_0(\Omega)$. In addition, we assume that $A$ is surjective on $\R^m$.
\end{assumption}

The following results relate the primal and the dual. 
\begin{prop}[Existence and strong duality] \label{prop:duality}
Under Assumptions \ref{ass:f} and \ref{ass:A}, the following statements are true:
\begin{itemize}
 \item The primal problem \eqref{eq:primal} and its dual \eqref{eq:dual} both admit a solution. 
 \item The following strong duality result holds
\begin{equation}
 \min_{\mu \in \calM(\Omega)} \norm{\mu}_{\calM(\Omega)} + f(A\mu) = \max_{q \in \R^m, \|A^*q\|_\infty \leq 1} - f^*(q).
\end{equation}
 \item Let $(\mu^\star,q^\star)$ denote a primal-dual pair. They are related as follows
 \begin{equation}\label{eq:primal_dual_relationship}
  A^*q^\star \in \partial_{\|\cdot\|_\calM}(\mu^\star) \mbox{ and } -q^\star\in \partial f(A\mu^\star).
 \end{equation}
\end{itemize}
\end{prop}
\begin{proof}
The stated assumptions ensure the existence of a feasible measure $\mu$. In addition, the primal function is coercive since $f$ is bounded below. This yields existence of a primal solution. The existence of a dual solution stems from the compactness of the set  $\{q \in \R^m, \|A^*q\|_\infty \leq 1\}$ (which itself follows from the surjectivity of $A$) and the continuity of $f^*$ on its domain. The strong duality result follows from \cite[Thm 4.2]{Borwein1992}. The primal-dual relationship directly derives from the first order optimality conditions.
\end{proof}

The left inclusion in equation \eqref{eq:primal_dual_relationship} plays an important role, which is well detailed in \cite{duval2015exact}. It implies that the support of $\mu^\star$ satisfies: $\supp(\mu^\star)\subseteq \{x\in \Omega, |A^*q^\star(x)|=1\}$. %In particular, if the active set $\Omega^\star = \{x\in \Omega, |A^*q^\star(x)|=1\}$ is finite, we can replace the infinite dimensional primal problem $\calP(\Omega)$ by its finite dimensional counterpart $\calP(\Omega^\star)$.

\section{An Exchange Algorithm and its convergence}

 \label{sec:exchange}
 
\subsection{The algorithm}

We assume that an initial grid $\Omega_0 \sse \Omega$ is given (e.g. a coarse Euclidean grid). 
Given a discretization $\Omega_k$, we can define a discretized primal problem \eqref{eq:discretePrimal} 
\begin{align*} \label{eq:discretePrimal}
    \tag{$\calP(\Omega_k)$} \inf_{\mu \in \calM(\Omega_k)} \norm{\mu}_{\calM} + f(A\mu),
\end{align*}
and its associated dual \eqref{eq:discreteDual}
\begin{align*}
    \tag{$\calD(\Omega_k)$}  \sup_{q \in \R^m, |A^*q(x)|\leq 1, \ \forall x\in \Omega_k} - f^*(q). \label{eq:discreteDual}
\end{align*}
%From this set, we can construct an associated triangulation $\calT_0$ consisting of simplices $(\omega_\ell)_\ell$ with vertices in $\Omega_0$.  A typical way to construct it is to use a Delaunay triangulation. For a grid $\Omega_k$ with associated triangulation $\calT_k$, we let $A_k$ denote the measurement operator with measurement functions $(a^k_i)_{i=1}^m$ defined through the following conditions:
% \begin{itemize}
%  \item $a^k_i=a_i$ on $\Omega_k$.
%  \item $a_i^k$ is linear on any triangle of $\calT_k$.
% \end{itemize}
%  
% Given a discretization $\Omega_k$, we can define a discretized primal problem \eqref{eq:discretePrimal} 
% \begin{align*} \label{eq:discretePrimal}
%     \tag{$\calP(\Omega_k)$} \inf_{\mu \in \calM(\Omega_k)} \norm{\mu}_{\calM(\Omega)} + f(A\mu)=\inf_{\mu \in \calM} \norm{\mu}_{\calM} + f(A_k\mu),
% \end{align*}
% and its associated dual \eqref{eq:discreteDual}
% \begin{align*}
%     \tag{$\calD(\Omega_k)$}  \sup_{q \in \R^m, |A^*q(x)|\leq 1, \ \forall x\in \Omega_k} - f^*(q) = \sup_{q \in \R^m, \norm{A_k^*q}_\infty \leq 1} - f^*(q). \label{eq:discreteDual}
% \end{align*}
% The ``equivalence'' between discretizing the domain $\Omega_k$ or the measurement operator $A_k$ was proven in \cite{flinth2017exact}.

 In this paper, we will investigate the exchange rule below:
 \begin{equation}\label{eq:exchange_rule}
  \Omega_{k+1}=\Omega_{k} \cup X_{k} \mbox{ where } X_k \mbox{ is defined in } \eqref{eq:defXk}.
 \end{equation}
 The implementation of this rule requires finding $X_k$, the set of  all the local maximizers of $\abs{A^*q_k}$ exceeding $1$. %This might be hard from a computational point of view, see Section \ref{sec:numerics} for a discussion on this issue.

%Section~\ref{sec:exchange} is organized as follows: In Section~\ref{sec::conv-exchange} we prove the convergence of this algorithm under mild assumption. Section~\ref{sec::exchange::assumption} is devoted to introducing the assumptions needed for obtaining a convergence rate of the algorithm, Section~\ref{sec::exchange::auxiliary} introduces some preliminary lemmata, Section~\ref{sec::exchange::fixed:grids} details fixed grid estimates and finally Theorem~\ref{th:MainExchange}  in Section~\ref{sec::exchange:linear} proves the eventual linear convergence rate of the  algorithm.

\subsection{A generic convergence result}
\label{sec::conv-exchange}
The exchange algorithm above converges under quite weak assumptions. 
For instance, it is enough to assume that the function $f$ is differentiable.
\begin{assumption} \label{ass:grid}
  The data fitting function $f:\R^m \to \R$ is differentiable with $L$-Lipschitz continuous gradient.
\end{assumption}

Alternatively, we may assume that the initial set $\Omega_0$ is fine enough,  which in particular implies that $|\Omega_0|\geq m$.
\begin{assumption} \label{ass:grid2}
  The initial set $\Omega_0$ is such that $A$ restricted to $\Omega_0$ is surjective.
\end{assumption}

We may now present and prove our first result.

\begin{theo}[Generic convergence\label{thm:generic_convergence}]
    Under assumptions \ref{ass:f}, \ref{ass:A} and  \ref{ass:grid} or \ref{ass:grid2}, a subsequence of $(\mu_k,q_k)$ will converge in the weak-$*$-topology towards a solution pair $(\mu^\star,q^\star)$ of \eqref{eq:primal} and \eqref{eq:dual}, as well as in objective function value. If the solution of \eqref{eq:primal} and/or \eqref{eq:dual} is unique, the entire sequence will converge.
\end{theo}
\begin{proof}
 First remark that the sequence $(\|\mu_k\|_\calM +f(A\mu_k))_{k \in \N}$ is non-increasing since the spaces $\calM(\Omega_k)$ are nested. Due to the boundedness below of $f$, the same must be true for $(\norm{\mu_k}_\calM)$. Hence there exists a subsequence $(\mu_k)$, which we do not relabel, that weak-$*$ converges towards a measure $\mu_\infty$. 

Now, we will prove that the sequence of dual variables $(q_k)_{k\in \N}$ is bounded. 
If Assumption \ref{ass:grid} is satisfied, then $f^*$ is strongly convex and since $0$ is a feasible point, we must have $q_k\in \{q\in \R^m, f^*(q)\leq f^*(0)\}$, which is bounded. 
Alternatively, if Assumption \ref{ass:grid2} is satisfied, notice that $1\geq \|A_k^* q_k\|_\infty \geq \|A_0^* q_k\|_\infty$. Since $A_0$ is surjective, the previous inequality implies that $(\|q_k\|_2)_{k \in \N}$ is bounded. Hence, in both cases, the sequence $(q_k)_{k\in\N}$ converges up to a subsequence to a point $q_\infty$.

%If $\norm{A^*q_k}_\infty \leq 1$ for some $k$, we will not add points to the grid, so that $q_{k'}=q_k$ for all $k'\geq k$, and in particular $\norm{A^*q_\infty}_\infty = \norm{A^*q_k}_\infty \leq 1$. 

The key is now to  prove that $\|A^*q_\infty\|_\infty \leq 1$. To this end, let us first argue that the family $(A^*q_k)_{k\in \N}$ is equicontiuous. For this, let $\epsilon>0$ be arbitrary. Since the functions $a_i \in \calC_0(\Omega)$ all are uniformly continuous, there exists a $\delta>0$ with the property
\begin{align*}
    \norm{x-y}_2 < \delta \, \Rightarrow \, \abs{a_i(x)-a_i(y)} < \frac{\epsilon}{\sup_{k} \norm{q_k}_1} \text{ for all } i.
\end{align*}
Consequently,
\begin{align}
 \norm{x-y}_2 < \delta \, \Rightarrow \, \abs{(A^*q_k)(x)- (A^*q_k)(y)} &= \abs{\sum_{i=1}^m (a_i(x)-a_i(y))q_k(i)} \leq \sum_{i=1}^m \abs{a_i(x)-a_i(y)}\abs{q_k(i)} \nonumber \\
 &< \frac{\epsilon}{\sup_{k} \norm{q_k}_1} \sum_{i=1}^m \abs{q_k(i)} \leq \epsilon. \label{eq:uniformcont}
\end{align}
%That is equicontinuity. (We could also just envoke Arzel\`{a}-Ascoli: $(A^*q_k)_{k\in \N}$ is, being a convergent sequence, relatively compact (as a set) in $\calC_0(\Omega)$...)

Due to the convergence of $(q_k)_{k \in \N}$, the sequence $(A^*q_k)_{k \in \N}$ is converging strongly to $A^*q_\infty$. We will now prove that $\|A^*q_\infty\|_\infty \leq 1$. If for some $k$, $\norm{A^*q_k}_\infty \leq 1$, we will have $A^*q_\ell = A^*q_k$ for all $\ell \geq k$, and in particular $q_\infty = q_k$ and thus $\norm{A^*q_\infty}\leq 1$. Hence, we may assume that $\norm{A^*q_k}_\infty >1$ for each $k$, i.e. that we add at least one point to $\Omega_k$ in each iteration.

Now, towards a contradiction, assume that $\norm{A^*q_\infty}_\infty =1 + 2\epsilon$ for an $\epsilon>0$. Set $\delta$ as in \eqref{eq:uniformcont}. For each $k \in \N$, let $x_k^\star$ be the element in  $\argmax_x \abs{(A^*q_k)(x)}$ which has the largest distance to $\Omega_k$. Due to $a_\ell \in \calC_0(\Omega)$ for each $k$, there needs to exist a compact subset $C \sse \Omega$ such that $(x_k^\star)_k \sse C$. Indeed, there exists for each $\ell=1, \dots, m$ a $C_\ell$ such that $\abs{a_\ell(x)}\leq (\sup_{k} \norm{q_k}_1)^{-1}$ for all $x \notin C_\ell$. Now, if $x \notin C\eqdef \bigcup_{\ell=1}^m C_\ell$, we get
\begin{align*}
 \abs{A^*q_k(x)} =\abs{\sum_{i=1}^m a_i(x)q_k(i)} \leq \sum_{i=1}^m \abs{a_i(x)}\abs{q_k(i)} \nonumber < \frac{1}{\sup_{k} \norm{q_k}_1} \sum_{i=1}^m \abs{q_k(i)} \leq 1
\end{align*}
for every $k$. Since $\abs{A^*q_k(x_k^\star)}>1$, we conclude $(x_k^\star)_k \sse C$. Consequently, a subsequence (which we do not rename) of $(x_k^\star)$ must converge. Thus, for some $k_0$ and every $k > k_0$, we have $\norm{x_k^\star- x_{k_0}^\star}_2 < \delta$. We then have
\begin{align*}
    \norm{A^*q_k}_\infty = \abs{(A^*q_k)(x_k^\star)} < \abs{(A^*q_k)(x_{k_0}^\star)} + \epsilon \leq 1+ \epsilon.
\end{align*}
In the last estimate, we used the constraint of \eqref{eq:discreteDual} and the fact that $x_{k_0}^\star \in \Omega_k$. Since the last inequality holds for every $k\geq k_0$, we obtain
\begin{align*}
    \norm{A^*q_\infty}_\infty = \lim_{k \to \infty} \norm{A^*q_k}_\infty \leq 1+ \epsilon,
\end{align*}
where we used the fact that $(A^*q_k)_k$ converges strongly towards $A^*q_\infty$. This is a contradiction, and hence, we do have $\norm{A^*q_\infty}_\infty \leq 1$.

% 
% 
% Let us define
% \begin{equation}
% \bar \Omega=\overline{\cup_{k\in \N} \Omega_k}. 
% \end{equation}
% We have $|A^*q^\star(x)|\leq 1$ for $x\in \bar \Omega$. To see this, take a point $x_j\in \Omega_{k_0}$, for an arbitrary $k_0\in \N$. We then have, due to the nestedness of the $\Omega_k$ and the constraint $\|A^*q_k|_{\Omega_k}\|_\infty\leq 1$, that $|A^*q_k (x_j)|\leq 1$ for all $k\geq k_0$. Consequently, 
% \begin{align*}
%  |A^*q^\star(x_j)| = \lim_{k\to \infty} |A^*p_k(x_j)| \leq 1.
% \end{align*}
% We can then take the closure by continuity of the functions $a[i]$.
% 
% Now assume that there exists a point $x\in \Omega$ such that $\|A^*q^\star(x)\|>1$. Then, there must exist an integer $k_0$ such that $\|A^*q_k(x)\|>1$ for all $k\geq k_0$. Let $\gamma_k(x)$ be any cell in $\calT_k$ containing $x$\footnote{In general, only one triangle contains $x$ except in the unlikely situation where $x$ belongs to an edge (or a face of lower dimension). To make a complexity analysis in dimension 2 or more, we will probably need to control those situations. [THIS IS RESOLVED???]}. The algorithm would then add points in such a way that $\diam(\gamma_{k+1}(x)) \leq \theta \cdot \diam(\gamma_{k}(x))$. Since this property is true for any $k\geq k_0$, $\lim_{k\to \infty} \diam(\gamma_{k}(x))=0$ and $x\in \bar \Omega$. This is a contradiction.

Overall, we proved that the primal-dual pair $(\mu_\infty,q_\infty)$ is feasible. It remains to prove that it is actually a solution.
To do this, let us first remark that $\norm{\mu_\infty}_\calM + f(A\mu_\infty) \geq - f^*(q_\infty)$ by weak duality. To prove the second inequality, first notice that the weak-$*$-continuity of $A$ implies that $A\mu_k \to A\mu_\infty$. Assumption \ref{ass:f} furthermore implies that $f$ is lower semi-continuous. As a supremum of linear functions, so is $f^*$. Since also $q_k \to q_\infty$, we conclude
\begin{align*}
    f^*(q_\infty) + f(A\mu_\infty) \leq \liminf_{k \to \infty} f^*(q_k) + f(A\mu_k).
\end{align*}
Assumptions \ref{ass:f}, \ref{ass:A} together with Proposition \ref{prop:duality} imply exact duality of the discretized problems. This means $f^*(q_k) + f(A\mu_k) = -\norm{\mu_k}_\calM$. Since the norm is weak-$*$-l.s.c. , we thus obtain
\begin{align*}
    \liminf_{k \to \infty} f^*(q_k) + f(A\mu_k) = \liminf_{k \to \infty}  - \norm{\mu_k}_\calM \leq - \liminf_{k \to \infty} \norm{\mu_k}_\calM \leq -\norm{\mu_\infty}_\calM.
\end{align*}
Reshuffling these inequalities yields $\norm{\mu_\infty}_\calM + f(A\mu_\infty) \leq - f^*(q_\infty)$, i.e., the reverse inequality.
Thus, $\mu_\infty$ and $q_\infty$ fulfill the duality conditions, and are solutions. The final claim follows from a standard subsequence argument.
\end{proof}

\begin{rem}
Let us mention that the convergence result in Theorem \ref{thm:generic_convergence} and its proof, is not new, see e.g. \cite{reemtsen1990modifications}. The proof technique can be applied to prove similar statements for other refinement rules. 
For instance, the result still holds if we add the single most violating point:
 \begin{equation}
  \Omega_{k+1} \supseteq \Omega_k \cup \{x_k\} \mbox{ with } x_k \in \argmax_{x\in \Omega} |A^*q_k|.
 \end{equation} 
% For instance, the result still holds if we replace the refinement rule by
%  \begin{equation}
%   \Omega_{k+1} \supseteq \Omega_k \cup \{x_k\}
%  \end{equation} 
%  with $x_k \in \argmax_{x\in G_{k+1}} |A^*q_k|$, where $(G_k)_k$ is a nested sequence of grids whose union is dense in $\Omega$, see \cite{reemtsen1991discretization}. This latest rule is simpler to implement since it only requires to assess a \emph{discrete} $\argmax$. However, the complexity of computing this $\argmax$ gets larger as $G_k$ gets finer. 
 \end{rem}
 
 The result that we have just shown is very generally applicable. It however does not give us any knowledge of the convergence rate. The next section will be devoted to proving a linear convergence rate in a significant special case.

% Main results

\subsection{Non degenerate source condition}
\label{sec::exchange::assumption}
%A priori, we could add densely

The idea behind adding points to the grid adaptively is to avoid a uniform refinement, which results in computationally expensive problems \eqref{eq:discreteDual}. However, there is a priori no reason for the exchange rule not to refine in a uniform manner.
%this: the local maximizers $\abs{A^*q_k}$ exceeding $1$ could 'move around' in $\Omega$ in an arbitrary fashion.
In this section, we prove that additional assumptions improve the situation. 
First, we will from now on work under Assumption \eqref{ass:grid}. It implies that the dual solutions $q_k$ are unique for every $k$, since Proposition \eqref{eq:duality_lipschitz_stronglyconvex} ensures the strong convexity of the Fenchel conjugate $f^*$. We furthermore assume that the $a_j$ are smooth.

\begin{assumption}[Assumption on the measurement functionals \label{ass:regf}]
The measurement functions $a_j$ all belong to $\calC_0^2(\Omega) \eqdef \calC_0(\Omega) \cap \calC^2(\Omega)$ and their first and second order derivatives are uniformly bounded on $\Omega$. We hence may define
\begin{align*}
     \kappa \eqdef \sup_{\|q\|_2\leq 1} \norm{A^*q}_\infty= \sup_{x\in \Omega} \norm{A(x)}_2, \quad \kappa_\nabla \eqdef   \sup_{\|q\|_2\leq 1} \norm{(A^*q)'}_\infty, \quad \kappa_{\hess} \eqdef  \sup_{\|q\|_2\leq 1} \norm{(A^*q)''}_\infty.
\end{align*}
\end{assumption}

We also assume the following regularity condition on the solution $q^\star$ of \eqref{eq:dual}, and its corresponding primal solution $\mu^\star$.
\begin{assumption}[Assumption on the primal-dual pair] \label{ass:dualCert}
We assume that \eqref{eq:primal} admits a unique $s$-sparse solution $\mu^\star$ supported on $\xi= (\xi_i)_{i=1}^s\in \Omega^s$:
\begin{equation}
 \mu^\star = \sum_{i=1}^s \alpha_i^\star \delta_{\xi_i}.
\end{equation}
Let $q^\star$ denote the associated dual pair. We assume that the only points $x$ for which $\abs{A^*q^\star(x)}=1$ are the points in $\xi$, and that the second derivative of  $\abs{A^*q^\star}$ is negative definite in each point $\xi_i$. It follows that there exists $\tau_0>0$ and $\gamma>0$ such that
    \begin{align} \label{eq:reg}
    |A^*q^\star|''(x) \preccurlyeq - \gamma \Id \text{ and } |A^*q^\star|(x) \geq \frac{\gamma\tau_0^2}{2} &\text{ for $x$ with } \dist(x, \xi) \leq \tau_0. \\
\abs{(A^*q^\star)(x)} \leq 1- \frac{\gamma \tau_0^2}{2} \ &\text{ for $x$ with } \dist(x, \xi) \geq \tau_0. \label{eq:reg2}
    \end{align}
    We note that  if Equations~\eqref{eq:reg} and \eqref{eq:reg2} are valid for some $(\gamma,\tau_0)$, they are also valid for any $(\tilde \gamma,\tilde \tau_0)$ with
   $\tilde \gamma \le \gamma$ and $\tilde \tau_0 \le \tau_0$. 
\end{assumption}

Assumption \eqref{ass:dualCert} may look very strong and hard to verify in advance. 
Recent advances in signal processing actually show that it is verified under clear geometrical conditions. First, there will always exists at most $m$-sparse solutions to problem \eqref{eq:primal}, \cite{Zuhovickii1948,fisher_spline_1975,boyer2018representer}. 
Therefore, the main difficulty comes from the uniqueness of the primal solution and from the two regularity conditions \eqref{eq:reg} and \eqref{eq:reg2}. These assumptions are called \emph{non-degenerate source condition} of the \emph{dual certificate $A^*q^\star$} \cite{duval2015exact}. Many results in this direction have been shown for $f= \xi_{\set{b}}$ or $f(\cdot)=\frac{L}{2}\|\cdot - b\|_2^2$, where $b=A\mu_0$ with $\mu_0$ a finitely supported measure. The papers \cite{candes2014towards, tang2013compressed,dossal2017sampling} deal with different Fourier-type operators, \cite{bodmann2018compressed} about a few other special cases whereas \cite{poon2018support} provides an analysis for arbitrary integral operators sampled at random.

% Common for all these considerations is that the measure obeys a separation condition on the ground truth measure $\mu_0$, and that the \emph{kernel function}
% \begin{align*}
%     K(x,y) = \sum_{i=1}^m a_i(x)a_i(y)
% \end{align*}
% is highly concentrated along the line $x=y$. 
% %The careful reader has probably already noticed that $f= \chi_{\set{b}}$ is not a smooth function. However, carefully examining the mentioned works reveals that they all provide arguments for the \emph{minimal $\ell_2$-norm certificate} to obey estimates such as \eqref{eq:reg}. This is enough to prove that the BLASSO ($f= \frac{1}{2\lambda} \norm{A\mu-b}_2^2$) for small $\lambda$ has similar properties -- this is one of the main findings in \cite{duval2015exact}. 
% Although concrete results are not known for other data fitting functions $f$, we choose to carry out our analysis in this generality, mainly because the exposition of the arguments get more transparent.

%If (dual) solution is nice, we will not

\subsection{Auxiliary results}
\label{sec::exchange::auxiliary}
In this and the following sections, we always work under Assumptions \ref{ass:f}, \ref{ass:A}, \ref{ass:grid}  without further notice. We derive several lemmata that are direct consequences of the above assumptions. The first two rely strongly on the Lipschitz regularity of the gradient of $f$.
\begin{lem}[Boundedness of the dual variables \label{lem:boundedness_dual}]
Let $\bar q=\argmin_{q\in \R^m} f^*(q)$ denote the prox-center of $f^*$. For all $k\in \N$, we have 
\begin{equation}
\|q_k\|_2\leq \sqrt{2L (f^*(0)-f^*(\bar q))} + \|\bar q\|_2 \eqdef R. 
\end{equation}
\end{lem}
\begin{proof}[Proof of Lemma \ref{lem:boundedness_dual}]
  For all $k\in \N$, we have $0\in \{q\in \R^m, \|A^*_k q\|_\infty \leq 1\}$, hence $f^*(q_k)\leq f^*(0)$. By strong convexity of $f^*$ and optimality of $\bar q$ and $q_k$, we get:
\begin{equation}
 f^*(0)\geq f^*(q_k) \geq f^*(\bar q)+\frac{1}{2L}\|q_k-\bar q\|_2^2.
\end{equation}
Therefore $\|q_k-\bar q \|_2\leq \sqrt{2L (f^*(0)-f^*(\bar q))}$ and the conclusion follows from a triangle inequality.
\end{proof}

%The next proposition is a general assertion on constrained minimization of convex function with Lipschitz continuous gradient. 
\begin{prop} \label{prop:supvsq_raw}
Let $q^\star$ be the solution of \eqref{eq:dual}. 
Let 
\begin{equation*}
\rho\eqdef \sqrt{\sup_{w\in \partial f^*(q^\star)} -L \sprod{w,q^\star}}. 
\end{equation*}
Then for any $q$, we have
 \begin{align*}
    f^*(q^\star)-f^*(q) + \frac{1}{2L} \norm{q-q^\star}_2^2 \leq \rho^2 L^{-1}(\sup_{x\in \xi} |A^*q|(x) -1).
 \end{align*}

\end{prop}
\begin{proof} Let $M=\{q\in\R^m, f^*(q)\leq f^*(q^\star)\}$ denote the sub-level set of $f^*$ and $
    D = \set{q\in \R^n \, \vert \, \sup_{x\in \xi} |A^*q|(x) \leq 1}$. We first claim that $M$ and $D$ only have the point $q^\star$ in common. Indeed $\mu^\star$ solves the problem $\mathcal P(\xi)$ and by strong duality of the problem restricted to $\calM(\xi)$, $q^\star$ solves $\mathcal D(\xi)$. 
By strong convexity of $f$, $q^\star$ is the unique solution $\mathcal D(\xi)$, this exactly means $M\cap D= \{q^\star\}$.
%    To prove this, consider the following problem
%    \begin{equation*}
%    \sup_{q\in \R^m, |A^*q(x)|\leq 1 \ \forall x\in \xi} -f^*(q).
%    \end{equation*}
%    If $q^\star$ solves this problem, it implies by strong convexity of $f^*$ that it is the unique solution. This exactly means $M\cap D= \{q^\star\}$. By definition of $q^\star$, we have
%    \begin{equation*}
%     -f^*(q^\star)=\sup \eqref{eq:dual}=\inf \eqref{eq:primal} =  \inf_{\mu \in \calM(\xi)} J(\mu)=\sup_{q\in \R^m, |A^*q(x)|\leq 1 \ \forall x\in \xi} -f^*(q),
%    \end{equation*}
%    by strong duality of the problem restricted to $\calM(\xi)$. This proves the claim.

     The fact that $M\cap D= \{q^\star\}$ implies that there exists a separating hyperplane there. Since the hyperplane must be tangent to $M$, it can be written as $\set{ q \, \vert \, \sprod{w, q} = \sprod{w,q^\star}}$ for a $w \in \partial f^*(q^\star)$, with $D  \subset \set{ q \, \vert \, \sprod{w, q} \geq \sprod{w,q^\star}}$. Consequently, letting $\epsilon=\sup_{x\in \xi} |A^*q(x)|-1$, we have
$$(1+\epsilon)D \subset \set{ q \, \vert \, \sprod{w, q} \geq (1+\epsilon)\sprod{w,q^\star}}= \set{ q \, \vert \, \sprod{w, q- q^\star} \geq \epsilon \sprod{w,q^\star}}.$$ 
Now, the strong convexity of $f^*$ implies for every $q \in (1+\epsilon)D \cap M$,
\begin{align*}
   f^*(q) \geq f^*(q^\star) + \sprod{w, q-q^\star} + \frac{1}{2L} \norm{q-q^\star}_2^2 \geq  f^*(q^\star) + \epsilon \sprod{w, q^\star} + \frac{1}{2L} \norm{q-q^\star}_2^2.
\end{align*}
Rearranging this, we obtain
\begin{align*}
 -\epsilon \sprod{w, q^\star} \geq  f^*(q^\star) -f^*(q) +  \frac{1}{2L} \norm{q-q^\star}_2^2.
\end{align*}
which is the claim.
\end{proof}

Before moving on, let us record the following proposition:
\begin{prop}\label{prop:Lipschitz_A_Aprime}
We have
\begin{eqnarray}
\label{eq:kappa:taylor}
\norm{A(x)-A(y)}_2\le \kappa_\nabla \Vert x-y\Vert_2 \quad \text{ and }\quad 
\norm{A'(x)-A'(y)}_F\le \kappa_{\hess} \Vert x-y\Vert_2.
\end{eqnarray}
\end{prop}
\begin{proof}
The proof of the first inequality of \eqref{eq:kappa:taylor} is a standard Taylor expansion :
\begin{align*}
   \norm{A(x) - A(y)}_2 &= \sup_{\substack{q \in \R^m \\ \norm{q}_2=1}} \sprod{q, A(x) - A(y)} = \sup_{\substack{q \in \R^m \\ \norm{q}_2=1}} \abs{(A^*q)(x) - (A^*q)(y)} \\
   &\leq \sup_{\substack{q \in \R^m \\ \norm{q}_2=1}} \sup_{z \in [x,y]} \sprod{(A^*q)'(z),x-y} \leq  \sup_{\substack{q \in \R^m \\ \norm{q}_2=1}} \norm{(A^*q)'}_\infty \norm{x-y}_2 \leq \kappa_\nabla \Vert x-y\Vert_2.
\end{align*}
The proof of the second part of \eqref{eq:kappa:taylor} follows the same lines as the first part and is left to the reader.
\end{proof}

 The next two lemmata aim at transferring bounds from the geometric distances of the sets $X_k$, $\Omega_k$ and $\xi$ to bounds on $|A^* q_k(\xi)|$. Using Proposition~\ref{prop:supvsq_raw}, we may then  transfer these bounds to bounds on the errors of the dual solutions and the dual (or primal) objective values.

\begin{lem} \label{lem:SupVsq}
The following inequalities hold
\begin{align}
  \norm{A^*q_k}_\infty \leq 1 + \frac{R\kappa_{\hess}}{2} \dist(\Omega_k, X_k)^2, \label{eq:normBound}
\end{align}
 \begin{align*}
 f^*(q^\star) -f^*(q_k) &\leq  \frac{R\kappa_{\hess}\rho^2}{2L} \dist(\Omega_k, X_k)^2, \\
  \norm{q_k - q^\star}_2 &\leq \dist(\Omega_k, X_k) \sqrt{R\kappa_{\hess}} \rho.
 \end{align*}
\end{lem}
\begin{proof}[Proof of Lemma \ref{lem:SupVsq}]
To show \eqref{eq:normBound}, first notice that 
\begin{align}
  \norm{A^*q_k}_\infty \leq 1 + \norm{(A^*q_k)''}_\infty \frac{\dist(\Omega_k, X_k)^2}{2}. \label{eq:normBound2}
\end{align}
%Let $B(\xi_i,\tau)$ denote an $\ell^2$-ball centered at $\xi_i$ of radius $\tau$.
Indeed, by definition, the global maximum $z$ of $|A^*q_k|$ lies in $X_k$ and satisfies $(A^*q_k)'(z)=0$. Furthermore, by construction, all points $x$ in $\Omega_k$ satisfy $|A^*q_k(x)|\leq 1$. 
Using a Taylor expansion, we get for all $x\in \Omega$
\begin{align*}
    \abs{A^*q_k(x)-A^*q_k(z)} \leq \norm{(A^*q_k)''}_\infty \frac{\norm{x-z}_2^2}{2}.
\end{align*}
Taking $x$ as the point in $\Omega_k$ minimizing the distance to $z$ leads to \eqref{eq:normBound2}. In addition, we have  $\norm{(A^*q_k)''}_\infty\leq R\kappa_{\hess}$ by Lemma \ref{lem:boundedness_dual}, so that $\norm{A^*q_k}_\infty \leq 1+\epsilon$ with $\epsilon=R\kappa_{\hess} \frac{\dist(\Omega_k, X_k)^2}{2}$.

Now, letting $C= \set{q \, \vert \, \norm{A^*q}_\infty \leq 1}$, we have just proven that $q_k \in (1+\epsilon)C$. Furthermore, due to the optimality of $q_k$ for the discretized problem and to the fact that $q^\star$ is feasible for that problem, we will have $f^*(q_k) \leq f^*(q^\star)$, i.e., $q_k$ is included in the $f^*(q^\star)$-sub-level set of $f^*$: $M=\{q\in \R^m | f^*(q)\leq f^*(q^\star) \}$. An application of Proposition~\ref{prop:supvsq_raw} now yields the result.
\end{proof}

\begin{lem} \label{lem:gridVsQ}
Suppose that $\dist(X_k,\xi)\leq \delta$ and $\dist(\Omega_k,\xi)\leq \delta$. Then
    \begin{align*}
        f^*(q^\star) - f^*(q_k) &\leq \frac{2R\kappa_{\hess}\rho^2}{L} \cdot \delta  \dist(\Omega_k,\xi) \\ 
        \norm{q_k- q^\star}_2 &\leq  \rho\sqrt{2R\kappa_{\hess}}\sqrt{ \delta \cdot \dist(\Omega_k, \xi)}.
    \end{align*}
\end{lem}
\begin{proof}
Let $y_k^i$ (resp. $x_k^i$) be the point closest to $\xi_i$ in $\Omega_k$ (resp. $X_k$).
 By assumption, we have $\|x_k^i-y_k^i\|_2\leq 2\delta$. 
For all $i$, we have 
\begin{align}
 |A^*q_k(\xi_i)|\leq |A^*q_k(y_k^i)| + \sup_{z\in [y_k^i, \xi_i]}  \|(A^*q_k)'(z)\|_2 \|\xi_i -y_k^i\|_2 \leq 1 + \sup_{z\in [y_k^i, \xi_i]}  \|(A^*q_k)'(z)\|_2 \|\xi_i -y_k^i\|_2. 
\end{align}
Then, for all $z\in [y_k^i, \xi_i]$, using the fact that $(A^*q_k)'(x_k^i)=0$, we get
\begin{equation*}
 \|(A^*q_k)'(z)\|_2 \leq R\kappa_{hess} \|z-x_k^i\|_2 \leq 2\delta R\kappa_{\hess}.
\end{equation*}
Hence, we have $|A^*q_k(\xi_i)| \leq 1 + 2\delta R\kappa_{\hess} \|\xi_i -y_k^i\|_2 \leq 1 + 2\delta R\kappa_{\hess} \dist(\Omega_k,\xi)$.
To conclude, we use Proposition~\ref{prop:supvsq_raw} again. 
\end{proof}

The last assertion takes full advantage of Assumption~\ref{ass:dualCert} and the fact that the function $|A^*q^\star|$ is uniformly concave around its maximizers. It allows to transfer bounds from $\Vert q_k-q^\star\Vert_2$ to bounds on the distance from $X_k$ to $\xi$.

\begin{prop} \label{prop:qVsRegime}
%Assume that $\norm{q_k-q^\star}_2 \leq \min\left(\frac{R\gamma\tau_0^2}{4\kappa}, \frac{R\gamma}{\kappa_{\hess}}, \frac{R\gamma}{\kappa_\nabla} \sqrt{4- \gamma \tau_0^2}\right)$, 
Define $c_q=\gamma \min\left(\frac{\tau_0^2}{2\kappa},\frac{\tau_0}{\kappa_{\nabla}},\frac{1}{\kappa_{\hess}}\right)$ and assume that $\norm{q_k-q^\star}_2 < c_q$, then
\[\dist(\xi,X_k) \le \frac{\kappa_\nabla}{\gamma} \norm{q_k-q^\star}_2. \]

Moreover, for each $i$, if $B_i$ is the ball or radius $\tau_0$ around $\xi_i$, then $X_k$ contains at most one point in $B_i$ and $A^*q_k$ has the same sign as $A^*q^\star(\xi_i)$ in $B_i$.

%  \begin{enumerate}[(i)]
%   \item $X_k$ contains at most $s$ points which satisfy $\norm{x_k^i - \xi_i}_2 \leq \tau$.
%   \item If $\omega_\ell^k$ is a cell as in $(ii)$, we must necessarily have
%     \begin{align*}
%         \sup_{x \in \mathrm{vert} \, \omega_\ell^k} \abs{(A^*q_k)(x)} \geq 1- \kappa_{\hess}\diam(\omega_\ell^k)^2
%     \end{align*}
%   \item If the initial grid also obeys the assumption of Lemma \ref{lem:nonsparse}, $S$ will contain \emph{exactly} $s$ points, one in each $\tau$-neighborhood $U_\tau(x_i)$, and there must be
%   \begin{align}
%     \sup_{x \in \mathrm{vert} \, \omega_\ell^k} \abs{(A^*q_k)(x)} = 1 \label{eq:exceedsinvertex}
%   \end{align}
%   for each cell $\omega_\ell^k$ as in $(ii)$
%  \end{enumerate}
\end{prop}
\begin{proof}
Define $\tau=\frac{\kappa_\nabla}{\gamma}\norm{q_k-q^\star}$ and note that $\tau < \tau_0$.
    By Proposition \ref{prop:Lipschitz_A_Aprime}, we have for each $x\in \Omega$
    \begin{align*}
        \abs{(A^*q_k)(x)-(A^*q^\star)(x)}&\leq\norm{A^*(q_k-q^\star)}_\infty \leq  \kappa\norm{q_k-q^\star}_2 < \frac{\gamma\tau_0^2}{2} \\
        \norm{(A^*q_k)'(x)-(A^*q^\star)'(x)}_2 & \le  \norm{(A^*(q_k-q^\star))'}_\infty \le \kappa_\nabla \norm{q_k-q^\star}_2 =  \gamma\tau\\
        \norm{(A^*q_k)''(x)-(A^*q^\star)''(x)}_2 &\le \norm{(A^*(q_k-q^\star))''}_\infty \le \kappa_{\hess}  \norm{q_k-q^\star}_2 <\gamma. 
\end{align*}
The above inequalities together with Assumption~\ref{ass:dualCert} imply the following for all $1\leq i \leq s$: 
\begin{enumerate}[(i)]
 \item For $x$ with $\norm{x-\xi_i}_2 \leq \tau_0$, we have  
$
    \sgn(A^*q_k)(x) = \sgn(A^*q^\star)(x) =\sgn(A^*q^\star)(\xi_i).
$ \item  For $x$ with $\norm{x-\xi_i}_2 \leq \tau_0$, we have  
$
    (\abs{A^*q_k})''(x) \prec (\abs{A^*q^\star})''(x) + \gamma\id \prec 0.
$
\item For $x$ with $\norm{x-\xi_i}_2 \geq \tau_0$, we have
 $\abs{(A^*q_k)(x)}<\abs{(A^*q^\star)(x)}  + \frac{\gamma\tau_0^2}{2} \leq 1 - \frac{\gamma \tau_0^2}{2} + \frac{\gamma\tau_0^2}{2} =1.
$
 \item For $x$ with $\tau < \norm{x-\xi_i}_2 \leq \tau_0$ 
, we have
 $\norm{(A^*q_k)'(x)}_2\ge \norm{(A^*q^\star)'(x)}_2 - \gamma\tau > 0.$
\end{enumerate}

The estimate $\norm{(A^*q^\star)'(x)}_2 > \gamma\tau$ deserves a slightly more detailed justification than the others.  Define $w = x-\xi_i$ and $g(\theta) = \sprod{(A^*q)'(\xi_i+\theta w),w} $ for $\theta \in (0,1)$. We may apply the mean value theorem to conclude that
\begin{align*}
 g(1)-g(0) = g'(\hat{\theta}) =  \sprod{(A^*q)''(\xi_i+\hat{\theta}w) w,w}
\end{align*}
for some $\hat{\theta}\in (0,1)$. Since $g(0)=\sprod{(A^*q^\star)'(\xi_i),w}=\sprod{0,w} =0$, and $\sprod{(A^*q^\star)''(\xi_i+\hat{\theta}w) w,w} \leq -\gamma \norm{w}_2^2$,  due to $(\abs{A^*q^\star})'' \preccurlyeq - \gamma \id$ in $\{x\in \Omega, \|x-\xi_i\|_2\leq \tau_0\}$, we obtain
\begin{align*}
 \norm{(A^*q^\star)'(x)}_2 \geq \frac{1}{\norm{w}_2} \abs{\sprod{(A^*q^\star)'(x),w}} = \frac{\abs{g(1)}}{\norm{w}_2} \geq \frac{\gamma \norm{w}_2^2}{\norm{w}_2} > \gamma \tau,
\end{align*}
since $\norm{w}_2=\norm{x-\xi_i}_2 > \tau$ by assumption. The last estimate was the claim $(iv)$.

This implies a number of things. First, any local maximum of $\abs{A^*q_k}$ with $\abs{A^*q_k}\geq 1$ must lie within a distance of $\tau$ from the set $\xi$ (since for all other points, we have $\abs{A^*q_k} < 1$ -- via $(iii)$ -- or $(Aq_k)'\neq 0$ -- via $(iv)$). Since $\abs{A^*q_k}$ is locally concave on the $\tau_0$-neighborhoods of the $\xi_i$ -- this follows from $(ii)$ -- at most one local extremum furthermore exists in each such neighborhood. This is the claim.

\end{proof}

\subsection{Fixed grids estimates}
\label{sec::exchange::fixed:grids}
In this section,  we consider a fixed grid $\Omega_0$ and ask what we need to assume about it in order to guarantee that the set of local maxima of  $ \abs{A^*q_0(x)}$ is close to true support $\xi$. We express our result in terms of a geometrical property that we can control, the \emph{width} of the grid $\dist(\Omega_0,\Omega)$.

\begin{theo} \label{th:fixedGrid}
    Assume that 
    $\dist(\Omega_0,\Omega) \le \frac{c_q}{\rho\sqrt{\kappa_{\hess}}}$,
then
\begin{align*}
    \dist(\xi, X_0) &\leq \frac{\kappa_\nabla\sqrt{R\kappa_{\hess}}\rho}{2\gamma} \dist(\Omega_0,\Omega)\\
     \norm{q_0 - q^\star}_2 &\le \rho \sqrt{R\kappa_{\hess}}\dist(\Omega_0,\Omega) \\
            \inf (\mathcal P(\Omega_0)) &\leq \inf \eqref{eq:primal} +  \frac{R\kappa_{hess}\rho^2}{2L} \dist(\Omega_0,\Omega)^2
\end{align*}

\end{theo}

\begin{proof}
It is trivial that $\dist(\Omega_0,X_0) \le \dist(\Omega_0,\Omega)$. Applying Lemma~\ref{lem:SupVsq}, we immediately obtain the bound on $\norm{q_0-q^\star}_2$. By the same lemma, 
 \begin{align*}
    \inf (\mathcal P(\Omega_0)) &= \sup (\mathcal D(\Omega_0)) = -f^*(q_0) \leq -f^*(q^\star) + \frac{R\kappa_{hess}\rho^2}{2L} \dist(\Omega_0,X_0)^2 \\
    &= \sup \eqref{eq:dual} + \frac{R\kappa_{hess}\rho^2}{2L} \dist(\Omega_0,\Omega)^2 = \inf \eqref{eq:primal} + \frac{R\kappa_{hess}\rho^2}{2L} \dist(\Omega_0,\Omega)^2.
 \end{align*}
In order to obtain the first bound, remark that $\norm{q_0-q^\star}_2 \leq c_q$ and use Proposition~\ref{prop:qVsRegime}.
\end{proof}

\begin{rem}
\label{rem:dist_not_reflexive}
Note that Theorem~\ref{th:fixedGrid} allows to control $\dist(\xi, X_0)$ but not $\dist(X_0, \xi)$. Indeed each $x \in X_0$ is guaranteed to be close to a $\xi_i$, but not every $\xi_i$ needs to have a point in $X_0$ closeby. Note however that the bounds on the optimal value indicates that in this case the missed $\xi_i$ is not crucial to produce a good candidate for solving the primal problem. We will provide more insight on this, in the case of $f$ being strongly convex, in Section \ref{sec:nonConvex}.
\end{rem}

\subsection{Eventual linear convergence rate}
\label{sec::exchange:linear}
In this section, we provide a convergence rate for the iterative algorithm. As a follow-up to Remark~\ref{rem:dist_not_reflexive}, the proof of convergence relies on the fact that the distances $\dist(X_k, \xi)$ and $\dist( \xi,X_k)$ are equal. In order to ensure this fact, one has to wait for a finite number of iterations, this is exactly the purpose of the next proposition.

\begin{prop} \label{prop:Good_Localization}
Let $B_i=\{x\in \Omega, \|x - \xi_i \|_2<\tau_0 \}$. There exists a finite number of iterations $N$, such that for all $k\geq N$, $X_k$ has exactly $s$ points, one in each $B_i$. It follows that $dist(X_k,\xi)=dist(\xi,X_k)$. Moreover if $S_k$ is the set of active point of $\mathcal D(\Omega_k)$, that is 
\[ S_k=\{z \in \Omega_k \text{ s.t. } |A^* q_k(z)|=1\}, \] then $S_k\subset \cup_i B_i$ and for each $i$, $B_i \cap S_k \ne \emptyset$. 
\end{prop}
\begin{proof}

 We first prove that $B_i$ contains a point in $S_k$. To this end, define the set of measures $\calM_-=\{\mu \in \calM(\Omega), \exists i\in \{1,\hdots, s\}, \supp(\mu)\cap B_i = \emptyset\}$ and
    \begin{equation*}
	 J_{+} = \min_{\mu \in \calM_-} \norm{\mu}_{\calM} + f(A\mu).
    \end{equation*}
    By assumption \eqref{ass:dualCert}, $J_{+}>J^\star$. Since $(J(\mu_k))_{k\in \N}$ converges to $J(\mu^\star)$, there exists $k_2\in \N$ such that $\forall k\geq k_2$, $J(\mu_k)<J_{+}$. Hence $\mu_k$ must for each $1\leq i \leq s$ have points $z_{k}^i\in \Omega_k$ such that $\mu_k$ has non-zero mass at $z_{k}^i$. Consequently, $|A^*q_k(z_{k}^i)|=1$, hence, each $B_i$ contains \emph{at least} one point in $\Omega_k$ such that $|A^*q_k(z_{k}^i)|=1$.
   
   Notice that $q_k$ converges to $q^\star$ by Theorem \ref{thm:generic_convergence}. Hence there  a finite number of iterations $k_1$ such that $\Vert q_k-q^\star\Vert< c_q$ for all $k\ge k_1$.  
    By item $(iii)$ of the proof of Proposition~\ref{prop:qVsRegime}, $|A^*q_k| <1$ outside  $\cup _i B_i$, 
    and by item $(ii)$, $|A^*q_k|$ is strictly concave in each $B_i$. Hence each $B_i$ contains exactly one maximizer of $|A^*q_k|$ exceeding one.
    
\end{proof}

We now move on to analyzing our exchange approach. Before formulating the main result, let us introduce a term: \emph{$\delta$-regimes}.

\begin{defi}
    We say that the algorithm enters a \emph{$\delta$-regime} at iteration $k_\delta$ if for all $k \ge k_\delta$, we have $\dist(\xi,X_k) \le \delta$. In particular it means that only points with a distance at most $\delta$ from $\xi$ are added to the grid.
\end{defi}

\begin{lem} \label{lem:regimeChange}
Let $\bar \tau_0 = \frac{\kappa_\nabla}{\gamma}c_q$ and $A=2^{d+1}d^{d/2}\left( \frac{\rho \sqrt{R\kappa_{hess}}\kappa_{\nabla}}{\gamma}\right)^{3d}$. Let $N$ be as in Proposition~\ref{prop:Good_Localization}.
\begin{enumerate}
\item For any $\tau$, the algorithm enters a \emph{$\tau$-regime} after a finite number of iterations.
\item Assume that $N$ iterations have passed and that the algorithm is in a $\tau$-regime with $\tau \leq \bar \tau_0$. Then for every  $\alpha \in (0,1)$ it takes no more than $\left\lceil \frac{A}{\alpha^{2d}} \right\rceil  +1$ iterations to enter an $\alpha\tau$-regime.
\end{enumerate}

\end{lem}
\begin{proof} 
Note that for any $\delta \le \bar \tau_0$, if there exists $p\in \N$ such that 
\begin{equation}
\label{eq::entering:tau:regime}
\|q_k-q^\star\|_2 \leq \frac{\gamma}{\kappa_{\nabla}}\delta \quad \text{ for all }k\geq p,
\end{equation} we will enter an $\delta$-regime after iteration $p$ by applying Proposition \ref{prop:qVsRegime}.  

To prove $(1)$, note that we without loss of generality can assume that $\tau \le \bar \tau_0$ (since entering a $\tau$-regime means in particular entering a $\tau'$-regime for any $\tau'\ge \tau$.) Then , since $\norm{q_k-q^\star}_2$ tends to zero as $k$ goes to infinity,
\eqref{eq::entering:tau:regime} with $\delta =\tau$ is true after a finite number of iterations.

To prove $(2)$, we proceed as follows :
Proposition \ref{prop:Good_Localization} ensures that in each iteration, exactly one point is added in each ball $\{x\in \Omega, \|x-\xi_i\|_2\leq \tau\}$. Let $k_0$ be the actual iteration, a covering number argument \cite{vapnik2013nature} ensures, for any $\Delta$ that after $\delta_0 =\left\lceil 2d^{d/2}\left(\frac{\tau}{\Delta} \right)^d\right\rceil$ iterations, each point in $X_k$ needs to lie at a distance at most $\Delta$ from $\Omega_k$, i.e., $\dist(\Omega_{k},X_k)\leq \Delta$.

Now, if we choose $\Delta=\left(\frac{\gamma}{\kappa_\nabla\rho \sqrt{R\kappa_{hess}}}\right)^3\frac{\alpha^2\tau}{2}$, Lemma \ref{lem:SupVsq} together with Proposition \ref{prop:qVsRegime} imply
\begin{align*}
 \dist(\Omega_{k_0+\delta_0+1},\xi)\leq \dist(X_{k_0+\delta_0},\xi) \leq \frac{\kappa_\nabla}{\gamma}  \rho\sqrt{R\kappa_{hess}} \dist(\Omega_{k_0+\delta_0},X_{k_0+\delta_0}) \leq \left( \frac{\gamma\alpha}{\kappa_{\nabla}\rho} \right)^2\frac{\tau}{2R\kappa_{hess}}
\end{align*}

Since $\Omega_{k+1}\subset \Omega_k$ for all $k$, the distance $\dist(\Omega_k,\xi)$ is non-increasing. As a result $\dist(\Omega_{k},\xi)\leq \left( \frac{\gamma\alpha}{\kappa_{\nabla}\rho} \right)^2\frac{\tau}{2R\kappa_{hess}}$  for all $k\geq k_0+\delta_0+1$. Since we are in $\tau$-regime, we know that $\dist(X_k,\xi)\leq \tau$ and $\dist(\Omega_k,\xi)\leq \tau$. Hence we can apply Lemma \ref{lem:gridVsQ} to obtain that 
\begin{align*}
\|q_k-q^\star\|_2\leq \sqrt{2R\kappa_{\hess} \tau \cdot \dist(\Omega_k,\xi)} \rho \leq \frac{\gamma}{\kappa_\nabla}\alpha \tau.
\end{align*}
Then inequality \eqref{eq::entering:tau:regime} is satisfied with $\delta=\alpha\tau$ and the algorithm enters a $\alpha\tau$-regime.
\end{proof}

The main result will tell us how many iterations we need to enter a $\tau$-regime. 

\begin{theo} \label{th:MainExchange}
Let $\tau \leq \bar{\tau}_0\eqdef\frac{\kappa_\nabla}{R\gamma}c_q$ and $k_0$ be the iteration on which the algorithm enters a $\bar \tau_0$-regime. Then $k_0 < \infty$, and the algorithm will enter a $\tau$-regime after no more than $k_0 + k_\tau$ iterations, where
\begin{align*}
    k_\tau :=  \left\lceil e 2^{d+1}d^{d/2}\left( \frac{\rho \sqrt{R\kappa_{\hess}}\kappa_{\nabla}}{\gamma}\right)^{3d}+1\right\rceil  \left\lceil 2d \log\left( \frac{\bar \tau_0}{\tau}\right) \right\rceil.
\end{align*}
Additionally, we will have
\begin{align}
\norm{q_k - q_*}_2 &\leq \tau \sqrt{2R\kappa_{\mathrm{hess}}}\rho \nonumber \\
\inf \eqref{eq:discretePrimal} &\leq \inf \eqref{eq:primal} +  \frac{2R\kappa_{\hess}\rho^2}{L} \cdot \tau^2  \label{eq:cvrate_cost_primal}
    \end{align}
    for $k \geq k_0+k_\tau+1$.
In other words, the algorithm will eventually converge linearly.
\end{theo}

\begin{proof}
 The fact that $k_0< \infty$ is the first assertion of Lemma~\ref{lem:regimeChange}. As for the other part, we argue as follows: Fix $\alpha \in (0,1)$. Since we have entered a $\bar \tau_0$-regime at iteration $k_0$, Lemma \ref{lem:regimeChange} implies that it will take no more than  $\left\lceil \frac{A}{\alpha^{2d}} \right\rceil  +1$ additional iterations to enter a $\alpha \bar \tau_0$. Repeating this argument, we see that after no more than
 \begin{align*}
    n \cdot \left(\left\lceil \frac{A}{\alpha^{2d}} \right\rceil  +1\right)
 \end{align*}
 iterations, we will have entered a $\alpha^n \bar \tau_0$ regime. Choosing $\alpha = e^{-1/2d}$ and $n  = \lceil 2d \log\left( \bar{\tau}_0/ \tau \right) \rceil$, we obtain the first statement.
 
 The second statement immediately follows from Lemma \ref{lem:gridVsQ} (as in the proof of Theorem~\ref{th:fixedGrid}) and the fact that entering a $\tau$-regime exactly amounts to that $\dist(X_k,\xi)\leq \tau$ for all future $k$, and therefore in particular $\dist(\Omega_{k+1},\xi)\leq \tau$.
\end{proof}

The inequality \eqref{eq:cvrate_cost_primal} upper-bounds the cost function for the problem \eqref{eq:discretePrimal}. 
In practice, the numerical resolution of this problem is hard since $\Omega_k$ contains clusters of points and in practice it is beneficial to solve the simpler discrete problem 
\begin{align*}
 \widehat{\mu}_k=\argmin_{\mu \in \calM(X_k)} \|\mu\|_{\calM} + f(A\mu) \tag{$\calP(X_k)$} \label{eq:simplePrimal}
\end{align*}
For this measure, we also obtain an a posteriori estimate of the convergence rate.
\begin{prop} Define $\widehat{\mu}_k$ as the solution of \eqref{eq:simplePrimal}, if $\dist(X_k,\xi)\leq \tau$, we have
\begin{equation} \label{eq:XkOptVal}
 J(\widehat{\mu}_k)\leq J(\mu^\star) + \left(\norm{\alpha^\star}_1 \frac{\kappa_{\hess} \norm{q^\star}_2}{2}  + \frac{L}{2} \|\alpha^\star\|_1^2 \kappa_\nabla^2\right) \tau^2. 
\end{equation}
\end{prop}
\begin{proof}
For any $i$, denote $x^i_k$ a point in $X_k$ closest to $\xi_i$ and define $\tilde \mu_k = \sum_{i=1}^s \alpha_i^\star \delta_{x_k^i}$. We have $J(\widehat{\mu}_k)\leq J(\tilde \mu_k)$ and $\|\tilde \mu_k\|_\calM \le \|\mu^\star\|_\calM$. Furthermore, we have 
\begin{align*}
f(A\tilde \mu_k) & \leq f(A\mu^\star) + \langle \nabla f(A\mu^\star), A\tilde \mu_k - A\mu^\star\rangle + \frac{L}{2} \|A\tilde \mu - A\mu^\star\|_2^2.
\end{align*}
The last term in the inequality is dealt with the following estimate:
\begin{align*}
 \|A\tilde \mu - A\mu^\star\|_2\leq  \sum_{i=1}^s |\alpha_i^\star| \|A(x_k^i)-A(\xi_i)\|_2 
 \leq \sum_{i=1}^s |\alpha_i^\star| \kappa_\nabla\|x_k^i-\xi_i\|_2 \leq \|\alpha^\star\|_1 \kappa_\nabla\tau.
\end{align*}
As for the penultimate term, remember that $q^\star = - \nabla f(A\mu^\star)$. This implies
\begin{align*}
    \sprod{\nabla f(A\mu^\star), A \tilde{\mu}_k - A \mu^\star} = \sprod{A^*q^\star, \mu^\star - \tilde{\mu}_k} = \sum_{i=1}^s \alpha_i^\star \left((A^*q^\star)(\xi_i)-A^*q^\star(x_k^i) \right)
\end{align*}
By making a Taylor expansion of $A^*q^\star$ in each $\xi_i$, utilizing that the derivative vanishes there, and that $\norm{(A^*q^\star)''(x)} \leq \kappa_{\hess} \norm{q^\star}_2$ for each $x \in \Omega$, we see that $\abs{(A^*q^\star)(x_k^i)-(A^*q^\star)(\xi_i)} \leq \frac{\kappa_{\hess} \norm{q^\star}_2}{2}\norm{x_k^i-\xi_i}_2^2$ for each $i$. This yields
\begin{align*}
     \sprod{\nabla f(A\mu^\star), A \tilde{\mu}_k - A \mu^\star} \leq \norm{\alpha^\star}_1 \frac{\kappa_{\hess} \norm{q^\star}_2 \tau^2}{2}.
\end{align*}
Overall, we obtain 
\begin{align*}
J(\widehat{\mu}_k)\leq J(\tilde \mu_k) = \|\tilde \mu_k\|_\calM + f(A\tilde \mu_k) \leq J(\mu^\star) + \norm{\alpha^\star}_1 \frac{\kappa_{\hess} \norm{q^\star}_2 \tau^2}{2} + \frac{L}{2} \|\alpha^\star\|_1^2 \kappa_\nabla^2\tau^2.
\end{align*} 
\end{proof}

\section{Convergence of continuous methods} \label{sec:GradientDescent}

 \label{sec:nonConvex}

%The exchange method tends to run into numerical issues for high values of $k$. The reason is the following: as iterations go by, the sets $\Omega_k$ contain denser clusters of points around the points in $\xi$. Since $A$ is continuous, some vectors in the set $\{A(x), x\in \Omega_k\}$ get highly correlated, which leads to numerical issues. A simple idea to tackle this problem is to stop the exchange algorithm after a few iterations and start a gradient descent on the function
In this section, we study an alternative algorithm that consists of using nonlinear programming approaches to minimize the following finite dimensional problem:
\begin{align}
    G(\alpha,X) \eqdef J\left(\sum_{i=1}^p\alpha_i\delta_{x_i}\right) = \norm{\alpha}_1 +f\left(A\left(\sum_i \alpha_i \delta_{x_i} \right) \right), \label{eq:defG}
\end{align}
where $X=(x_1,\hdots, x_p)$.
This principle is similar to continuous methods in semi-infinite programming \cite{rettich98} and was proposed specifically for total variation minimization in \cite{boyd2017alternating,denoyelle2018sliding,traonmilin2018basins,chizat2018global}. By Proposition \ref{prop:Good_Localization}, we know that after a finite number of iterations, $X_k$ will contain exactly $s$ points located in a neighborhood of $\xi$. This motivates the following hybrid algorithm:
\begin{itemize}
 \item Launch the proposed exchange method until some criterion is met. This yields a grid $X^{(0)}=X_k$ and we let $p=|X_k|$.
 \item Find the solution of the finite convex program 
 \begin{align*}
\alpha^{(0)} = \min_{\alpha \in \R^p} G(\alpha, X^{(0)}).
 \end{align*}
 \item Use the following gradient descent:
 \begin{align}\label{eq:continuous_gradient}
(\alpha^{(t+1)}, X^{(t+1)})= (\alpha^{(t+1)}, X^{(t+1)}) - \tau \nabla G( \alpha^{(t)}, X^{(t)}),
 \end{align} 
 where $\tau$ is a suitably defined step-size (e.g. defined using Wolfe conditions).
\end{itemize}

We tackle the following question: does the gradient descent algorithm converge to the solution if initialized well enough? 

%When using this method, two questions arise: when to stop the exchange algorithm and does the gradient descent algorithm converge to the solution if initialized well enough? We answer those questions in the next paragraphs.

% \subsection{Stopping criterion for the exchange algorithm}
% \label{sec:stopping_criterion}
% 
% One of the major issues for the use of hybrid algorithms is the stopping criterion. Unless some prior assumption on the problem's solution is available (in our case, the values of $\tau_0$ and $\gamma$ in Assumption \ref{ass:dualCert}), there is no clear swapping rule. 
% Lemma \ref{lem:SupVsq} however shows that the quantity $\dist(\Omega_k, X_k)$ - which can be computed easily - provides strong certifications on the current solution $q_k$ in terms of violation of the constraints, cost function and distance to the minimizer $q^\star$. We therefore recommend the rule: run the exchange algorithm until $\dist(\Omega_k, X_k)\leq \eta$ for some parameter $\eta>0$.

\subsection{Existence of a basin of attraction}
This section is devoted to proving the existence of a basin of attraction of a descent method in $G$. Under two additional assumptions, we state our result in Proposition~\ref{prop:G}.

\begin{assumption} \label{ass:strongConv}
 The function $f$ is twice differentiable and $\Lambda$-strongly convex. 
 \end{assumption}

 The twice differentiability assumption is mostly due to convenience, but the strong convexity is crucial. The  second assumption is related to the structure of the support $\xi$ of the solution $\mu^\star$. 

\begin{assumption} \label{ass:TransMatrix}
For any $x,y \in \Omega$ denote $K(x,y)=\sum_{\ell} a_\ell (x)a_\ell (y)$.
The \emph{transition matrix}
\begin{align*}
     T(\xi)=\begin{bmatrix}
       [K(\xi_i,\xi_j)]_{i,j=1}^s &[\nabla_x K(\xi_i,\xi_j)^*]_{i,j=1}^s \\
        [\nabla_x K(\xi_i, \xi_j)]_{i,j=1}^s &[\nabla_x \nabla_y K(\xi_i,\xi_j)^*]_{i,j=1}^s
    \end{bmatrix} \in \R^{s+sd,s+sd} .
\end{align*}
is assumed to be positive definite, with a smallest eigenvalue larger than $\Gamma>0$. 
\end{assumption}
It is again possible to prove for many important operators $A$ that this assumption is satisfied if the set  $\xi$ is separated. See the references listed in the discussion about Assumption \ref{ass:dualCert}. The following proposition describes the links between minimizing $G$ and solving \eqref{eq:primal}.
% For $\mu^*$, remembering that $q^* = - \nabla f(A\mu^*)$ (gradient vanishes)
%  \begin{align*}
%  \frac{\partial^2 G}{\partial^2 \alpha} (\beta, \gamma) &=   f''(A\mu^*)(A(\xi) \beta,A(\xi)\gamma) \\ \frac{\partial^2 G}{\partial \alpha \partial x} (\beta,\updelta) &= f''(A\mu^*) (A(\xi) \beta, A'(\xi)D(\alpha)\updelta) - \sum_{i}(A^*q^*)'(\xi_i) \beta_i \updelta_i \\
%   \frac{\partial^2 G}{\partial^2 x} (\updelta, \upepsilon) &=   f''(A\mu^*)(A'(\xi)D(\alpha)\updelta ,A'(\xi_i)D(\alpha)\upepsilon) - \sum_i \alpha_i(A^*q^*)''(\xi_i)(\updelta_i, \upepsilon_i)
% \end{align*}

\begin{prop} \label{prop:G} 
Let $\mu^\star=\sum_{i=1}^s \alpha^\star_i \delta_{\xi_i} \neq 0$ be the solution of \eqref{eq:primal}. Under Assumption \ref{ass:strongConv} and  \ref{ass:TransMatrix}, $(\alpha^\star,\xi)$ is global minimum of $G$.  Additionally, $G$ is differentiable with a Lipschitz gradient and strongly convex in a neighborhood of $(\alpha^\star,\xi)$. 

Hence, there exists a basin of attraction around $(\alpha^\star,\xi)$ such that performing a gradient descent on $G$ will yield the solution of \eqref{eq:primal} at a linear rate. 
%Said basin includes the set of $(\alpha,X)$ with  $\norm{x_i - \xi_i}_2 \leq \tau_0$ for each $i$ and
% \begin{align*}
%  \Lambda(\Gamma- G_1(X)) - G_2(\alpha,X) \geq 0 \quad \text{ and } \quad \gamma \abs{\alpha_i} + \Lambda(\Gamma- G_1(X)) \alpha_i^2 - G_2(\alpha,X) \geq 0 \text{ for all $i$},
% \end{align*}
% where we defined
% \begin{align*}
%  G_1(X)& \eqdef\frac{s}{R^2} \max_i\norm{x_i-\xi_i}_2  \sqrt{\kappa^2+\kappa_\nabla^2} \sqrt{\kappa_{\hess}^2+\kappa_\nabla^2},\\
%  G_2(\alpha, X) &\eqdef 2\kappa_{\hess} \max_i\norm{x_i-\xi_i}_2 + \frac{(2s\kappa_\nabla+ \norm{\alpha}_1 \kappa_{\hess})L}{R}\left(\frac{\kappa}{R} \norm{\alpha-\alpha^\star}_1 +\frac{\kappa_\nabla}{R}\max_i\norm{x_i-\xi_i}_2\norm{\alpha^\star}_1\right).
% \end{align*}
\end{prop}

The rest of this section is devoted to the proof of Proposition~\ref{prop:G}. Let us begin by stating a simple auxiliary result.
\begin{lem} \label{lem:posDefProd}
    Let $U$ and $V$ be vector spaces and $C: V \to V$ be a linear operator with $C \succcurlyeq \lambda \id_V$ for a $\lambda \geq 0$. Then, for any $B:U \to V$
    \begin{align*}
      B^*CB  \succcurlyeq \lambda B^*B.
    \end{align*}
\end{lem}

\begin{proof}
  If $B^*CB-\lambda B^*B$ is positive semidefinite, we claim holds. Since for $v \in U$ arbitrary
  \begin{align*}
        \sprod{(B^*CB-\lambda B^*B)v,v} = \sprod{C(Bv),Bv} - \lambda \sprod{Bv,Bv} \geq \lambda \norm{Bv}_V^2 - \lambda \norm{Bv}_V^2 =0,
        \end{align*}
 the latter is the case.
\end{proof}
Let us introduce some notation that will be used in this section: for an $X=(x_1,\hdots,x_p) \in \Omega^p$ for some $p$, $A(X)$ denotes the matrix $[a_i(x_j)]$. Analogously, $A'(X)$ and $A''(X)$ denote the operators
\begin{align*}
 A'(X) : (\R^d)^p \to \R^m, (v_i)_{i=1}^p \mapsto  \left(\sum_{i=1}^p \partial_x a_j(x_i)v_i\right)_j, \quad A''(X): (\R^d \times \R^d)^p \to \R^m, (v_i,w_i)_{i=1}^p \mapsto \sum_{i=1}^p A''(x_i)[v_i,w_i]
\end{align*}
respectively. Note that for $q \in \R^m$ and $X \in \Omega^p$,
\begin{align*}
 A(X)^*q &= ((A^*q)(x_i))_{i=1}^p \eqdef (A^*q)(X) \in \R^p \\
 A'(X)^*q &= (\nabla (A^*q)(x_1), \dots, \nabla (A^*q)(x_p)) \in (\R^d)^p \\
 A''(X)^*q &= ((A^*q)''(x_1), \dots , (A^*q)''(x_p)) \in (\R^d \times \R^d)^p
\end{align*}
 We will also use the shorthands $\mu= \sum_{i} \alpha_i \delta_{x_i}$, $G_f(\alpha,X) = f(A\mu)$, and, for $\alpha \in \R^p$, $D(\alpha)$ denotes the operator
\begin{align*}
    D(\alpha): (\R^d)^p \to (\R^d)^p, (v_i)_{i=1}^p \mapsto (\alpha_i v_i)_{i=1}^p.
\end{align*}
 We  have
   \begin{align*}
        \frac{\partial G_f}{\partial \alpha}(\alpha,X) \beta &= \sprod{\nabla f(A\mu),A(X)\beta} \\ 
        \frac{\partial G_f}{\partial X} \updelta &= \sprod{\nabla f(A\mu),  A'(X)D(\alpha)\updelta},
        \end{align*}
     so that in points $(\alpha,X)$ with $\alpha_i \neq 0$ for all $i$, and in particular in a neighborhood of $(\alpha^\star,\xi)$, $G$ is differentiable and its gradient is given by :
        \begin{align}
        \label{eq:formula:gradient:G}
            \R^p \times (\R^p)^d \ni \nabla G(\alpha,X) =  \left(\sgn(\alpha) - (A^*q)(X), -D(\alpha)(A^*q)'(X)  \right), \quad \text{ with } q=-\nabla f(A\mu).
        \end{align}
  
        As for the second derivatives, we have 
        \begin{align*}
 \frac{\partial^2 G_f}{\partial^2 \alpha}(\alpha,X) [\beta, \gamma] &=   f''(A\mu)(A(X) \beta,A(X)\gamma) \\ 
 \frac{\partial^2 G_f}{\partial \alpha \partial X}(\alpha, X)[ \beta,\updelta] &= f''(A\mu) (A(X) \beta, A'(X)D(\alpha)\updelta) + \sprod{\nabla f(A\mu), A'(X)D(\beta) \updelta} \\
  \frac{\partial^2 G_f}{\partial^2 X} (\alpha,X)[\updelta, \upepsilon] &=   f''(A\mu)(A'(X)D(\alpha)\updelta ,A'(X)D(\alpha)\upepsilon) + \sprod{\nabla f(A\mu), A''(X)(D(\alpha) \updelta, \upepsilon)}.
\end{align*}

We may now prove our claims.
\begin{proof}[Proof \ref{prop:G}] First, let us note that due to the optimality conditions of \ref{eq:primal}, we know that 
\begin{align*}
    q^\star = - \nabla f(A\mu^\star).
\end{align*}
Letting $q=-\nabla f(A\mu)$, it is furthermore fruitful to decompose the Hessian of $G$ into two parts:
 \begin{align*}
    H_1(\alpha,X)&=  \begin{bmatrix}
                            A(X)^*f''(A\mu) A(X)  & A(X)^* f''(A\mu) A'(X)D(\alpha) \\
                            D(\alpha)^* A'(X)^*f''(A\mu)A(X) & D(\alpha)^*A'(X)^*f''(A\mu)A'(X)D(\alpha)
                        \end{bmatrix} \\
    H_2(\alpha,X) [(\beta,\updelta),(\gamma,\upepsilon)] &=  -\sum_{i=1}^s \beta_i (A^*q)'(x_i)\epsilon_i  + \gamma_i (A^*q)'(x_i)\updelta_i +\alpha_i(A^*q)''(x_i)[\updelta_i,\upepsilon_i] , 
 \end{align*}

Now, $\abs{A^*q^\star}$ has local maxima in the points $\xi_i$, so that $(A^*q^\star)'(\xi)=0$. In these points, we furthermore have that  $\sgn(\alpha_i^\star)= A^*q^\star(\xi_i)$, so that the gradient of $G$ given in \eqref{eq:formula:gradient:G} vanishes.

 To prove the rest, it is enough to show that the Hessian of $G_f$ is positive definite in a neighborhood around $(\alpha^\star,\xi)$. Let $(\alpha,X)$ be arbitrary.  $H_1$ is an operator of the form $M_1^*M_2(X)^*{\mathcal L} M_2(X)M_1$, with   ${\mathcal L}= f''(A\mu): \R^m \to \R^m$ and 
 \begin{align*}
    M_1= \begin{bmatrix}\id& 0 \\ 0 & D(\alpha)
                        \end{bmatrix} : \R^p \times (\R^d)^s \to \R^s \times (\R^d)^s, \quad  M_2(X)= \begin{bmatrix}A(X) & A'(X)
                        \end{bmatrix}: \R^s \times (\R^d)^s \to \R^m.
 \end{align*}
Due to the $\Lambda$-strong convexity of $f$, $\mathcal{L} \matgeq \Lambda \id$. We furthermore have
\begin{align*}
    M_1^*M_1 = \begin{bmatrix} \id & 0 \\ 0 & D(\alpha)^*D(\alpha) \end{bmatrix}
\end{align*}

Let us now turn to $M_2(X)^*M_2(X)$. If we define $M_2(\xi) = \begin{bmatrix}A(\xi) & A'(\xi)\end{bmatrix}$, we have
\begin{align*}
    M_2(\xi)^*M_2(\xi) = \begin{bmatrix} A(\xi)^*A(\xi)& A(\xi) A'(\xi)^*\\
                          A'(\xi)^*A(\xi) & A'(\xi)^*A'(\xi)^*
                         \end{bmatrix} = T(\xi) \succcurlyeq \Gamma \id
\end{align*}
by Assumption \eqref{ass:TransMatrix}. We however have
\begin{align*}
    \norm{M_2(X)^*M_2(X) -M_2(\xi)^*M_2(\xi)}_2 \leq \norm{M_2(X)^*}_2 \norm{M_2(X)-M_2(\xi)}_2 + \norm{M_2(X)-M_2(\xi)}_2 \norm{M_2(\xi)}_2.
\end{align*}
Now, by definition of $\kappa$ and $\kappa_\nabla$,
\begin{align*}
    \norm{M_2(\xi)^*}_2^2 = \sup_{\norm{q}_2\leq 1} \sum_{i=1}^s\norm{(A^*q)(\xi_i)}_2^2 + \norm{(A^*q)'(\xi_i)}_2^2 \leq s( \kappa^2 + \kappa_\nabla^2),
\end{align*}
and similarly for $ \norm{M_2(X)}_2 = \norm{M_2(X)^*}_2$. Also, we have, by \eqref{eq:kappa:taylor}:
\begin{align*}
    \norm{M_2(X)-M_2(\xi)}_2^2 
    &\leq  s(\kappa_\nabla^2+\kappa_{\hess}^2) \sup_i\norm{x_i-\xi_i}_2^2,
\end{align*}
so that all in all
\begin{align*}
    M_2^*M_2 \geq \Gamma - s \max_i\norm{x_i-\xi_i}_2  \sqrt{\kappa^2+\kappa_\nabla^2} \sqrt{\kappa_{\hess}^2+\kappa_\nabla^2}
\end{align*}
We may now apply Lemma \ref{lem:posDefProd} twice to conclude
\begin{align}
\label{eq:proofofG:H1:bound}
    H_1 &\succcurlyeq \Lambda\left(\Gamma - s \max_i\norm{x_i-\xi_i}_2  \sqrt{\kappa^2+\kappa_\nabla^2} \sqrt{\kappa_{\hess}^2+\kappa_\nabla^2} \right) \begin{bmatrix} \id & 0 \\ 0 & D(\alpha)^*D(\alpha) \end{bmatrix} \\ &=: \Lambda (\Gamma - G_1(X))    \begin{bmatrix} \id & 0 \\ 0 & D(\alpha)^*D(\alpha) \end{bmatrix},  \nonumber
\end{align}
where we defined
\begin{align*}
    G_1(X)&= s \max_{1 \leq i \leq s} \norm{x_i-\xi_i}_2 \sqrt{\kappa^2+\kappa_\nabla^2} \sqrt{\kappa_{\hess}^2 + \kappa_\nabla^2}
    \end{align*}
It remains to analyze $H_2$. Define the bilinear form 
\begin{align*}
\calH_2[(\beta,\updelta),(\gamma,\upepsilon)] := 
-\sum_{i=1}^s \beta_i (A^*q^\star)'(x_i)\updelta_i +\gamma_i (A^*q^\star)'(x_i)\upepsilon_i  + \alpha_i(A^* q^\star)''(x_i)[\updelta_i,\upepsilon_i].
\end{align*}
Then, if we define $w= \nabla f(A\mu) - \nabla f(A\mu^\star)=q^\star-q$, we have
\begin{align*}
   (H_2(\alpha,X) -\calH_2)[(\beta,\updelta),(\gamma,\upepsilon)] &=  \sum_{i=1}^s \beta_i (A^*w)'(x_i)\upepsilon_i  + \gamma_i (A^*w)'(x_i)\updelta_i + \alpha_i(A^* w)''(x_i)[\updelta_i,\upepsilon_i].
\end{align*}
This makes it evident that
\begin{align*}
    \norm{H_2(\alpha, X)- \calH_2}_{2\to 2} \leq  2s\norm{(A^\star w)'}_\infty+
    \norm{\alpha}_1 \norm{(A^\star w)''}_\infty
 \le 
    (2s\kappa_\nabla+ \norm{\alpha}_1 \kappa_{\hess}) \norm{w}_2.
\end{align*}
The $L$-Lipschitz gradient of $f$ proves that $\norm{w}_2\leq L \norm{A\mu-A\mu_*}_2$. Using \eqref{eq:kappa:taylor} yields directly:
\begin{align*}
    \norm{A\mu-A\mu^\star}_2 &\leq \norm{A(X)(\alpha-\alpha^\star)}_2 + \norm{A(X)\alpha^\star-A(\xi)\alpha^\star}_2 \\
    &\leq \kappa \norm{\alpha-\alpha^\star}_1 +\kappa_\nabla \max_i\norm{x_i-\xi_i}_2\norm{\alpha^*}_1.
\end{align*}
We still need to bound $\calH_2$.  First remember that Assumption \ref{ass:dualCert} asserts that for each $i$, $\sgn \alpha_i^\star (A^*q^\star)'' \matleq -\gamma \id$ and $\sgn(\alpha_i)=\sgn(\alpha_i^\star)$ in the ball of radius $\tau_0$ around $\xi_i$. Consequently, if $(\alpha,X)$ is chosen so that for each $i$, $\norm{x_i-\xi_i}_2\leq \tau_0$  we get $-\alpha_iA^*q^\star(x_i) \matgeq \abs{\alpha_i} \gamma \id$, and
\[\calH_2[(\beta,\updelta),(\beta,\updelta)] \geq \sum_{i=1}^s \gamma \abs{\alpha_i} \norm{\updelta_i}_2^2 -2\beta_i (A^*q^\star)'(x_i)\updelta_i\]
By definition of $\kappa_{\hess}$, we can further estimate
\[\Vert (A^*q^\star)'(x_i)-(A^*q^\star)'(\xi_i)\Vert_2 \le \kappa_{\hess} \norm{q^\star}{\norm{x_i-\xi_i}}_2.\]
Using $(A^*q^\star)'(\xi_i)=0$, we  obtain
\begin{align*}
    \calH_2 \succcurlyeq  \gamma\begin{bmatrix}0 & 0 \\ 0 & D(\abs{\alpha}) \end{bmatrix}- 2\norm{q^\star}_2\kappa_{\hess}s  \max_i\norm{x_i-\xi_i}_2 \id.
\end{align*}
If we now define
\begin{align*}
    G_2(\alpha,X)= \kappa \norm{\alpha-\alpha^\star}_1+ \kappa_\nabla \norm{x_i-\xi_i}_2 \norm{\alpha^\star}_1 + 2 \norm{q^\star} \kappa_{\hess}s \max_i \norm{x_i-\xi_i}_2,
\end{align*}
we obtain
\begin{align*}
 H_2(\alpha,X) &\succcurlyeq   \begin{bmatrix}-G_2(\alpha,X)\id & 0 \\ 0 & \gamma D(\abs{\alpha})-G_2(\alpha,X) \id \end{bmatrix}
\end{align*}
Further utilizing the definition of $G_1$ and \eqref{eq:proofofG:H1:bound}, we arrive at
\begin{align*}
    H_1(\alpha,X) + H_2(\alpha,X) \geq \begin{bmatrix}\left( \Lambda(\Gamma- G_1(X)) - G_2(\alpha,X)\right)\id & 0 \\ 0 & \gamma D(\abs{\alpha}) + \Lambda(\Gamma- G_1(X)) D(\alpha^2) - G_2(\alpha,X)\id\end{bmatrix}.
\end{align*}
Since $G_1(X), G_2(\alpha,X) \to 0$ for $\alpha \to \alpha^\star$ and $X \to \xi$, we obtain the claim.
\end{proof}

\subsection{Eventually entering the basin of attraction}
% It is well known that if we start a gradient descent around $(\alpha^\star,\xi)$ in a set in which $G$ is strictly convex, a gradient descent (with correct stepsize) will converge linearly towards $(\alpha^\star,\xi)$. 
The following proposition shows that $(\tilde{\alpha},X_k)$ defined as the amplitudes and positions of the Dirac-components of the solution $\widehat{\mu}$ of \eqref{eq:simplePrimal},
 $(\tilde{\alpha}, X_k)$ will lie in the basin described by Proposition \ref{prop:G}. This result is stated in Corollary~\ref{cor:alphaClose}, the rest of this section is dedicated to proving it.
 \begin{prop} \label{prop:alphaClose}
    Assume that Assumptions \ref{ass:strongConv}  and \ref{ass:TransMatrix} are true. Consider an $s$-sparse measure
    \begin{align*}
        \tilde{\mu} = \sum_{\ell=1}^s \tilde{\alpha}_\ell \delta_{\tilde{x}_\ell}
    \end{align*}
for some $\tilde{\alpha} \in \R^s$ and $(\tilde{x}_\ell)_{\ell=1\dots s}$ pairwise different points of $\Omega$. We then have
\begin{align*}
    \norm{\tilde{\alpha}-\alpha^\star}_2 \leq \frac{1}{\sqrt{\Gamma}} \left(\kappa_\nabla \norm{\tilde{\mu}}_\calM \sup_{\substack{1 \leq \ell \leq s}} \norm{\xi_\ell-\tilde{x}_\ell}_2 + \sqrt{\frac{2}{\Lambda}\left(J(\tilde{\mu})-J(\mu^\star)\right)}\right).
\end{align*}
 \end{prop}

\begin{proof}
    Let $A(\xi)^\dagger$ be the Moore-Penrose inverse of $A(\xi) =[A(\xi_1), \dots, A(\xi_s)]$. Due to Assumption \ref{ass:TransMatrix}, $A(\xi)^\dagger$ has full rank and has an operator norm no larger than $\Gamma^{-1/2}$. Since
   \begin{align*}
        \tilde{\alpha} = \alpha^\star + A(\xi)^\dagger ( A(\xi) \tilde{\alpha} - A \tilde{\mu}) + A(\xi)^\dagger(A\tilde{\mu} - A(\xi)\alpha^\star),
    \end{align*} 
bounds on $A(\xi) \tilde{\alpha} - A \tilde{\mu}$ and $A\tilde\mu - A(\xi)\alpha^\star$ will therefore transform to a bound on $\tilde{\alpha}-\alpha^\star$.

Let us begin with the former. We have
\begin{align*}
    \norm{A(\xi) \tilde{\alpha}-A \tilde{\mu}}_2 \leq \sum_{\ell=1}^s \abs{\alpha_\ell} \norm{A(\xi_\ell) - A(\tilde{x}_\ell)} \leq \sum_{\ell=1}^s \kappa_\nabla \abs{\tilde{\alpha}_\ell} \norm{\xi_\ell-\tilde{x}_\ell}_2  = \kappa_\nabla \norm{\tilde{\alpha}}_1 \sup_{\substack{1 \leq \ell \leq s \\ \tilde{\alpha}_\ell \neq 0}} \norm{\xi_\ell-\tilde{x}_\ell}_2,
\end{align*}
where we used the Cauchy-Schwarz inequality in the last step.

To bound the latter, recall that $\Lambda$-strong convexity of $f$ means that
\begin{align} \label{eq:variation}
    f(A\tilde{\mu}) \geq f(A\mu^\star) + \sprod{\nabla f(A\mu^\star), A\tilde{\mu}-A \mu^\star} + \frac{\Lambda}{2} \norm{A\tilde{\mu}-A\mu^\star}_2^2.
\end{align}
The optimality conditions for \eqref{eq:primal} tell us that $q^\star= - \nabla f(A\mu^\star)$, and hence
\begin{align*}
     \sprod{\nabla f(A\mu^\star), A\tilde{\mu}-A \mu^\star} = \sprod{A^*q^\star, \mu^\star-\tilde{\mu})}= \sum_{\ell=1}^s \alpha_\ell^\star (A^* q^\star)(\xi_\ell) - \tilde{\alpha}_\ell A^*q^\star(\tilde{x}_\ell) \geq \norm{\alpha^\star}_1 - \norm{\tilde{\alpha}}_1,
\end{align*}
where we in the last step used that $\norm{A^*q^\star}_\infty \leq 1 $. Plugging the above inequality in \eqref{eq:variation} yields
\begin{align*}
    \frac{\Lambda}{2} \norm{A\tilde{\mu}-A\mu^\star}_2^2 \leq J(\tilde{\mu})- J(\mu^\star).
\end{align*}
The claim follows.
\end{proof}
\begin{cor}
\label{cor:alphaClose}
 By Proposition~\ref{prop:Good_Localization}, if $k$ is large enough  then $X_k$ contains exactly $s$ points. In this case, let $\widehat{\mu}_k=\sum_{i=1}^s \widehat{ \alpha}_i \delta_{\hat{x}_i^k}$ be the solution of \eqref{eq:simplePrimal}. Applying Proposition~\ref{prop:alphaClose}, recalling that $\max_{i} \norm{\xi_i-\hat{x}_i^k}_2 \le \dist(X_k,\xi)$ and using the bound \eqref{eq:XkOptVal}, we obtain :
    \begin{align*}
    \norm{\widehat{\alpha}-\alpha^\star}_2 \leq \frac{\dist(X_k,\xi)}{\sqrt{\Gamma}} \left(\kappa_\nabla \norm{\widehat{\mu}_k}_\calM  + \sqrt{\frac{2}{\Lambda}\left(\norm{\alpha^\star}_1 \frac{\kappa_{\hess} \norm{q^\star}_2}{2}  + \frac{L}{2} \|\alpha^\star\|_1^2 \kappa_\nabla^2\right)}\right).
\end{align*}
Since  $\dist(X_k,\xi)$ is guaranteed to eventually converge to zero by Theorem \ref{th:MainExchange} and $\Vert \widehat{\mu}_k \Vert_\calM$ are bounded 
( e.g. by lower boundedness of $f$ and upper boundedness of $J(\widehat\mu_k)$)
, $(\widehat{\alpha},X_k)$ will eventually lie in the basin of attraction of $G$.   
    
 \end{cor}

\section{Description of the hybrid approach} \label{sec:hybrid}

To conclude this paper, we propose a method alternating between an exchange step and a continuous gradient descent. It is detailed in Algorithm \ref{alg:hybrid}. The idea is, after each iteration of an exchange algorithm, to start a gradient descent of $G$ initialized at the solution $\widehat{\mu}_k$ of \eqref{eq:simplePrimal}. If this gradient descent converges to a measure $\bar{\mu}_k$, we can subsequently test if it is an optimal point by checking if $\bar{q}_k= -\nabla f(A\bar{\mu}_k)$ fulfills the stopping criterion $\norm{A^*\bar{q}_k}_\infty \leq 1+\epsilon$, where $\epsilon$ is a user defined stopping criterion (the latter is justified by Proposition \ref{prop:supvsq_raw}). If so, we may output $\bar{\mu}_k$, and if not, we may instead continue our exchange algorithm, possibly after adding also the support points of $\bar{\mu}_k$. Its behavior is described in the following theorem. 
%It can be seen as a direct corollary of Theorem \ref{thm:generic_convergence}, Theorem \ref{th:MainExchange} and Proposition \ref{prop:G}.

\begin{theo}[Convergence guarantees for the alternating method]
Algorithm \ref{alg:hybrid} comes with the following guarantees:
\begin{enumerate}
 \item \emph{(Theorem \ref{thm:generic_convergence})} Under Assumptions \ref{ass:f}, \ref{ass:A} and \ref{ass:grid}, it is guaranteed to stop after a finite number of iterations for any stopping criterion $\epsilon>0$. 
 \item \emph{(Theorem \ref{th:MainExchange})} If in addition Assumptions \ref{ass:regf} and \ref{ass:dualCert} are satisfied, then the algorithm eventually converges linearly: $k \geq N+k_\tau$ with $k_\tau \leqsim \log(\tau^{-1})$, we have $\dist(\Omega_{k},\xi)\leq \tau$. 
 \item \emph{(Proposition \ref{prop:G}, Theorem \ref{th:MainExchange} and Proposition \ref{prop:alphaClose})} If in addition Assumptions \ref{ass:strongConv} and \ref{ass:TransMatrix} are satisfied, then - for large enough $k$ - the low complexity gradient descent \eqref{eq:continuous_gradient}  method converges linearly : $\|(\alpha^{(t)},X^{(t)})-(\alpha^\star,\xi)\|_2\leq c^t\|(\alpha^{(0)},X^{(0)})-(\alpha^\star,\xi)\|_2$ for some $0\leq c<1$. 
\end{enumerate}
\end{theo}

Overall, this method has many desirable properties: the continuous method should be used whenever the exchange method reaches its basin of attraction since its per iteration cost is much cheaper. However, it is unclear in general that this basin even exists. In that case, the exchange method should be preferred since it eventually converges linearly under quite mild assumptions. The proposed algorithmic scheme somehow captures the best of all methods. Let us notice that it is very similar in spirit to the sliding Frank-Wolfe algorithm proposed in \cite{denoyelle2018sliding}, apart from the fact that we suggest adding \emph{all} the points $X_k$ violating the constraints, while the single most violating point is added in \cite{denoyelle2018sliding}. We believe that the proposed analysis sheds some light on the good numerical performance of this method.

Arguably the most complicated step in this algorithm is to evaluate $X_k$, the set of local maximizers of $A^*q_k$ exceeding $1$. This is an impossible task for an arbitrary function $A^*q_k$. However, a simple heuristic described in the next section provided rather satisfactory results for the measurement functions considered in this paper (trigonometric polynomials and Gaussian convolution).

Apart from this, let us outline that the subproblems in this algorithm are well suited for numerical resolution. In the exchange algorithm, we only solve the dual problems $\calD(\Omega_k)$ which are strongly convex. 
Hence first-order methods for instance come with guarantees of convergence to $q_k$ in $\ell^2$-norm.  
Recovering the masses $\hat \alpha_k$, solutions of $\calP(X_k)$ is also stable since $X_k$ (the local maximizers of $A^*q_k$) is typically a well separated set of low cardinality. The gradient descent (or alternative nonlinear programming approach) on $G(\alpha,X)$ is performed over a low dimensional set. If the convergence is not satisfactory (e.g. the norm of $\nabla G$ doesn't decay fast enough), it can be stopped, and we can switch back to the exchange algorithm. 

\begin{algorithm}[htbp]
\caption{Alternating method \label{alg:hybrid}}
\begin{algorithmic}[1]
\State \textbf{Input:} Operator $A$, data fitting term $f$, stopping criterion $\epsilon>0$.
\State Set $q_0=0$, $k=0$, $\Omega_0=\emptyset$
\State Evaluate $X_0$ in \ref{eq:defXk} and $\|A^* q_0\|_\infty$ \algorithmiccomment{Nonconvex - Possibly complicated}
\While{$\|A^* q_k\|_\infty > 1+\epsilon$}
    \State $k=k+1$
    \State Set $\Omega_{k} = \Omega_{k-1}\cup X_k$
    \State Solve $\calD(\Omega_{k})$ to retrieve $q_k$ \algorithmiccomment{Convex - Stable}
    \State Evaluate $X_k$ in \ref{eq:defXk} and $\|A^*q_k\|_\infty$ \algorithmiccomment{Nonconvex - Possibly complicated}
    \State Solve $\calP(X_k)$ to retrieve $\widehat{\alpha}_k$ \algorithmiccomment{Convex - Low dimensional}
    \State Gradient descent on $G(\alpha,X)$ in \eqref{eq:defG} starting from $(\widehat{\alpha}_k,X_k)$ \algorithmiccomment{Nonconvex - Low dimensional}
 %   \If{Not converged}
 %     \State Continue
 %   \ElsIf{Converged to $(\bar \alpha_k, \bar X_k)$}
 \If{ Gradient descent converged to $(\bar \alpha_k, \bar X_k)$}
      \State Define $q_k=-\nabla f(A\bar \mu_k)$ with $\bar \mu_k=\sum_{i=1}^{|X_k|} \bar \alpha_k(i) \delta_{\bar X_k(i)}$
      \State Evaluate $X_k$ in \ref{eq:defXk} and $\|A^*q_k\|_\infty$ \algorithmiccomment{Nonconvex - Possibly complicated}
      \State (Optional) Define $\Omega_{k}=\Omega_k\cup{\bar X_k}$.
    \EndIf
\EndWhile
\State Solve $\calP(X_k)$ to retrieve $\alpha_k$  \algorithmiccomment{Convex - Low dimensional}
\State \textbf{Output:} $\mu_k= \sum_{i=1}^{|X_k|} \alpha_k(i) \delta_{X_k(i)}$ and $q_k=-\nabla f(A \mu_k)$.
\end{algorithmic}
\end{algorithm}

\section{Numerical Experiments} \label{sec:numerics}

To test our theory, we have implemented our algorithm in MATLAB. Before displaying the results of the experiments, let us discuss a few key steps in the implementation. In the entire section, we assume that $\Omega = [0,1]^d$ for $d=1$ or $2$ for simplicity. Note that this is no true restriction: we can always by scaling and translation ensure that $\Omega \sse [0,1]^d$, and trivially extend the measurement functions by $0$ to the entirety of $[0,1]^d$.

% \paragraph{Keeping Track of $\Omega_k$}
% Assuming that $\Omega$ is bounded, we can by linear transformations always ensure that $\Omega$ is included in the unit cube $[0,1]^d$. We start with $\Omega_0$ as a uniform, rectangular grid of $[0,1]^d$  with cells of length $2^{-L}$ for some integer $L$. We then choose to refine a cell $\omega$ by subdividing it into $2^d$ new rectangles. This ensures that assumption \ref{ass:refinement} is fulfilled. More importantly, it conveniently allows us to use a $2^d$-ic (i.e. dyadic for $d=1$, quadric for $d=2$, and so on) tree structure to keep track of the active points. The nodes of the tree corresponds to different possible active subsquares in the grid. If the node is a leaf, the subsquare is actually active, and if not, its children correspond to the subsquares its corresponding cells has been divided into (see Figure \ref{fig:graphVsGrid} ). Given such a tree, it is easy to generate the points in $X_k$ (we simply loop over all leaves and collect their corners), which can then be used to generate the matrix $A_k$ for solving the finite-dimensional problems in the exchange algorithm. The refinement process is also particularly simple using this data structure: refining a cell simply means that the adding $2^d$ children to the corresponding leaf in the tree.

\paragraph{Evaluating $X_k$}
Each iteration of the exchange algorithm requires the exact calculation of the local maximizers of $A^*q_k$ exceeding $1$. This is, in general, an impossible task. We resort to the following heuristic method: Given a $q_k$, we first evaluate $|A^*q_k|$ on a fixed rectangular grid $G=((n)^{-1}[0, \dots, n])^d$, and determine all of the discrete peaks, i.e. points in which $\set{A^*q_k}$ is larger than all of its neighbors in the grid, and where $A^*q_k$ exceeds $1-\epsilon_1$ for a threshold $\epsilon_1>0$. Next, we start a gradient descent in each of these points, stopping them once $\norm{(A^*q_k)'}_2$ is lower than another threshold. Since it is possible that several of these gradient descents land in the same point $x$, we subsequently check if the set contains sets of points which are too close to each other - if this is the case, we discard all but one of them in such a group. We finally remove any point in which $\abs{A^*q_k}$ is not larger than $1-\epsilon_2$, for a small $\epsilon_2>0$.

 \paragraph{Solving the Discrete Problems} We have chosen to solve the problems \eqref{eq:discreteDual} and $(\calP(X_k))$ using an accelerated proximal gradient descent \cite{nesterov2013gradient}.
 %, a package for solving convex programs \cite{cvx1,cvx2}. The internal solver in CVX is an interior point method. Other first-order approaches could be used with a lower precision. 
 %The exchange algorithm is ran until $\dist(\Omega_k, X_k) \leq \epsilon_3$, as explained in  paragraph \ref{sec:stopping_criterion}.

 %Note that this inevitably leads to non-sparse solution vectors (entries in the vector that should be equal to zero are non-zero, but small). In practice, this problem is solved 
 
 %As was discussed at the end of section \ref{sec:GradientDescent}, this problem can however be dealt with by simply thresholding the vector $\mu_k$ before carrying out the merging procedure outlined above.

\subsection{Example 1: Super-resolution from Fourier measurements in 1D.}
We start by testing our algorithm on a popular instance of problem \eqref{eq:primal}: super-resolution of a measure $\mu \in \calM(0,1)$ from finitely many of its Fourier moments 
\begin{align*}
   y_k= \sprod{a_k, \mu} = \int_{0}^1 \exp(-ikx)d\mu,  -m/2\leq k \leq m/2-1.
\end{align*}
We use  a quadratic data fidelity term $f(z) = \frac{L}{2} \norm{z-y}_2^2$. This example is well studied by the signal processing community \cite{tang2013compressed,candes2014towards,duval2015exact,poon2018support}.
%, and it is well established that the regularity Assumption \ref{ass:dualCert} is satisfied when $y = A\mu_0$ for a sparse atomic measure, and $L$ is large.

We chose $m$ to be equal to $30$, and a vector $y$ generated as $A\mu_0$, where $\mu_0$ is chosen at random as a $5$-sparse atomic measure with amplitudes close to $1$ or $-1$. The positions of the Dirac masses were chosen as a small random perturbation from a uniform grid. The initial grid $\Omega_0$ was chosen as a uniform grid with $8$ points, i.e. $[0, \frac{1}{8}, \dots, \frac{7}{8}]$. We made $100$ experiments, with $20$ iterations of the exchange algorithm. The evolution of $\mu_k$ and $q_k$ for the first iterations for a typical iteration is displayed in Figure \ref{fig:Fourier_evol1}. We see that after already $8$ iterations, $A^*q_k$ appears to be very close to $A^*q^\star$. Before this iteration, the algorithm 'chooses' to add points relatively uniformly to the grid, but after that, new points are only added close to $\xi$. This is further emphasized by Figure \ref{fig:Fourier_evol2}, in which $X_k$ is plotted for each iteration, along with size of $\Omega_k$.

\begin{figure}
 \centering
 \includegraphics[width=1\textwidth]{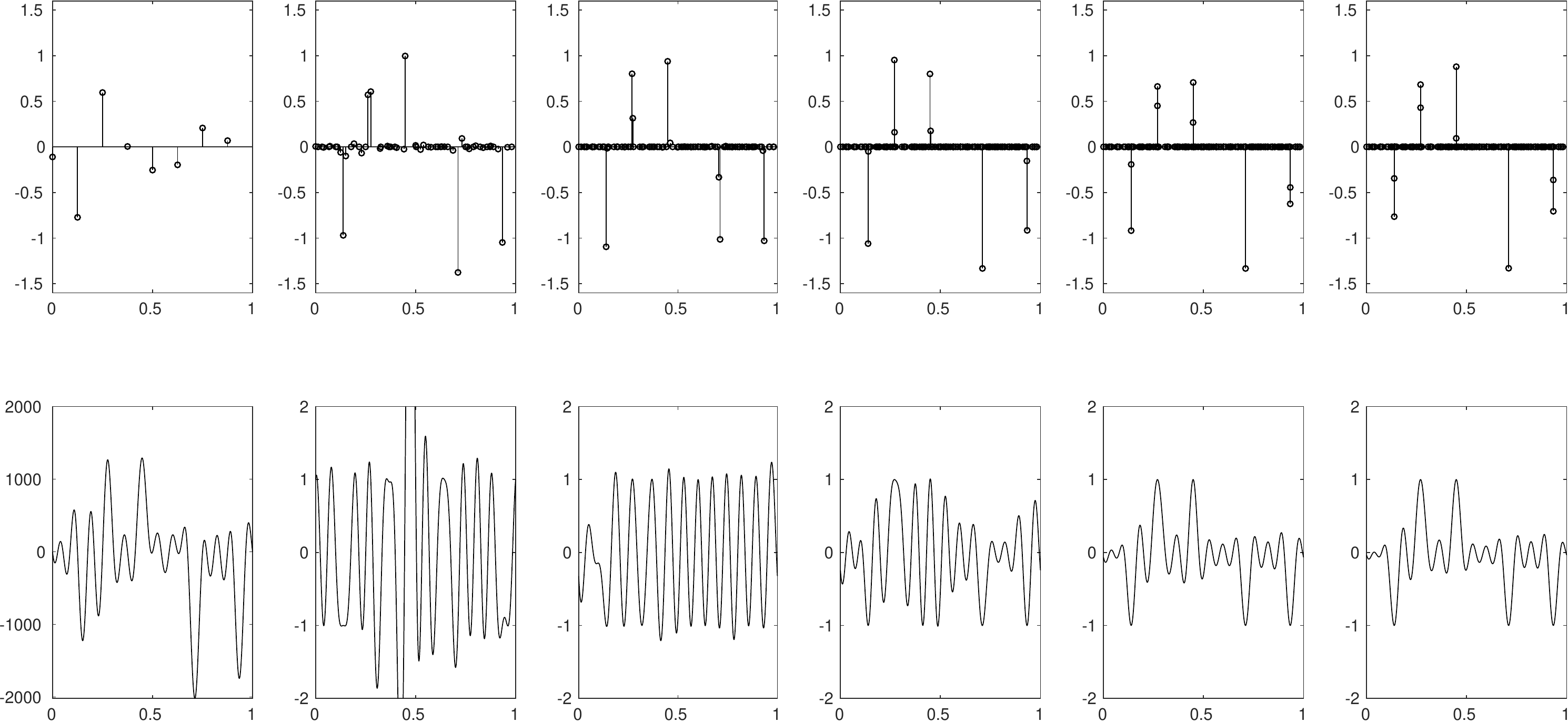}
\caption{Above: $\mu_k$ for $k=0,2,4,6,8,20$ along one run of the algorithm. Below: $A^*q_k$ for $k=0,2,4,6,8,20$ along the same run. Note that the range of the first plot is different from the others. \label{fig:Fourier_evol1} }
\end{figure}

\begin{figure}
 \centering
 \includegraphics[width=1\textwidth]{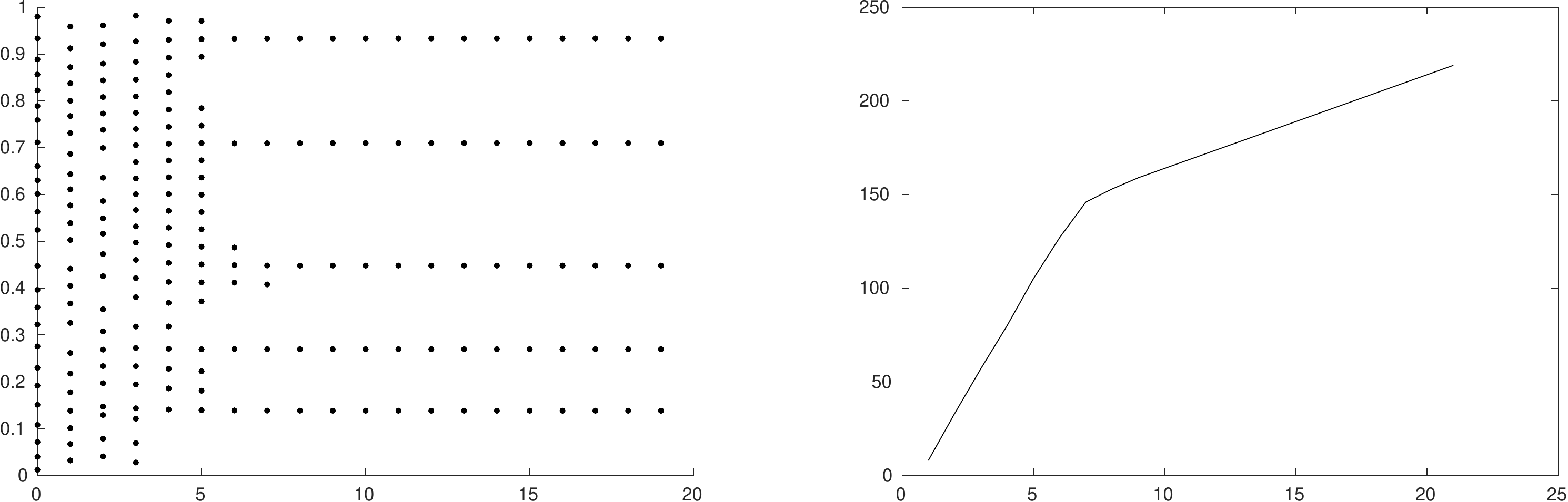}
 \caption{Left: The set $X_k$ of added points for each iteration along a run of the algorithm. Right: The total number of points in $\Omega_k$ along the same run.\label{fig:Fourier_evol2}}
\end{figure}

To track the success of the algorithm a bit more systematically, we chose to track the evolution of $\dist(\xi,X_k)$, $\dist(\Omega_k,X_k)$ and $\dist(\Omega_k,\xi)$. The median over the 100 iterations, along with confidence intervals covering all experiments but the top and bottom $5\%$ are plotted in Figures \ref{fig:distances}. We see that all of the quality measures seem to converge linearly to 0.

\begin{figure}
 \centering
 \includegraphics[width=1\textwidth]{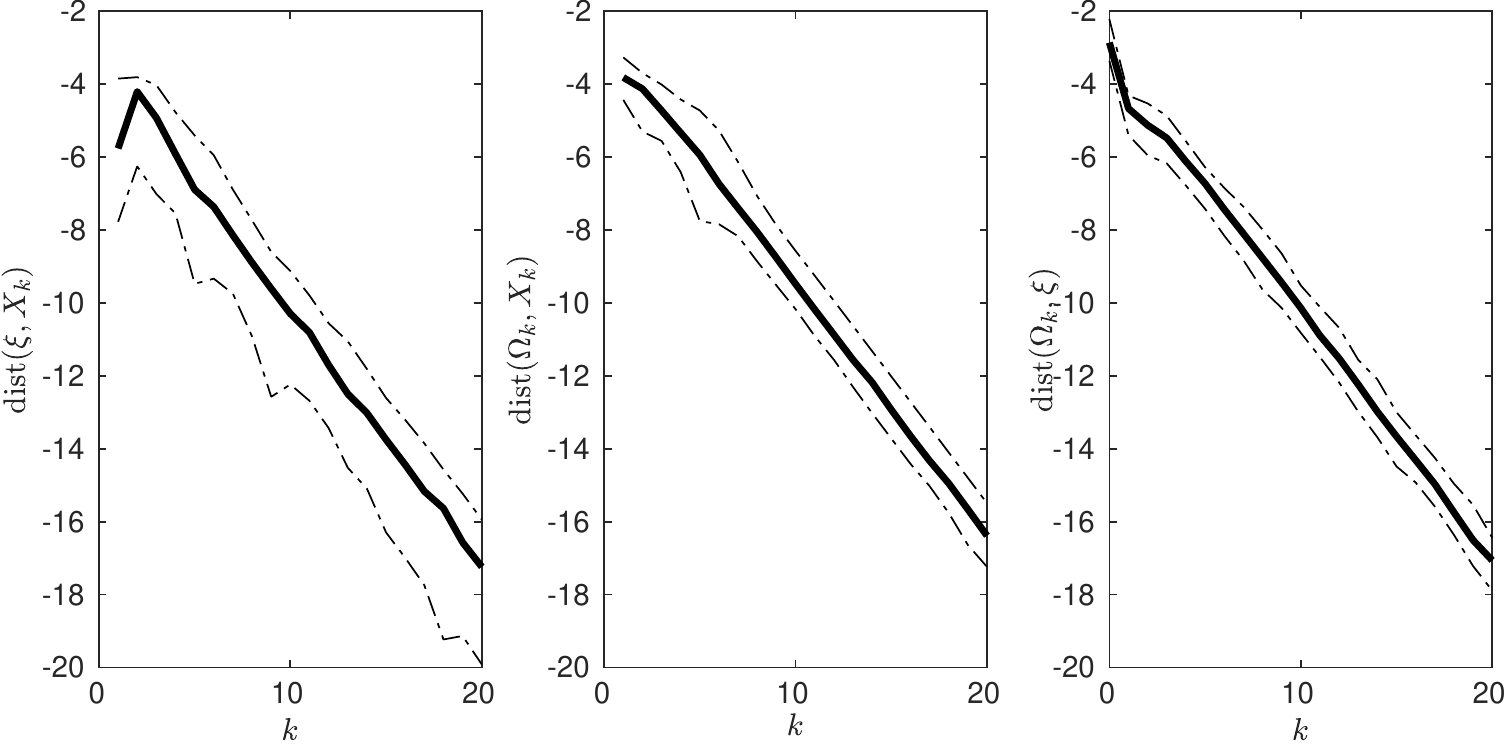}
 \caption{Logarithmic plot of $\dist(\xi,X_k)$, $\dist(\Omega_k,X_k)$ and $\dist(\Omega_k,\xi)$. Shown is the median value (oblique line) along with confidence intervals(dashed) covering all but the top and lower $5\%$. \label{fig:distances}}
\end{figure}

Finally, we performed the same analysis for the optimum gap $\min \eqref{eq:discretePrimal}-\min\eqref{eq:primal}$, the error $\norm{q_k-q^\star}_2$ and the sizes of the grids $\Omega_k$. ($\min\eqref{eq:primal}$ was in each case chosen as the lowest value of $\min \eqref{eq:discretePrimal}$ over all iterations $k$, and $q^\star$ as the corresponding dual solution). We see that the optimum gap seems to converge exponentially to $0$ right from the first iteration, wheras the error $\norm{q_k-q^\star}_2$ initially does not. The 'two-phase'-effect is also easy to spot: After about $5-6$ iterations, the algorithm switches from adding many points to adding only few points close to $\xi$. Interestingly, the plateau of the $q$-errors seems to be simultaneuos with the 'phase-transition'. 

\begin{figure}
\centering\includegraphics[width=1\textwidth]{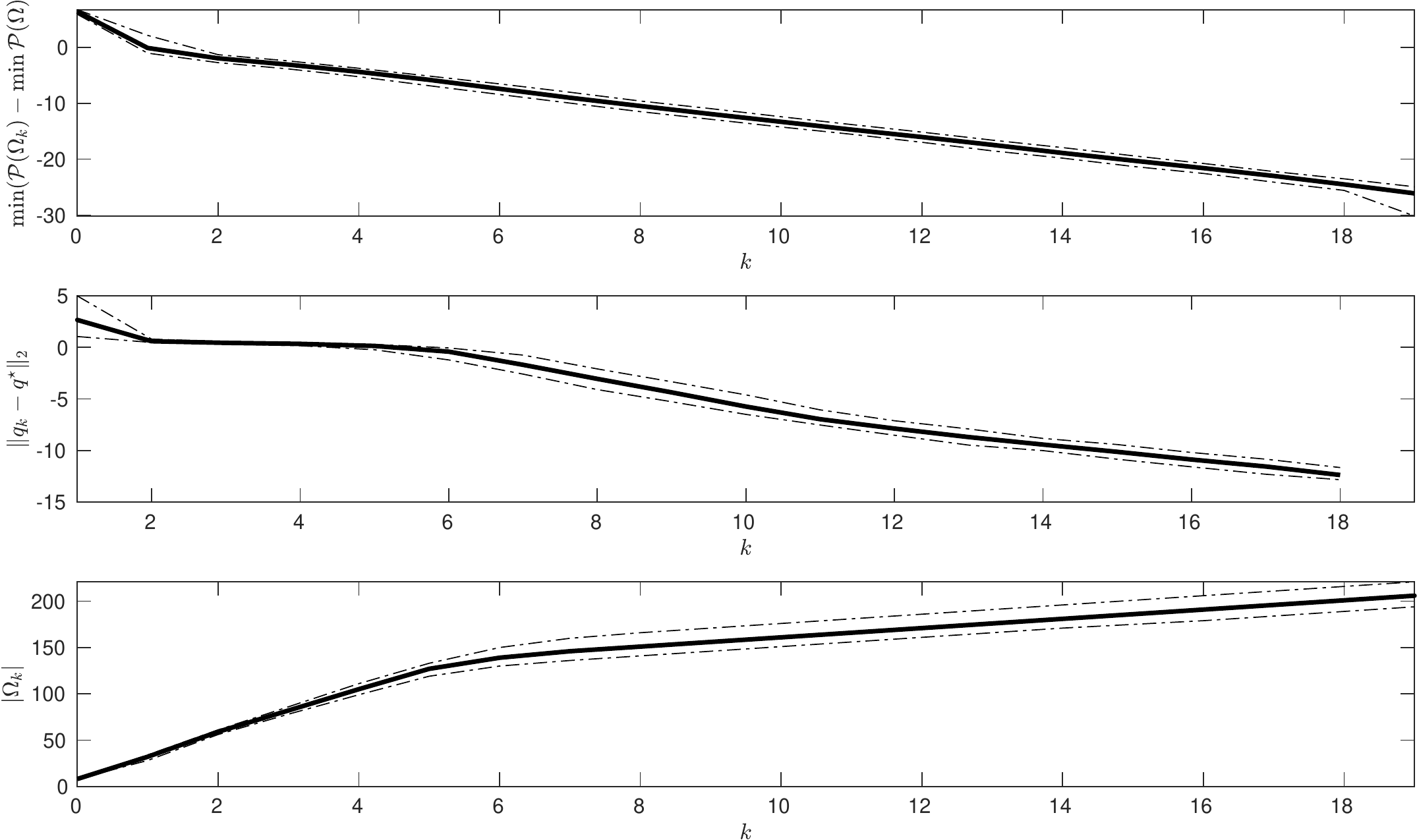}
 \caption{Plot of the evolution optimum gap, $q$-error and grid sizes. The top two plots are logarithmic, while the bottom one is not. The oblique lines are represent the median iterations, the dashed ones are confidence intervals covering all but the top and bottom $5\%$ values. \label{fig:Pqn}}
\end{figure}

\subsection{Example 2: Super-resolution from Gaussian measurements in 2D}
Next, we perform a study in a two-dimensional setting. 
We consider $\Omega=[-1,1]^2$ and measurement functions of the form
\begin{align*}
    a_i(x) = \exp\left( - \frac{\norm{x-x_i}^2}{2\sigma^2} \right),
\end{align*}
where the points $x_i$ live on a Euclidean grid of size $64\times 64$, restricted to the domain $[-0.5,0.5]^2$. 
We then add white Gaussian noise to the measurements, leading to pictures of the type shown in Fig. \ref{fig:super_resolution_2D}. Here, the true underlying measure contains $11$ Dirac masses with random positive amplitudes and random locations on $[-0.4,0.4]^2$. 

\begin{figure}
 \centering\includegraphics[width=.4\textwidth]{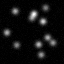}
 \caption{Measurements $y$ associated to a super-resolution experiment. A sparse measure is convolved with a Gaussian kernel and Gaussian white noise is added. \label{fig:super_resolution_2D}}
\end{figure}

\subsubsection{Exchange algorithm}

The evolution of the grids $\Omega_k$ and of the dual certificates $|A^*q_k|$ is shown in Figure \ref{fig:Dual_Certificate}. As can be seen, points are initially added anywhere in the domain, but after a few iterations, they all cluster around the true locations, as expected from the theory. To further stress this phenomenon and illustrate our theorems and lemmata, we display many quantities of interest appearing in our main results in Fig. \ref{fig:SeveralQuantities}.  
the distance from $X_k$ to $\xi$ (where $\xi$ is estimated as $X_{40}$) on Fig. \ref{sub:distX}, the distance from $\Omega_k$ to $\xi$ on Fig. \ref{sub:distO}, the evolution of $J(\widehat \mu_k)-J(\widehat \mu_{40})$ on Fig. \ref{sub:distJ}, $\|A^*q_k\|_\infty-1$ on Fig. \ref{sub:viol}. Finally, the number of maxima of $|A^*q_k|$ is shown on Fig. \ref{sub:max}.
As can be seen, the number of maxima quickly stabilizes, suggesting that we reached a $\tau_0$-regime. Then all the quantities (cost function, distance from $\xi$, violation of the constraints) seem to converge to $0$ linearly. This is not true after iteration 15, and we suspect that this is solely due to numerical inaccuracies when computing the solution of the discretized problems. Notice however that the accuracy of the Dirac locations drops below $10^{-3}$ after 14 iterations, and that this accuracy is more than enough for the particular super-resolution application. Notice that if we wished to reach this accuracy with a fixed grid, we would need a Euclidean discretization containing $10^6$ points, while we here needed only 152 ($|\Omega_{14}|=152$). In addition, the $\ell^1$ resolution is stable since it is accomplished on a grid $X_{14}$ containing only $11$ points.

\begin{figure}[h]
  \begin{center}
    \subfloat[$|A^*q_1|$]{
      \includegraphics[width=0.32\textwidth]{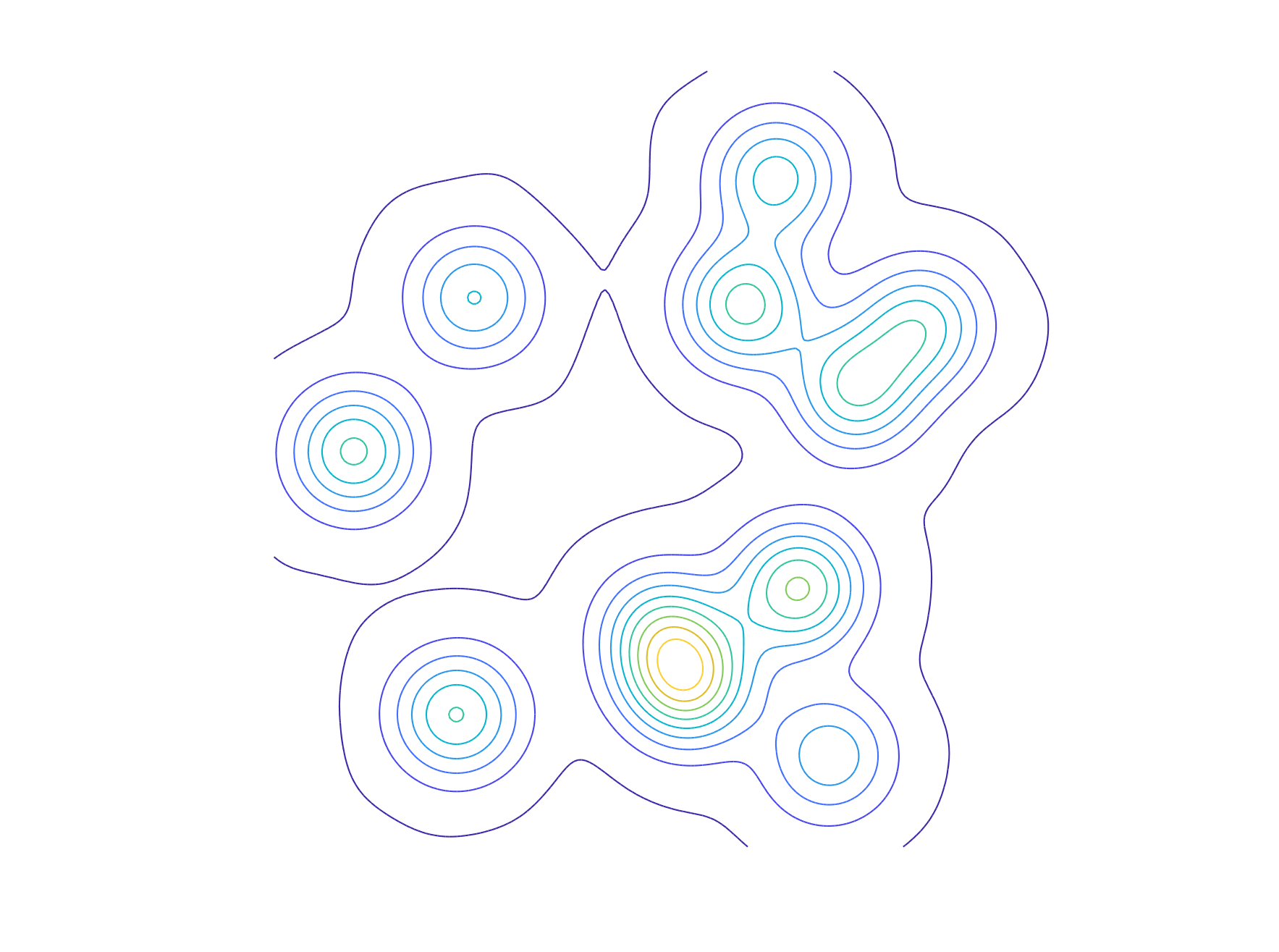}}
    \subfloat[$|A^*q_2|$]{
      \includegraphics[width=0.32\textwidth]{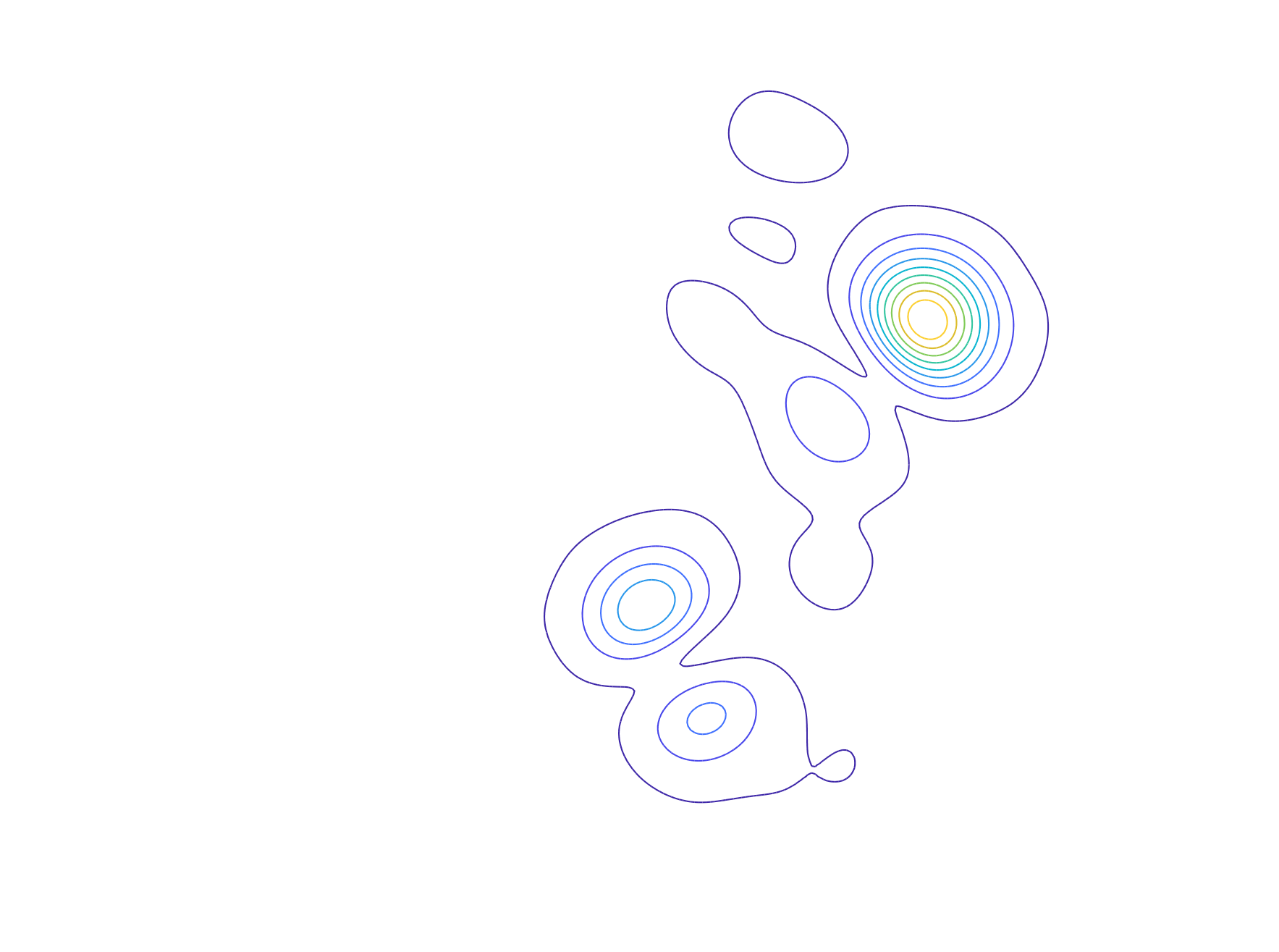}}
    \subfloat[$|A^*q_4|$]{
      \includegraphics[width=0.32\textwidth]{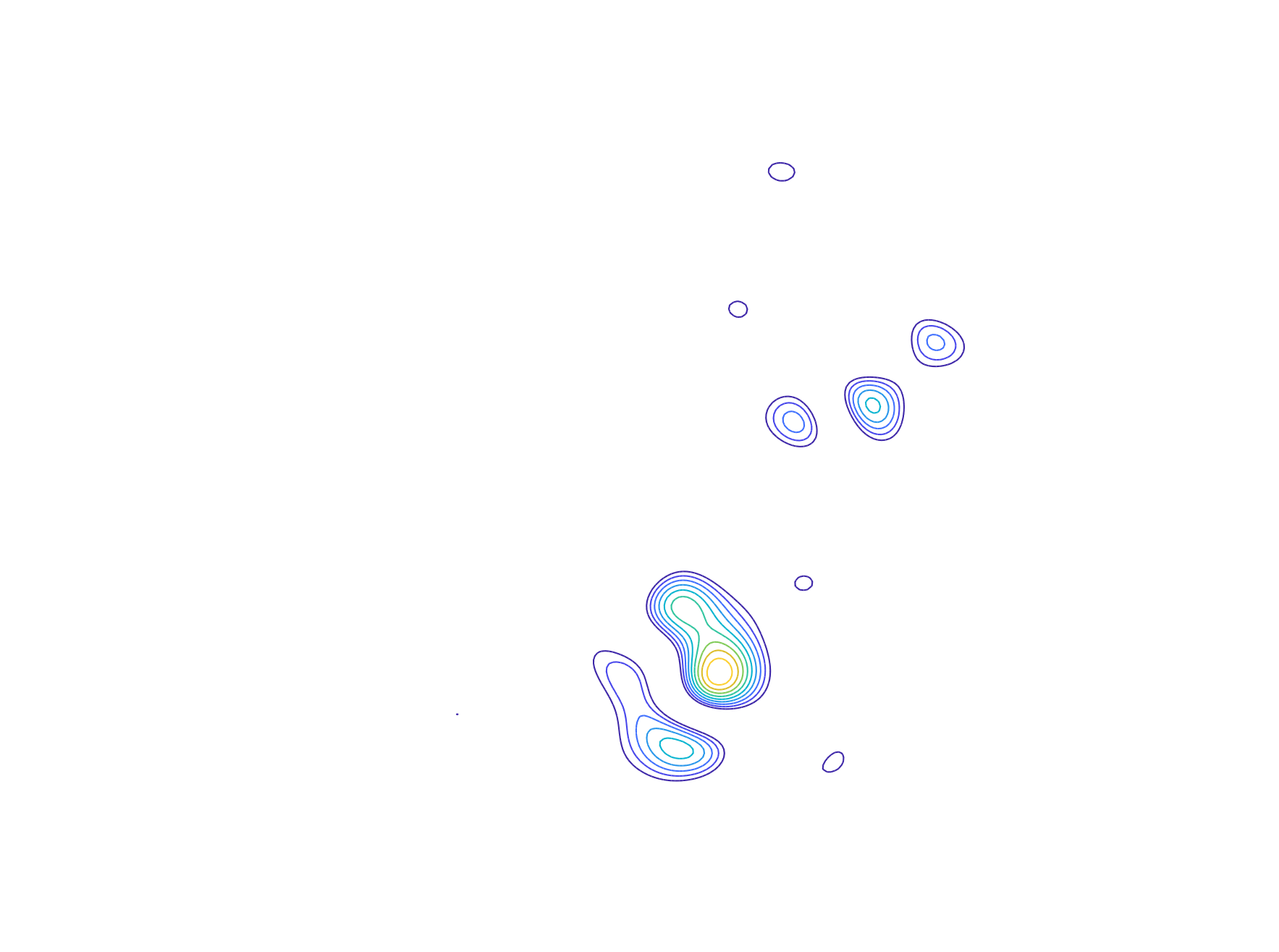}}

    \subfloat[Grid 1]{
      \includegraphics[width=0.32\textwidth]{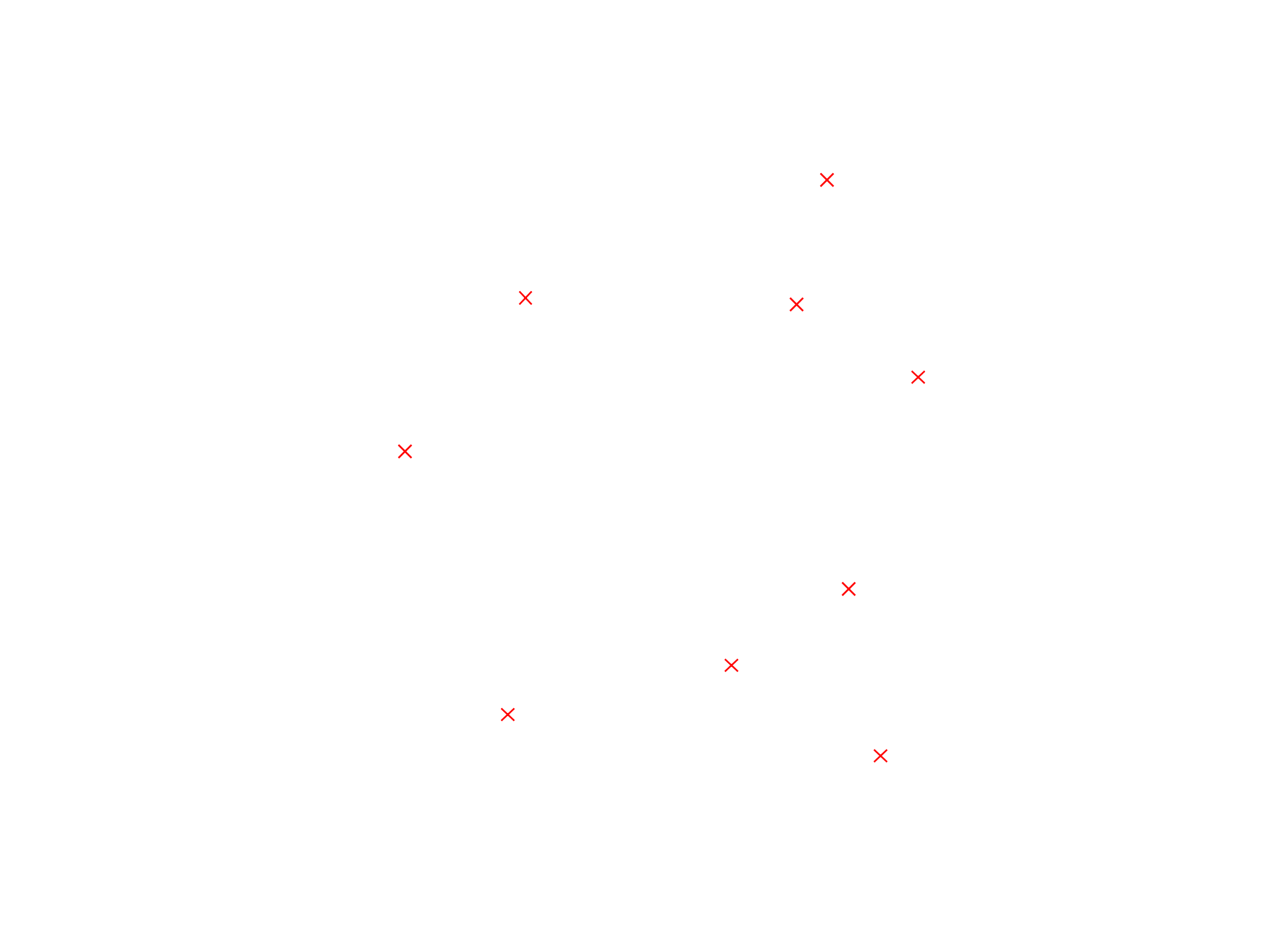}}
    \subfloat[Grid 2]{
      \includegraphics[width=0.32\textwidth]{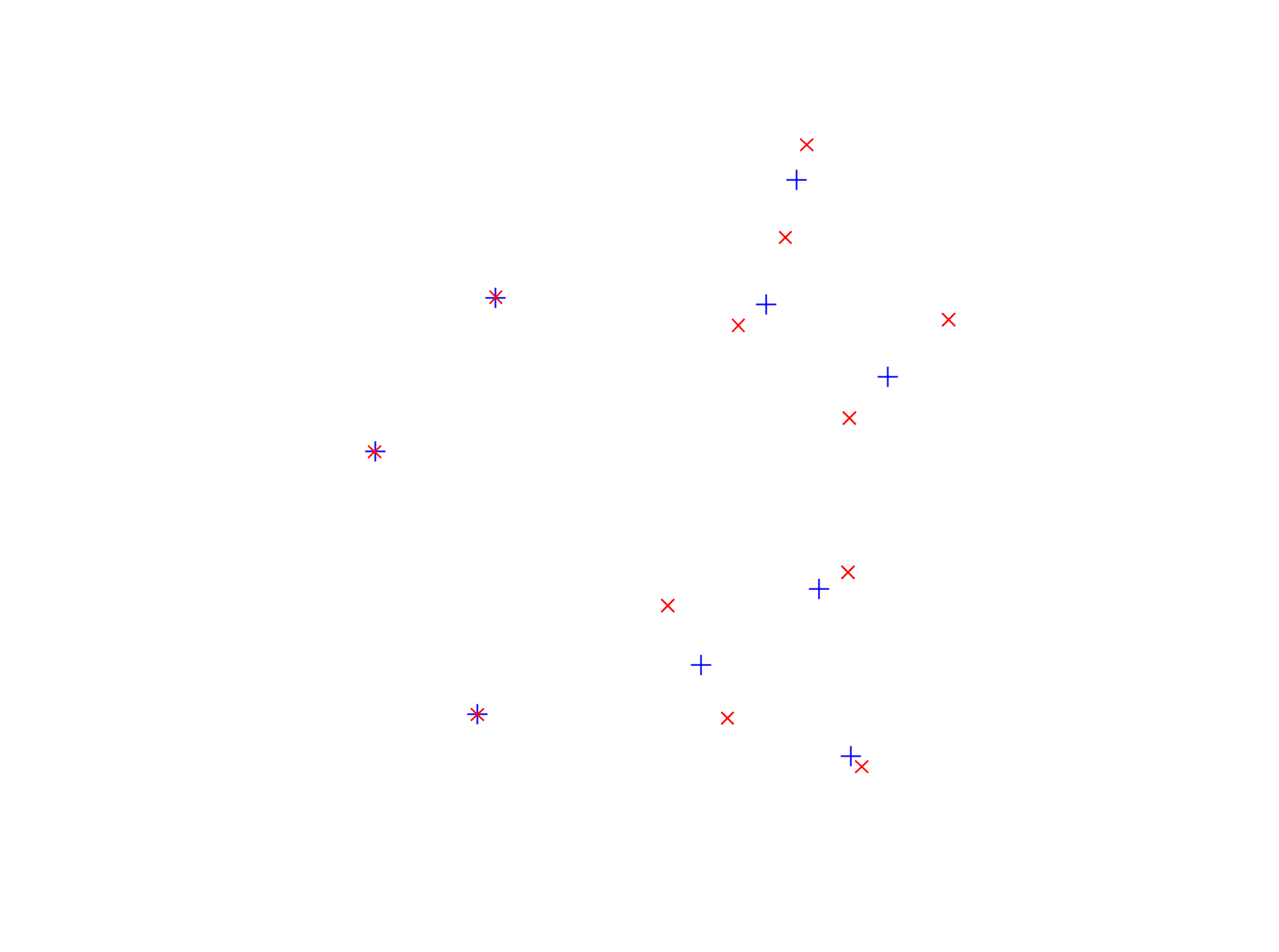}}
    \subfloat[Grid 4]{
       \includegraphics[width=0.32\textwidth]{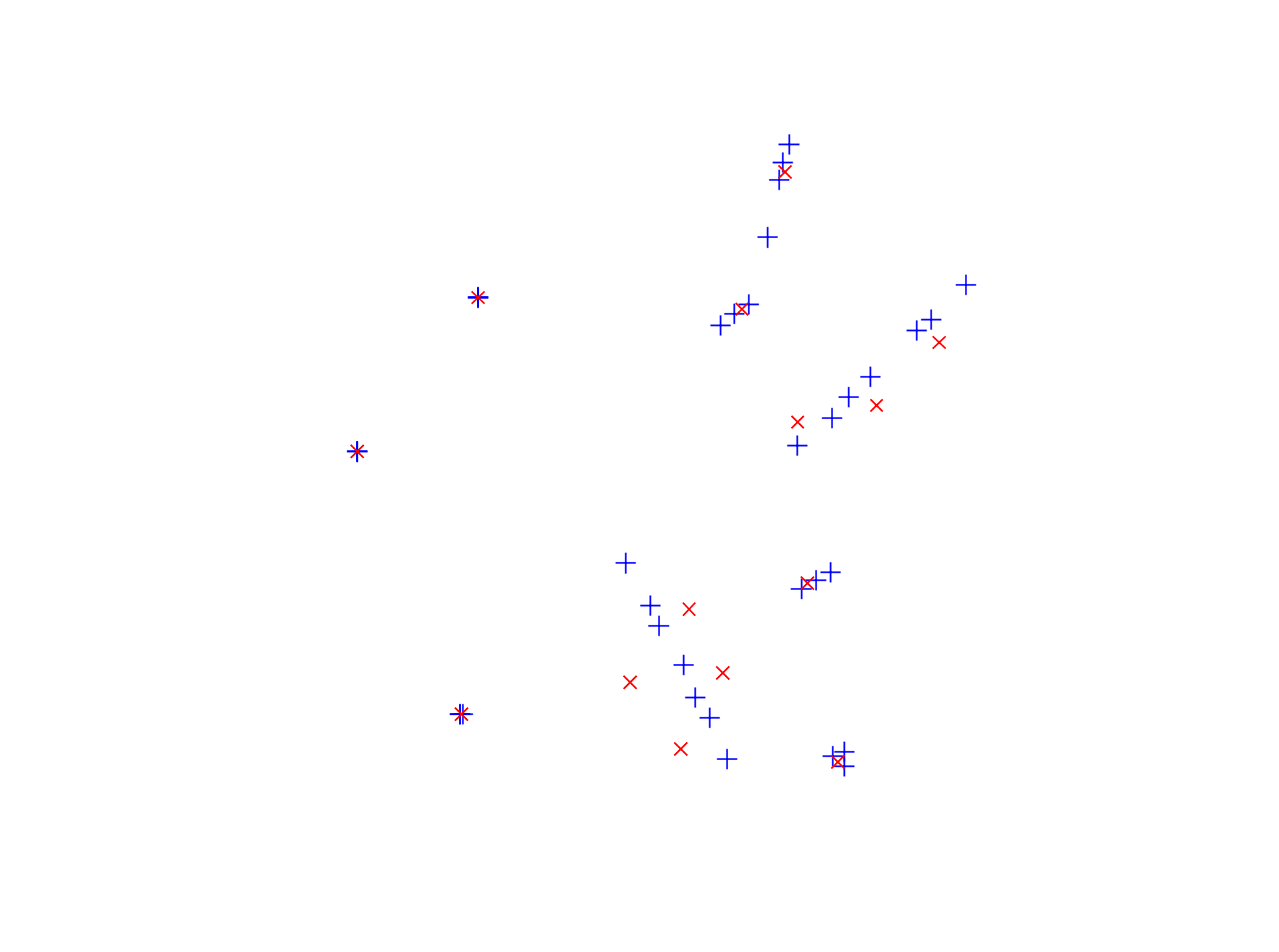}}
      
      \subfloat[$|A^*q_8|$]{
      \includegraphics[width=0.32\textwidth]{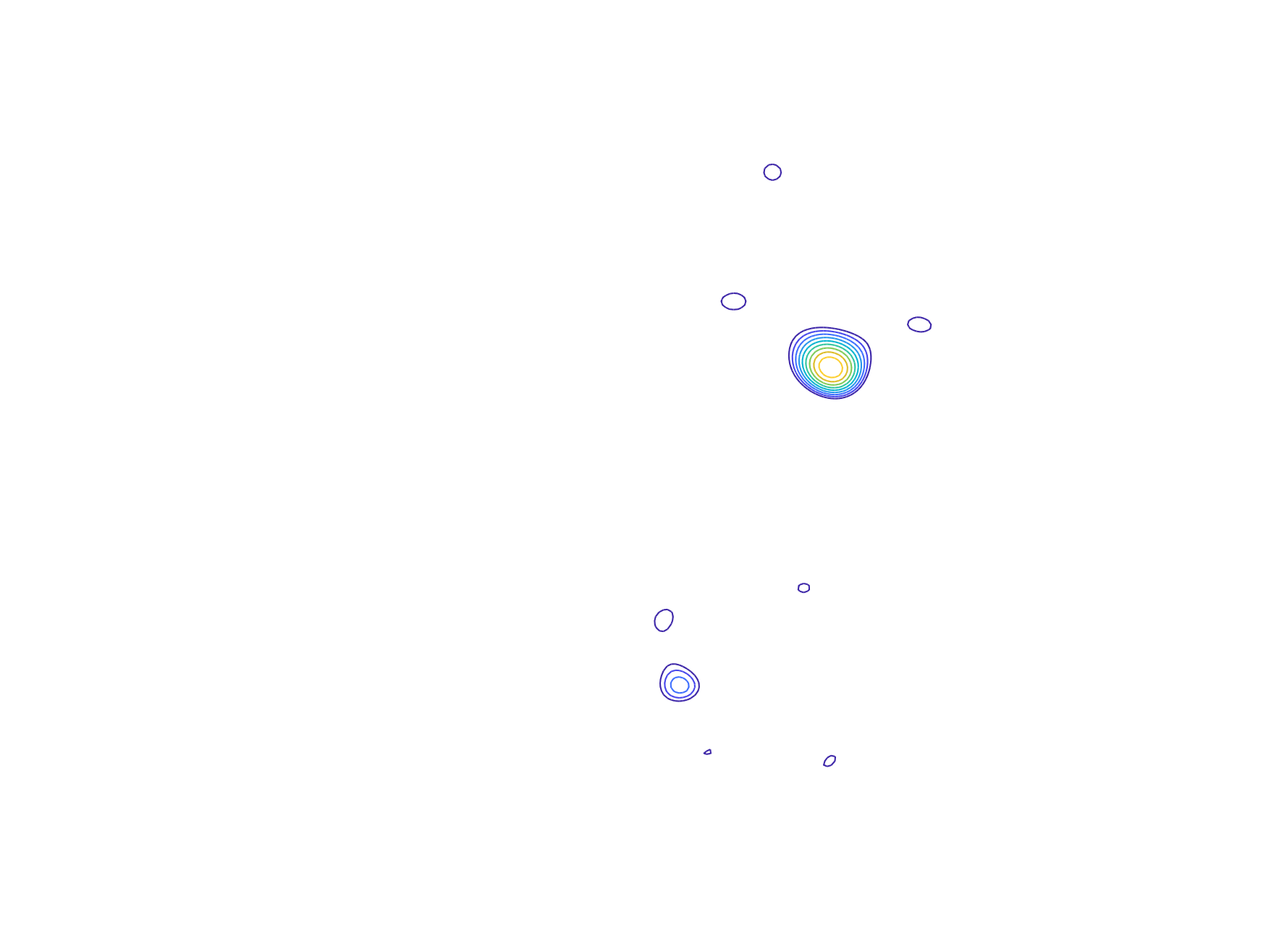}}
    \subfloat[$|A^*q_{10}|$]{
      \includegraphics[width=0.32\textwidth]{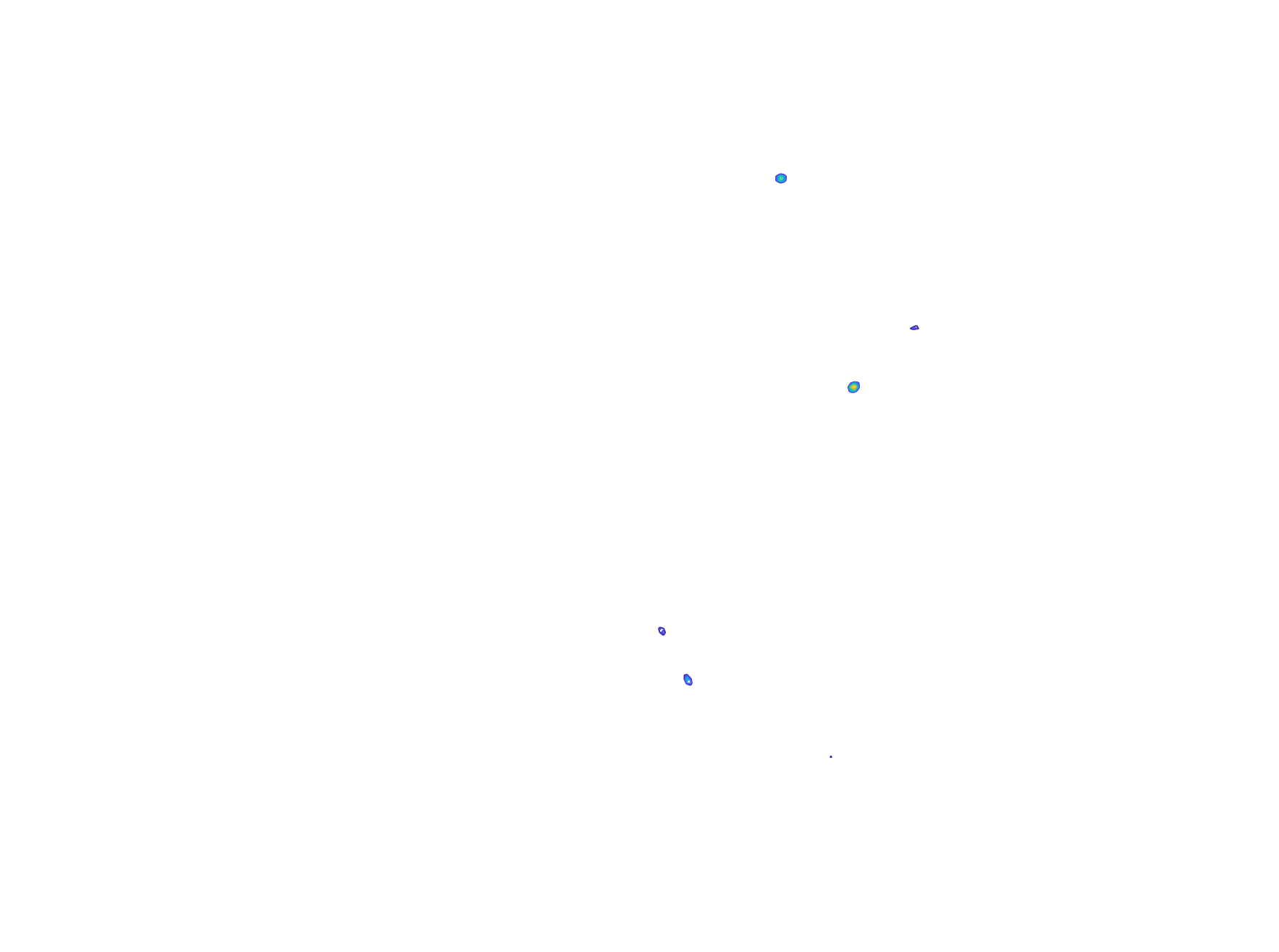}}
    \subfloat[$|A^*q_{12}|$]{
      \includegraphics[width=0.32\textwidth]{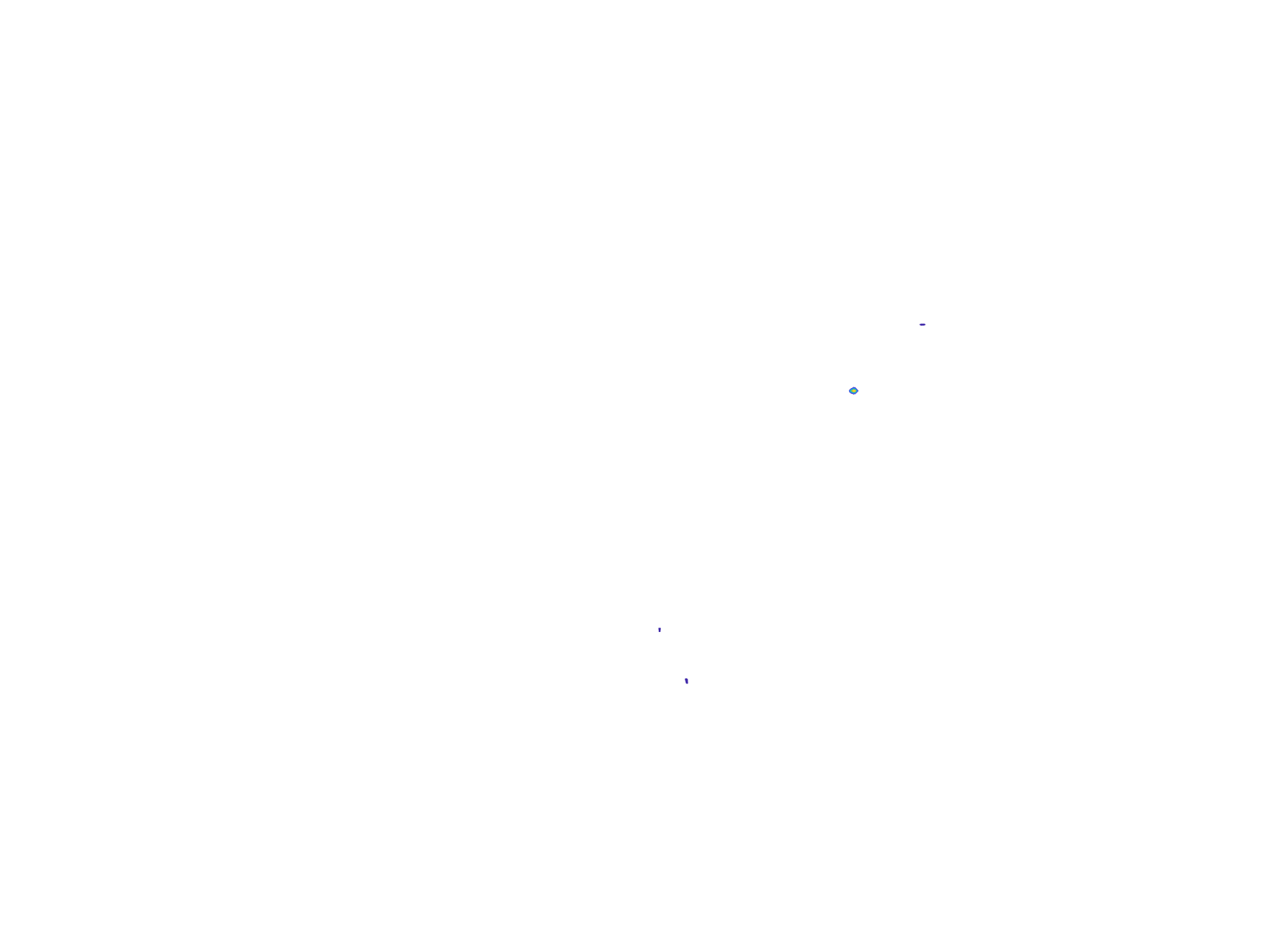}}

    \subfloat[Grid 8]{
      \includegraphics[width=0.32\textwidth]{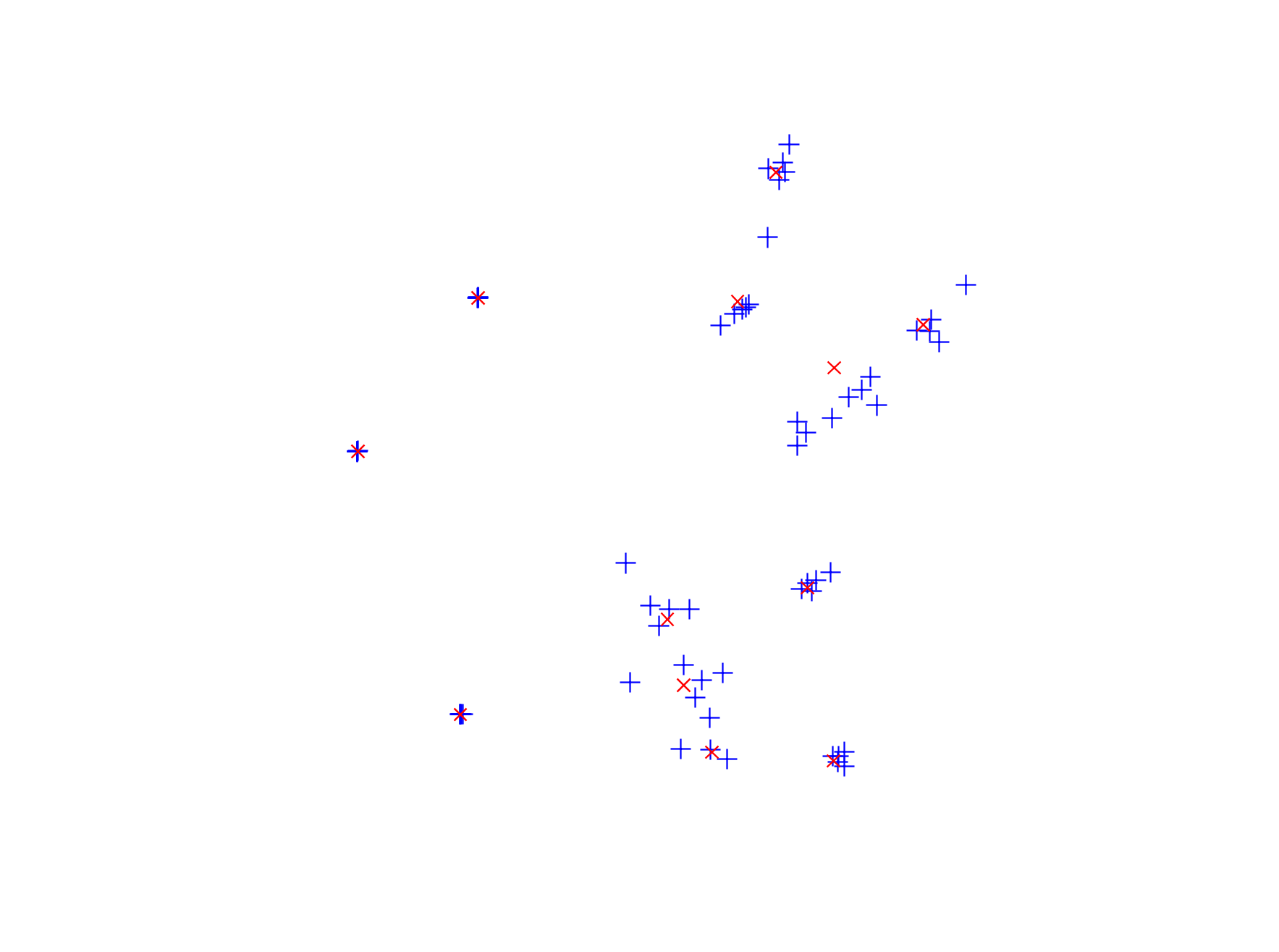}}
    \subfloat[Grid 10]{
      \includegraphics[width=0.32\textwidth]{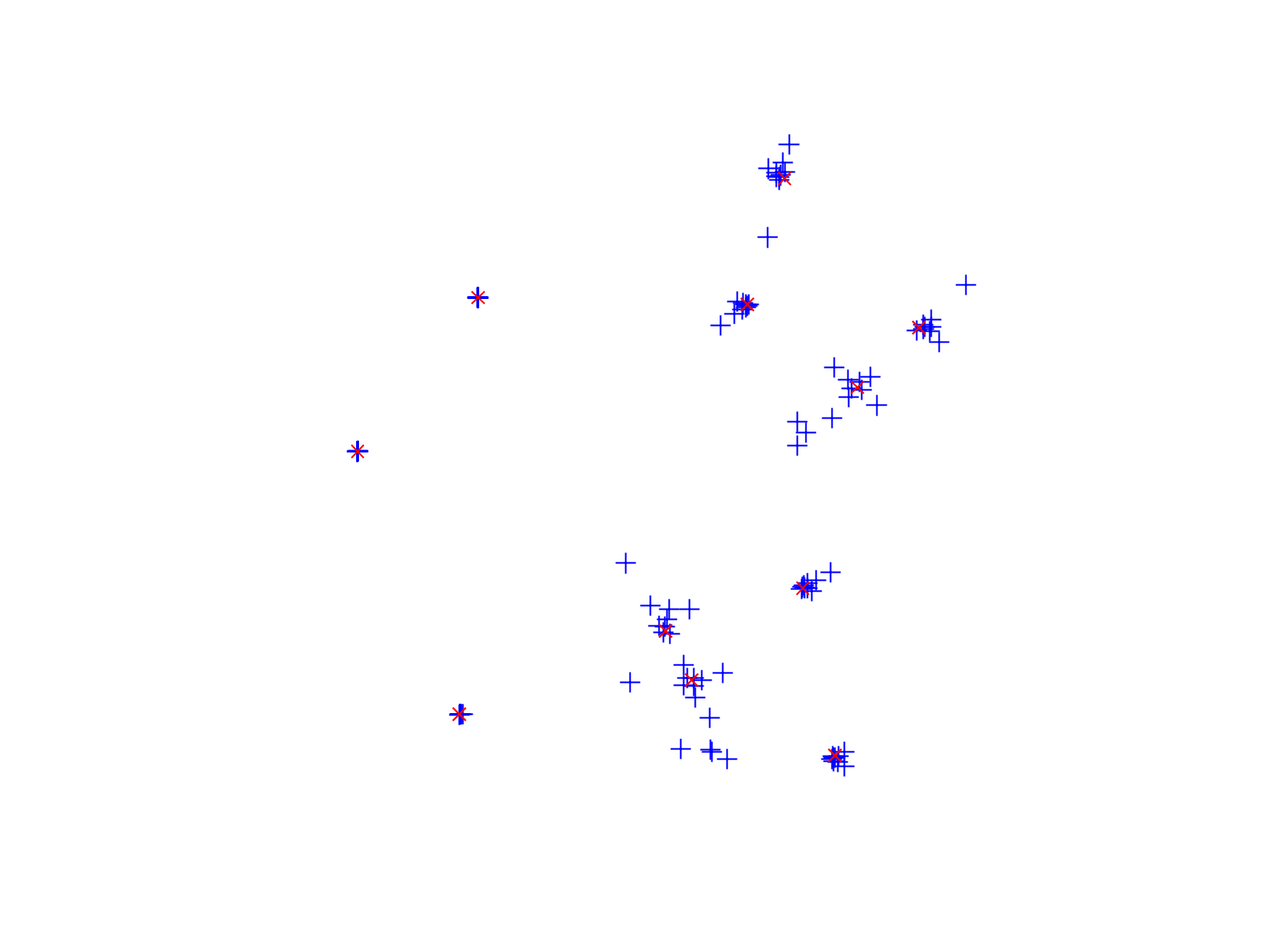}
      \label{sub:ref6}}
    \subfloat[Grid 12]{
       \includegraphics[width=0.32\textwidth]{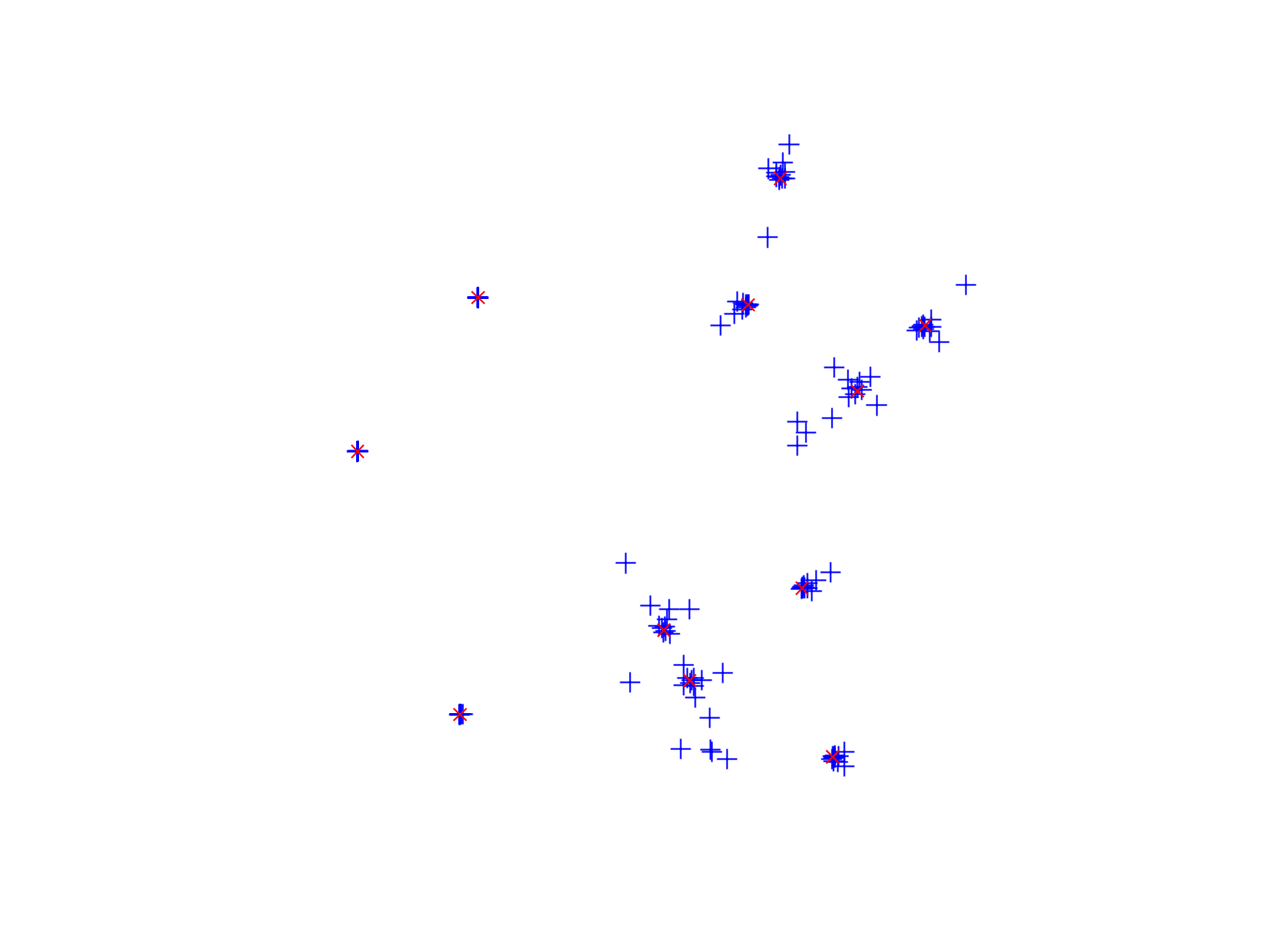}}
      
      \caption{Evolution of the dual certificate and of the grid through the 12 first iterations. This is a contour plot with the levels from 1 to the maximum of $|A^*q_i|$ indicated.}
    \label{fig:Dual_Certificate}
  \end{center}
\end{figure}

\begin{figure}[h]
  \begin{center}
    \subfloat[$J(\widehat \mu_k)-J(\mu_{40})$]{\includegraphics[width=0.3\textwidth]{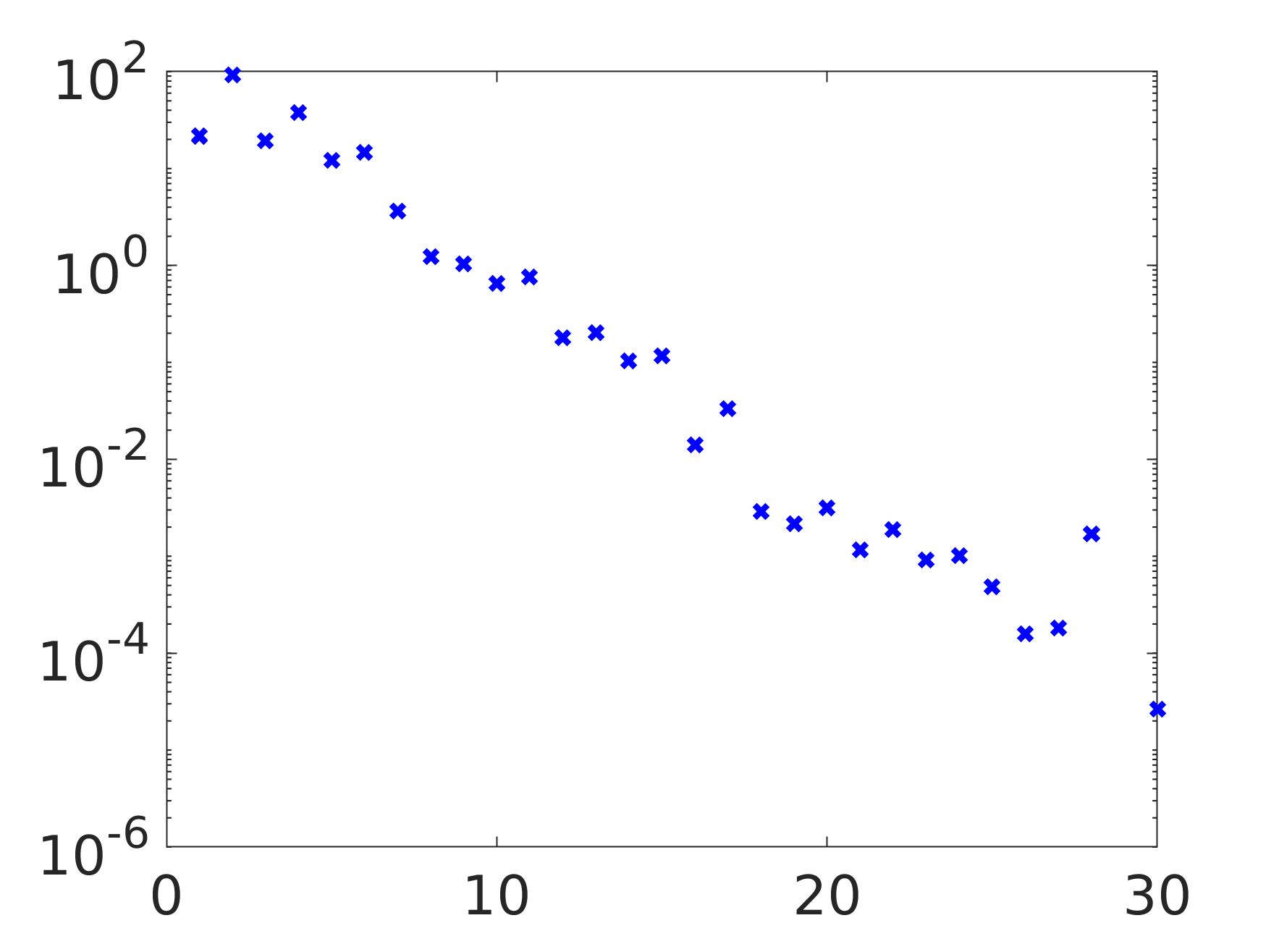} \label{sub:distJ}}
    \subfloat[$\dist(\Omega_k,\xi)$]{\includegraphics[width=0.3\textwidth]{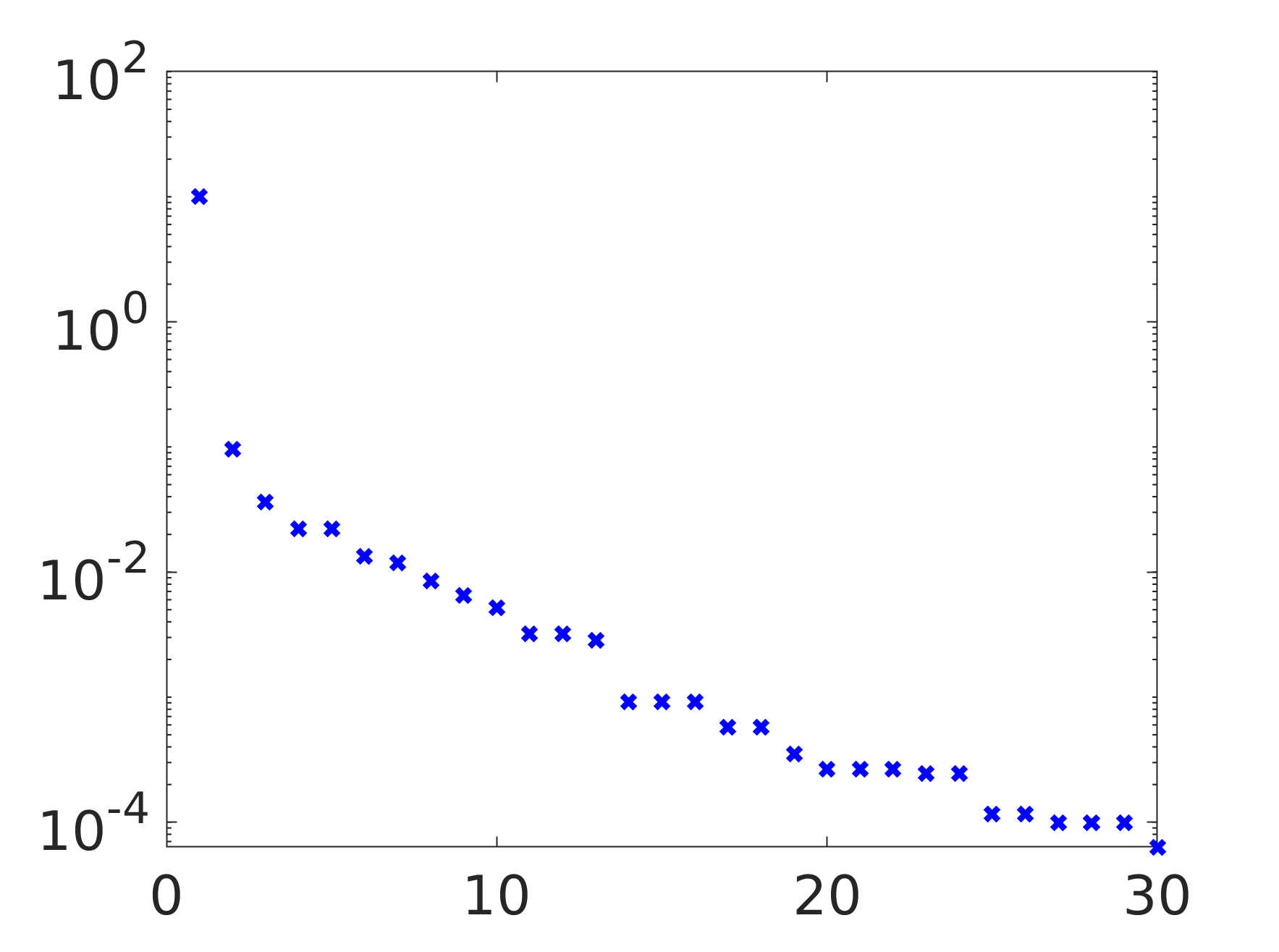} \label{sub:distO}}
    \subfloat[$\dist(X_k,\xi)$]{\includegraphics[width=0.3\textwidth]{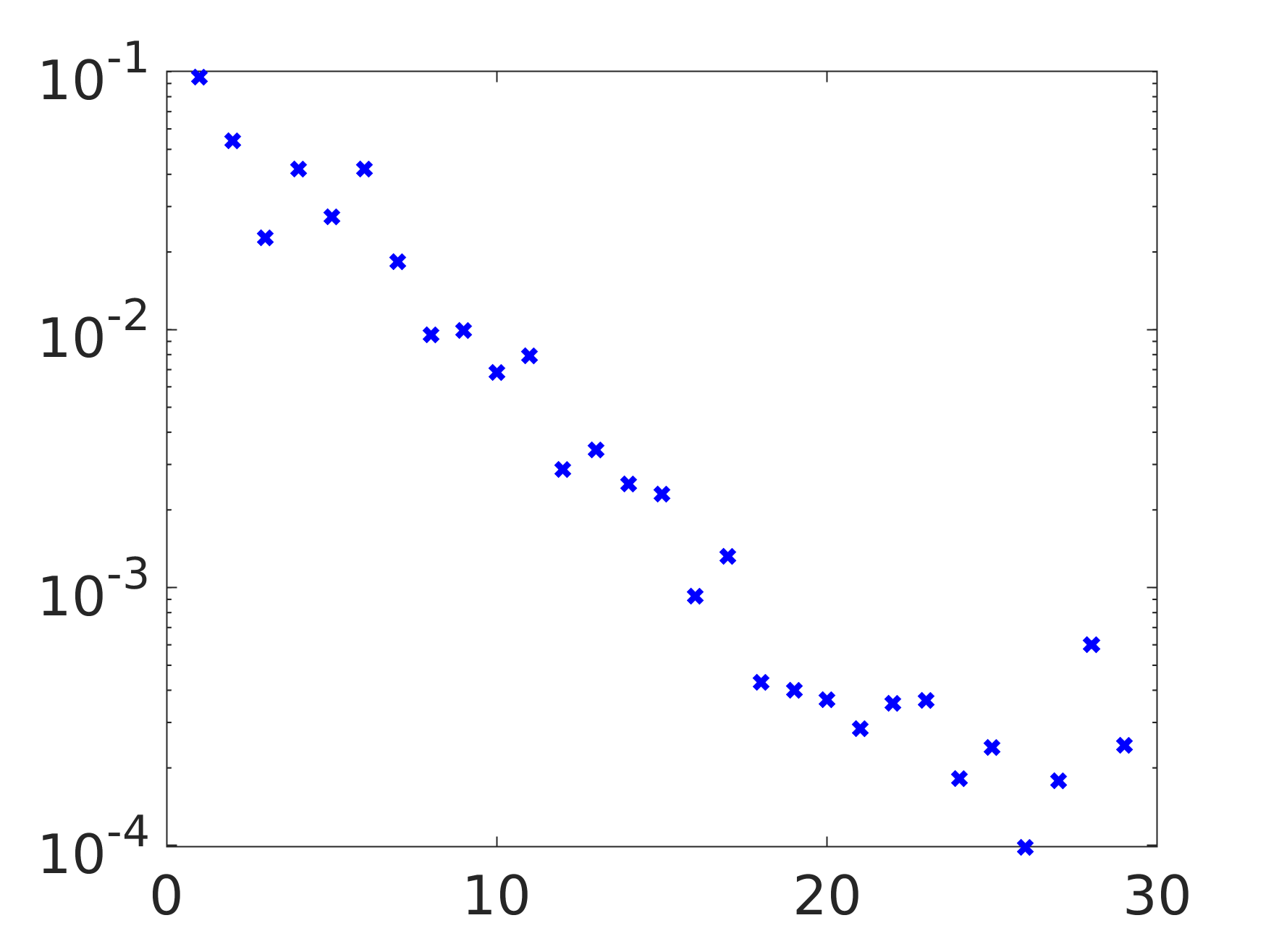} \label{sub:distX}}
    
    \subfloat[$\|q_k-q_{40}\|_2$]{\includegraphics[width=0.3\textwidth]{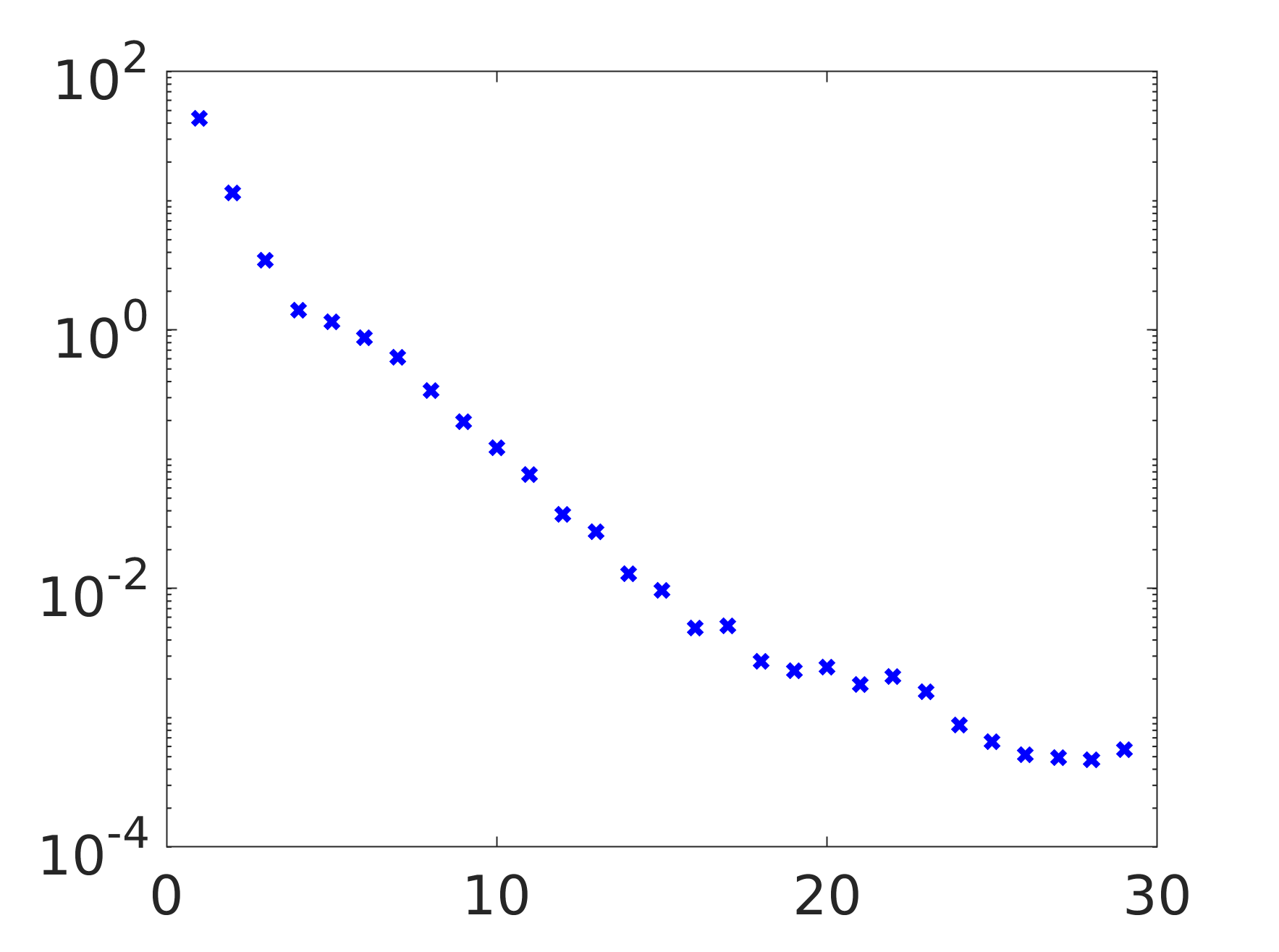} \label{sub:distq}}
    \subfloat[$\|A^*q_k\|_\infty-1$]{\includegraphics[width=0.3\textwidth]{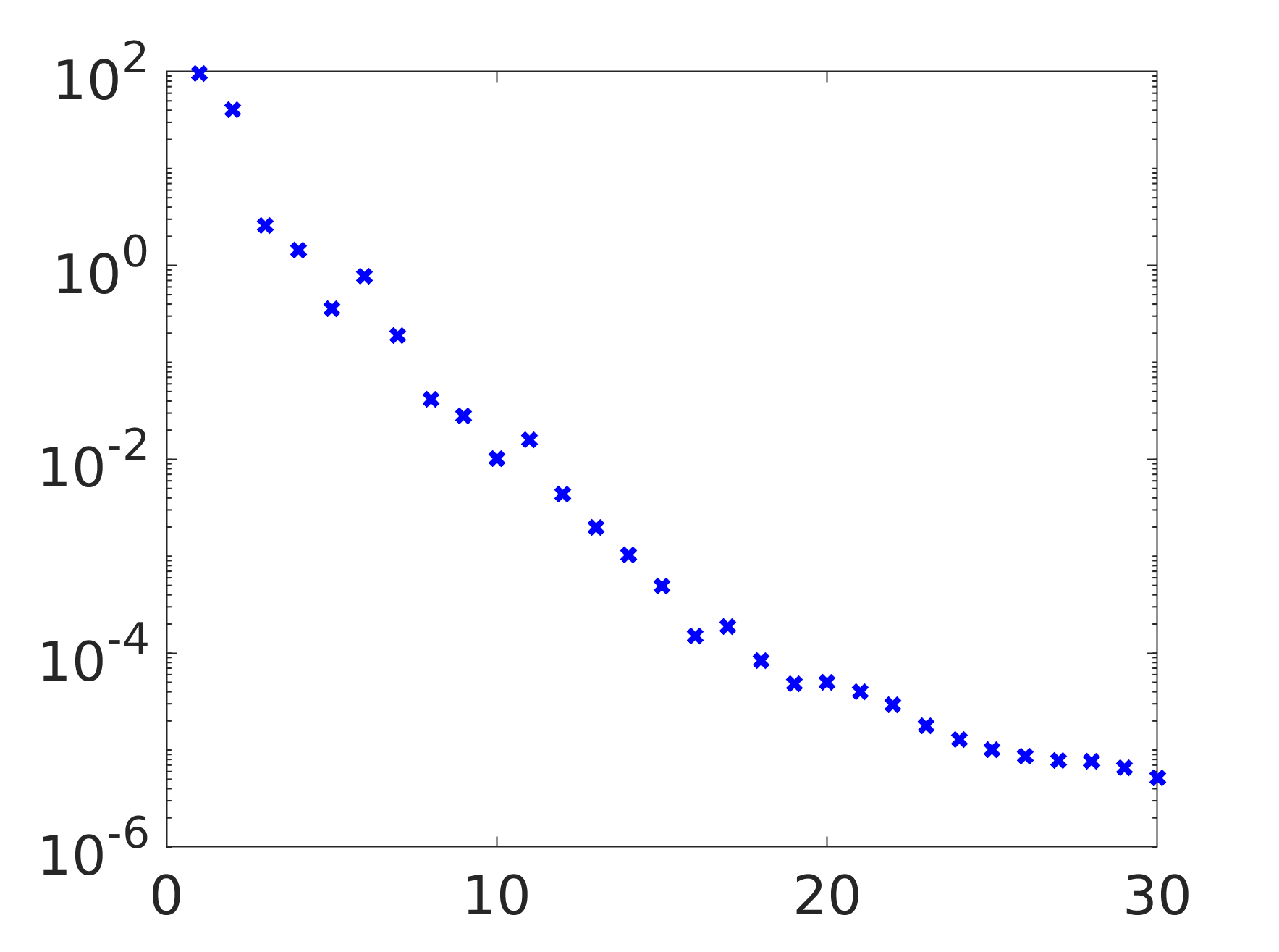} \label{sub:viol}}
    \subfloat[$|X_k|$]{\includegraphics[width=0.3\textwidth]{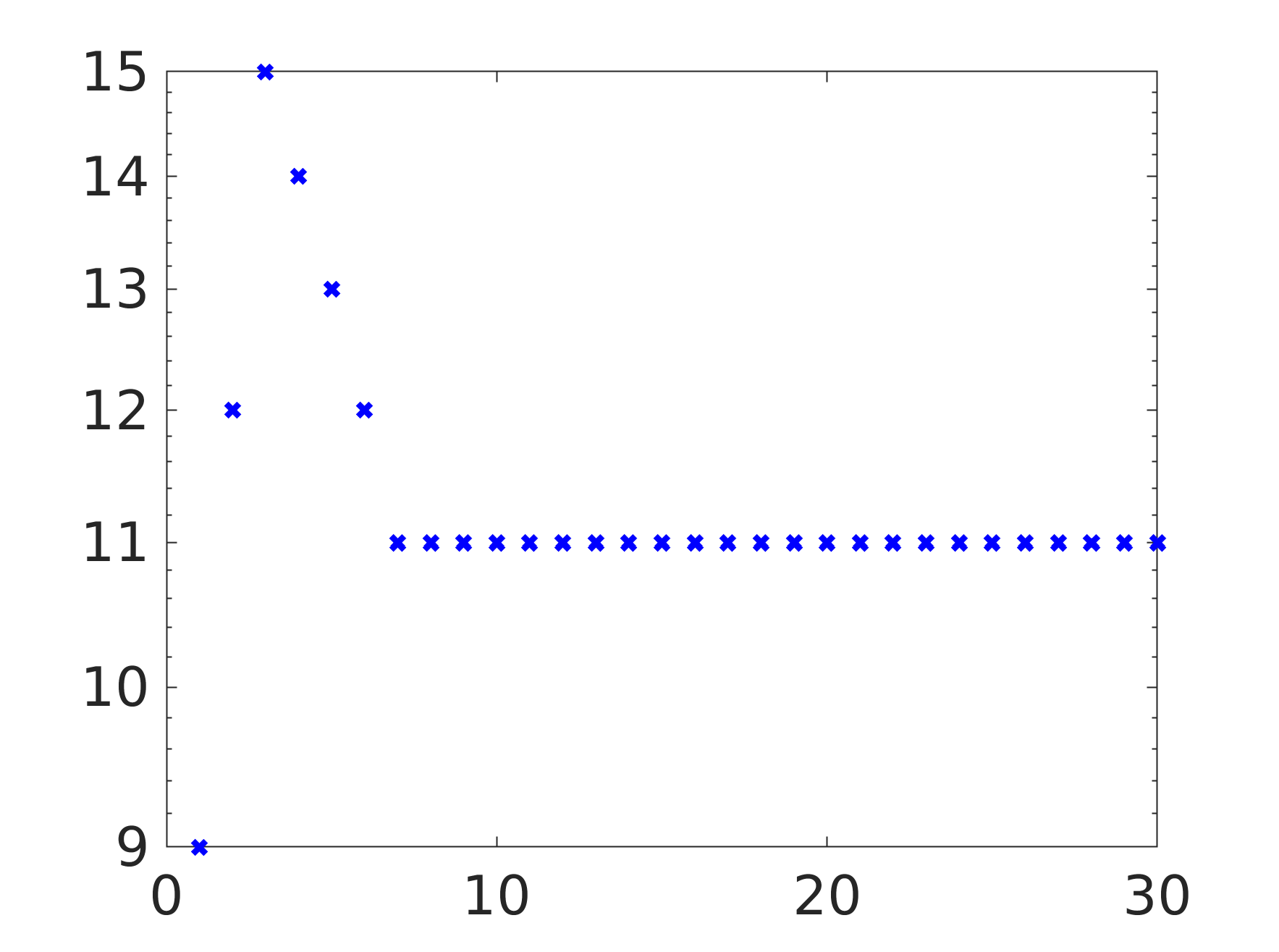} \label{sub:max}}
      \caption{Plot of several quantities of interest along the exchange algorithm's iterates.}
    \label{fig:SeveralQuantities}
  \end{center}
\end{figure}

\subsubsection{Continuous method}

In this experiment, we evaluate the behavior of the gradient descent \eqref{eq:continuous_gradient} depending on the initialization $(\alpha^{(0)},X^{(0)})$ and on the number of iterations. 
We use the same setting as in the previous section. 
The left graph of Fig. \ref{fig:continuous_method} illustrates that the gradient descent typically converges linearly when initialized close enough to the true minimizer $(\alpha^{\star},\xi)$. 
This was predicted by Theorem \ref{prop:G}. 
In this case (and actually all the others related to this experiment), it converges to machine precision in less than 1000 iterations. 
This is remarkable since the gradient descent is a simple algorithm that can be easily improved by using e.g. Nesterov acceleration (we proved that the function is locally convex) or other optimization schemes such as L-BFGS.

In order to evaluate the size of the basin of attraction around the global minimizer, we start from random points of the form $(\alpha^{(0)},X^{(0)}) = (\alpha^{\star},\xi) + (\Delta_\alpha, \Delta_X)$, where $\Delta_\alpha$ and $\Delta_X$ are random perturbations with an amplitude set as $\|(\Delta_\alpha, \Delta_X)\|_2 = \gamma \|(\alpha^{\star},\xi)\|_2$, with $\gamma$ in $[0,1]$. We then run $50$ gradient descents with different realizations of $(\alpha^{(0)},X^{(0)})$ and record the success rate (i.e. the number of times the gradient descent converges to $(\alpha^{\star},\xi)$ with an accuracy of at least $10^{-6}$). We plot this success rate with respect to $\gamma$ in Fig. \ref{sub:success_rate}. As can be seen, the success rate is always $1$ when the relative error $\gamma$ is less than $5\%$, showing that for this particular problem, a rather rough initialization suffices for the gradient descent to converge to the global minimizer.

\begin{figure}[h]
  \begin{center}
    \subfloat[$G(\alpha^{(t)},X^{(t)})-G(\alpha^{\star},\xi)$]{\includegraphics[height=0.3\textwidth]{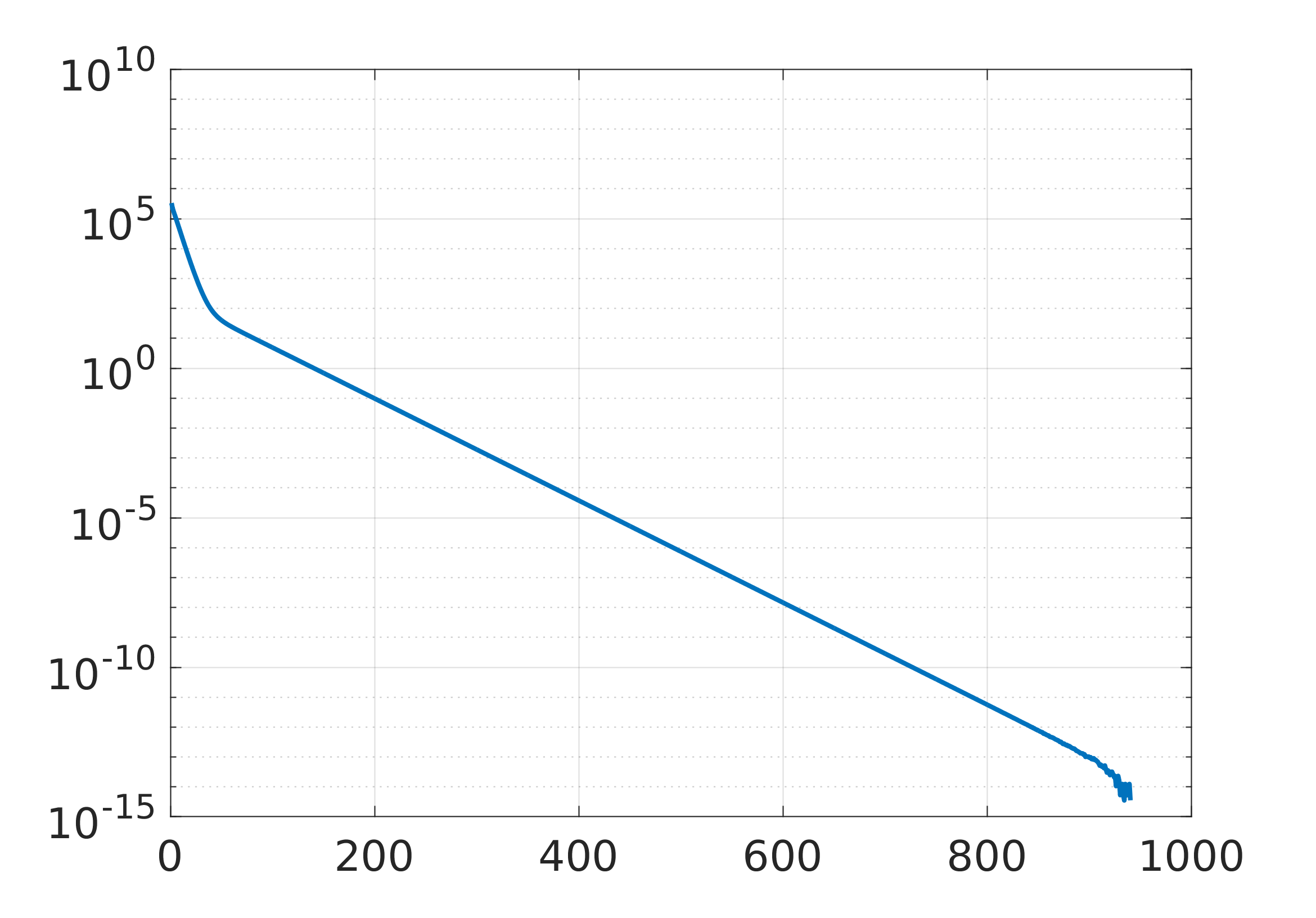}}
    \subfloat[Success rate VS starting point \label{sub:success_rate}]{\includegraphics[height=0.3\textwidth]{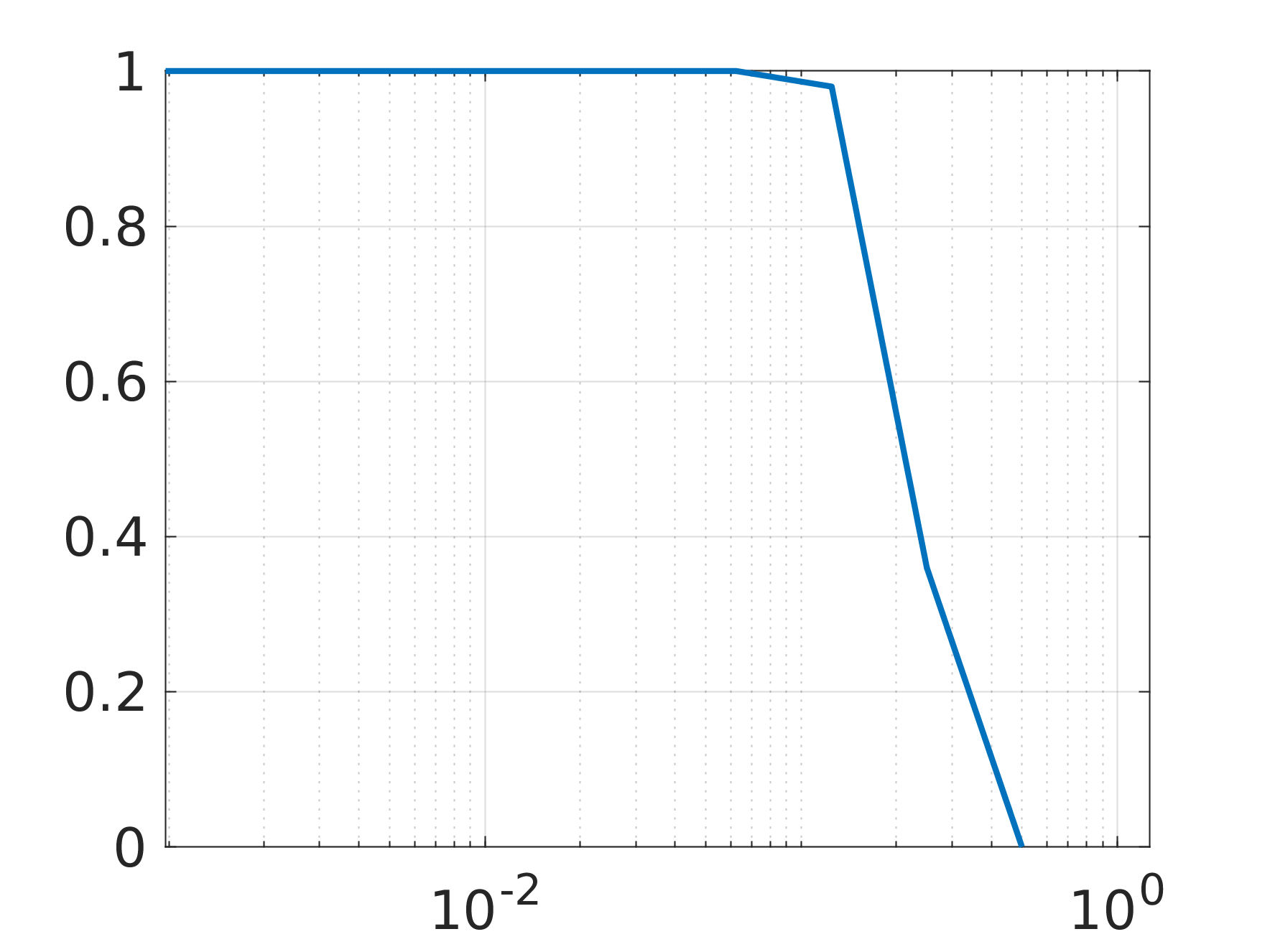}}
    \caption{Left: Typical convergence curve in logarithmic scale when the initial guess $(\alpha^{(0)},X^{(0)})$ is good enough. Right: Success rate of the continuous descent method over 50 runs of the algorithm, depending on the relative amplitude of the perturbation. \label{fig:continuous_method}}
  \end{center}
\end{figure}

\subsubsection{Alternating method}

The alternating method suggested in Algorithm \ref{alg:hybrid} turns out to converge in a single iteration when applied to the setting described above.
We therefore apply it to a more challenging scenario with 30 Dirac masses instead of 11 and more noise. 
The measurements $y$ are shown in Fig. \ref{fig:second_measurements}. 
We compare three implementations: a pure exchange method, an alternating method as in Algorithm  \ref{alg:hybrid} without line 14 and an alternating method as in in Algorithm \ref{alg:hybrid} with line 14.
The conclusions are as follows:
\begin{itemize}
 \item All methods rapidly conclude that the underlying measure contains 30 Dirac masses.  (The pure exchange algorithm after 10 iterations, the alternating method with line 14 already after the first). 
 \item The pure exchange algorithm quickly gets to a point close to the optimum. The positions then slowly converge to the tue locations. It does however eventually find the basin of attraction of $G$ (in this example, it needed 10 iterations).
 \item Line 14 in the alternating method improves the convergence significantly. In fact, omitting it, we need 10 iterations to find the basin of attraction, whereas the version with the line finds it directly. Investigating this effect more closely is an interesting line of future research.
\end{itemize}

\begin{figure}[h]
  \begin{center}
    \subfloat[Measurements $y$ (dense)]{\includegraphics[width=0.45\textwidth]{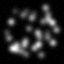}}\ \ 
    \subfloat[Ground truth and recovered solution]{\includegraphics[width=0.45\textwidth]{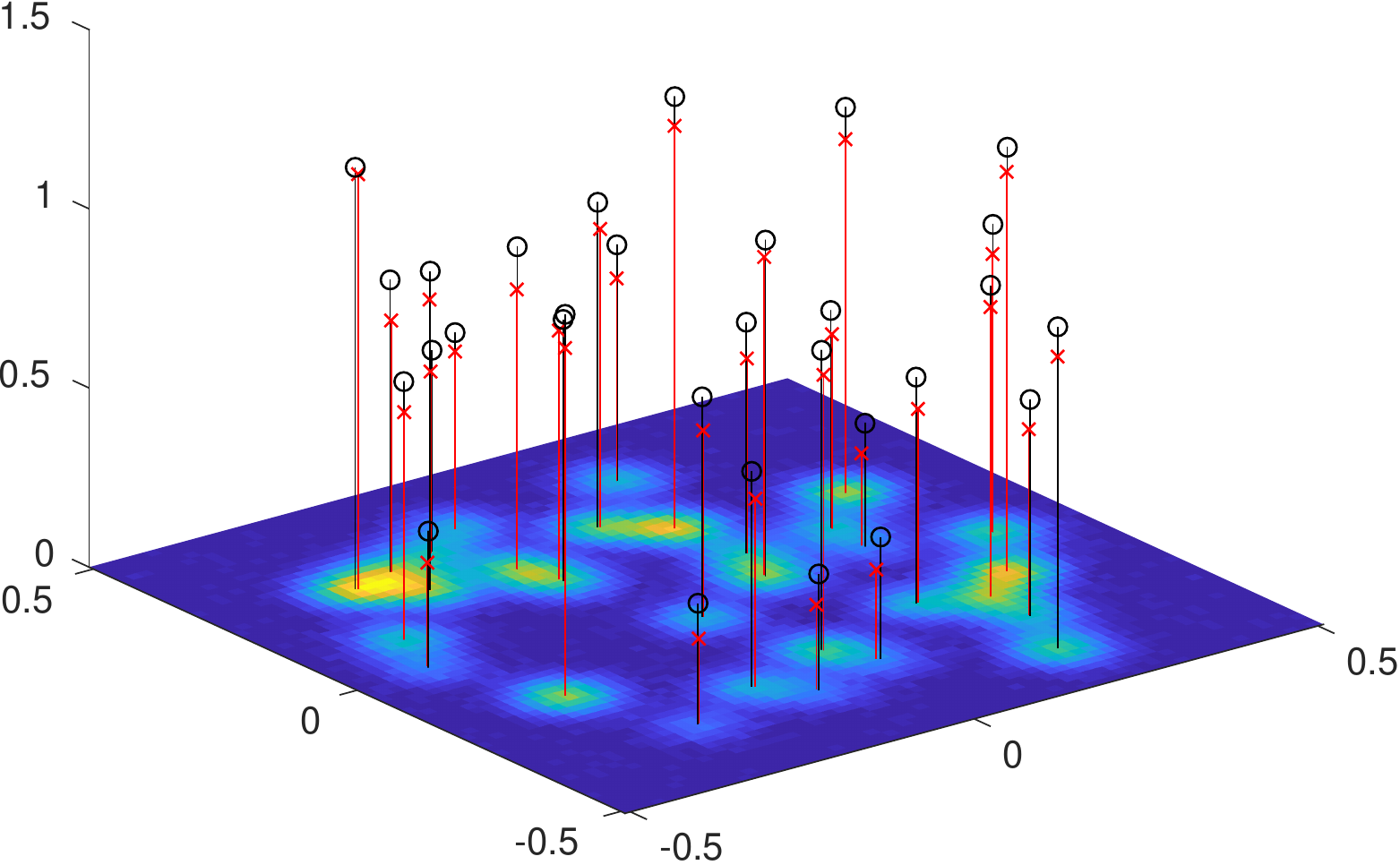}}
    \caption{Left: measurements associated to a denser measure with more noise. Right: 3D illustration of the recovery results. The blue vertical bars with circles indicate the locations and amplitude of the ground truth. The red bars with crosses indicated the recovered measures. Apart from a slight bias in amplitude due to the $\ell^1$-norm, the ground truth is near perfectly recovered. \label{fig:second_measurements}}
  \end{center}
\end{figure}

\bibliographystyle{plain}
\bibliography{iterative_bio}

\end{document}